\documentclass[oneside,a4paper]{amsart}

\usepackage{amsmath,amsthm,verbatim,amssymb,amsfonts,amscd, graphicx}
\usepackage{graphics}

\usepackage[utf8]{inputenc}
\usepackage[T1]{fontenc}

\usepackage{stmaryrd}
\usepackage{xspace}
\usepackage[all]{xy}
\usepackage[margin=2.5cm]{geometry}

\usepackage{blindtext} 
\usepackage[linktocpage]{hyperref}

\theoremstyle{plain}
\newtheorem{thm}{Theorem}[section]
\newtheorem{cor}[thm]{Corollary}
\newtheorem{lem}[thm]{Lemma}
\newtheorem{prop}[thm]{Proposition}

\theoremstyle{definition}
\newtheorem*{defi}{Definition}
\newtheorem*{expl}{Example}

\newtheorem*{rmq}{Remark}

\newcommand{\ho}{\mathrm{H}}
\newcommand{\Hom}{\mathrm{Hom}}

\newcommand{\Oc}{\mathcal{O}}

\newcommand{\m}{\mathfrak{m}}

\newcommand{\X}{\mathcal{X}}

\newcommand{\Li}{\overline{\mathcal{L}}}
\newcommand{\li}{\mathcal{L}}
\newcommand{\A}{\mathbb{A}}
\newcommand{\Q}{\mathbb{Q}}
\newcommand{\Z}{\mathbb{Z}}
\newcommand{\R}{\mathbb{R}}
\newcommand{\bbC}{\mathbb{C}}
\newcommand{\pr}{\mathbb{P}}
\newcommand{\F}{\mathbb{F}}

\newcommand{\Ker}{\mathrm{Ker}}

\newcommand{\Spec}{\mathrm{Spec}}

\usepackage[shortlabels]{enumitem}
\setlist[enumerate]{nosep}

\begin{document}

\title{
	On the Bertini regularity theorem for arithmetic varieties
	}
\author{Xiaozong WANG}
\address{Morningside Center of Mathematics, Chinese Academy of Sciences, Beijing, 100190, China}
\email{xiaozong.wang@amss.ac.cn}

\begin{abstract}
	Let $\mathcal{X}$ be a regular projective arithmetic variety equipped with an ample Hermitian line bundle $\overline{\mathcal{L}}$. We prove that the proportion of global sections $\sigma$ with $\left\lVert \sigma \right\rVert_{\infty}<1$ of $\overline{\mathcal{L}}^{\otimes d}$ whose divisor does not have a singular point on the fiber $\mathcal{X}_p$ over any prime $p\leq e^{\varepsilon d}$ tends to $\zeta_{\mathcal{X}}(1+\dim \mathcal{X})^{-1}$ as $d\rightarrow \infty$. 
\end{abstract}

\maketitle

\tableofcontents

\section{Introduction}
The main result of this article is Theorem \ref{mainvar}.

The classical Bertini theorem states that if $X$ is a smooth quasi-projective variety of dimension $m$ over an infinite field $k$ embedded into some projective space $\pr^n_k$, the intersection of $X$ with a general hyperplane of $\pr^n_k$ is smooth of dimension $m-1$. Here general means that the set of hyperplanes satisfying this property is the set of $k$-points of an open subscheme $U$ of the dual projective space $(\pr^n_k)^{\vee}$ of $\pr^n_k$. 
This open subscheme $U$ of the dual projective space exists regardless of the conditions on the base field, but it's the infiniteness of the field $k$ that guarantees the existence of infinitely many $k$-points in $U$. We have similar theorems on reducedness, irreducibility, connectedness, etc. A good reference for these results is \cite{Jo83}.

When $k$ is a finite field, this theorem still gives us an open subscheme of $(\pr^n_k)^{\vee}$ parametrizing hyperplanes whose intersection with $X$ is smooth, but may fail to give such a hyperplane as the open subscheme may have no $k$-point. In \cite{Po04}, Poonen proved that if we consider the proportion of hypersurfaces of degree $d$ whose intersection with $X$ is smooth of dimension $m-1$ among all the degree $d$ hypersurfaces, this proportion tends to $\zeta_X(1+m)^{-1}=\prod_{x\in |X|}(1-\#\kappa(x)^{-(1+m)})$ when $d$ tends to infinity. Here $|X|$ denotes the underlying topological space of $X$ consists of its closed points. In \cite{CP16}, Charles and Poonen also considered hypersurfaces of degree $d$ of $\pr^n_k$ whose intersection with an irreducible subscheme $X$ of dimension at least $2$ is still irreducible, and proved that the proportion of such hypersurfaces tends to $1$ when $d$ tends to infinity. 

It is also of interest to have a good analogue of Bertini smoothness theorem for quasi-projective schemes over $\Spec\ \Z$. But as Poonen explained in \cite[Section 5.7]{Po04}, smoothness condition is too strong in the arithmetic situation. We need to consider regularity instead. 
In the same article, Poonen established a density for subsets of $\bigcup_{d\geq 0}\ho^0(\pr^n_{\Z}, \mathcal{O}(d))$ and showed that for a regular subscheme $\X$ of $\pr^n_{\Z}$, assuming the \emph{abc} conjecture and an auxiliary conjecture, the density of sections $f\in\bigcup_{d\geq 0}\ho^0(\pr^n_{\Z}, \mathcal{O}(d))$ such that $\mathrm{div}\ f\cap \X$ is regular is $\zeta_{\X}(\dim \X+1)^{-1}$. Poonen's result depends on the embedding of $\X$ into $\pr^n_{\Z}$ and the choice of a coordinate system in $\pr^n_{\Z}$. It would be better to have a more general result without an explicit choice of an embedding into some projective space. This leads us to look for a similar result in the setting of Arakelov geometry. 

\subsection{Main theorems}

Let $\X$ be a projective arithmetic variety, i.e. an integral separated scheme which is flat, projective of finite type over $\Spec\ \Z$. If $\X$ is regular and that $\li$ is an ample line bundle on $\X$, we want to define a good density for a subset $\mathcal{P}\subset  \bigcup_{d\geq 0}\ho^{0}(\X,\li^{\otimes d})$ so that the density of the subset of sections $\sigma\in \bigcup_{d\geq 0}\ho^{0}(\X,\li^{\otimes d})$ whose divisor is regular is positive. This will imply the existence of global sections $\sigma$ with regular divisor $\mathrm{div}\sigma$ for sufficiently large $d$.\\

In the arithmetic case, we use the word ``singular'' as ``not regular''. If we say that $\X$ is \emph{singular at a closed point $x$}, we mean that $\X$ is \emph{not regular at $x$}, which means  that 
\[
	\dim_{\kappa(x)} \frac{\mathfrak{m}_{\X, x}}{\mathfrak{m}^2_{\X, x}}\not=\dim \X,
\]
where $\mathfrak{m}_{\X, x}$ is the maximal ideal of the stalk $\mathcal{O}_{\X, x}$ of the structure sheaf of scheme $\X$ on $x$ and 	$\dim\X$ is the dimension of $\X$ as a scheme.


In order to get good positivity properties of the ample line bundles on arithmetic varieties, we add on them a Hermitian structure and consider the notion of arithmetic ampleness for Hermitian line bundles on projective arithmetic varieties developed by Henri Gillet and Christophe Soulé in \cite{GS92} and by Shouwu Zhang in \cite{Zh92} (for arithmetic surfaces) and \cite{Zh95}. Assume that $\X$ is a projective arithmetic variety. An ample Hermitian line bundles $\Li=(\li, \left\lVert \cdot \right\rVert)$ on $\X$ is an ample line bundle $\li$ equipped with a Hermitian metric $\left\lVert \cdot \right\rVert$ on the restriction $\li_{\bbC}$ to the fiber $\X(\bbC)$ with additional positivity conditions. For such $\Li$ on $\X$, we consider the set of \emph{strictly effective sections} 
\[
	\ho_{\mathrm{Ar}}^0(\X, \Li):= \{ \sigma\in \ho^0(\X, \Li) \ ;\ \left\lVert \sigma \right\rVert_{\infty}<1 \}
\]
as an analogue of $\ho^0(X,\li)$ for an ample line bundle $\li$ on a projective variety $X$ defined over a field, and 
\[
	h_{\mathrm{Ar}}^0(\X, \Li):= \log\left( \#\ho_{\mathrm{Ar}}^0(\X, \Li) \right)
\]
as an analogue of $h^0(X,\li)$. Here
\[
	\left\lVert \sigma \right\rVert_{\infty}=\sup_{z\in \X(\bbC)}\left\lVert \sigma(z) \right\rVert. 
\]

We will give a precise definition of an ample Hermitian line bundle and discuss some of its properties in Section \ref{arithample}. \\

For a fixed ample Hermitian line bundle $\Li$, we say that a subset $\mathcal{P}$ of $\bigcup_{d\geq 0}\ho^0(\X, \Li^{\otimes d})$ has Arakelov density $\rho$ for some $0\leq \rho\leq 1$ if
\[
	\lim_{d\rightarrow \infty} \frac{ \#\left(\mathcal{P}\cap \ho_{\mathrm{Ar}}^0(\X, \Li^{\otimes d})\right) }{\#\ho_{\mathrm{Ar}}^0(\X, \Li^{\otimes d})}=\rho.
\]
We define the upper density and lower density in the same way. We denote the density, the upper density and the lower density of $\mathcal{P}$, when exist, by $\mu_{\mathrm{Ar}}(\mathcal{P}),\overline{\mu_{\mathrm{Ar}}}(\mathcal{P})$ and $\underline{\mu_{\mathrm{Ar}}}(\mathcal{P})$, respectively.\\

Our main result is the following:

\begin{thm}\label{mainvar}
	Let $\X$ be a regular projective arithmetic variety of dimension $n$, and let $\Li$ be an ample Hermitian line bundle on $\X$. 
	There exists a constant $\epsilon_0>0$ such that for any $\varepsilon$ with $0<\varepsilon<\epsilon_0$ by denoting
	\[
		\mathcal{P}_{d,p\leq e^{\varepsilon d}}:=\left\{ \sigma\in \ho^0(\X,\Li^{\otimes d})\ ;\ \begin{array}{ll}
		\mathrm{div}\sigma \text{ has no singular point of residual} \\
		\text{characteristic smaller than or equal to } e^{\varepsilon d}
		\end{array}\right\}
	\]
	and $\mathcal{P}_{A,\varepsilon}=\bigcup_{d\geq 0} \mathcal{P}_{d,p\leq e^{\varepsilon d}}$, we have
	\[
		\mu_{\mathrm{Ar}}(\mathcal{P}_{A,\varepsilon})=\zeta_{\X}(1+n)^{-1},
	\]
	where $\zeta_{\X}(s)$ is the zeta function
	\[
		\zeta_{\X}(s)= \prod_{x\in |\X|}(1-\#\kappa(x)^{-s})^{-1}.
	\]
\end{thm}
Here $\kappa(x)$ is the residual field of $x$, and the residual characteristic of a closed point $x$ in $\X$ is the characteristic of its residue field.

\begin{thm}\label{regular}
	Let $\X$ be a regular projective arithmetic variety of dimension $n$, and let $\Li$ be an ample Hermitian line bundle on $\X$. Set
	\[
		\mathcal{P}_{d}:=\left\{ \sigma\in \ho^0(\X,\Li^{\otimes d})\ ;\ 
		\mathrm{div}\sigma \text{ is regular} \right\}
	\]
	and $\mathcal{P}=\bigcup_{d\geq 0} \mathcal{P}_{d}$. We have
	\[
		\overline{\mu_{\mathrm{Ar}}}(\mathcal{P})\leq\zeta_{\X}(1+n)^{-1},
	\]
	where $\overline{\mu_{\mathrm{Ar}}}(\mathcal{P})$ is the upper density of $\mathcal{P}$.
\end{thm}
\begin{proof}[Proof]
	If a section $ \sigma\in \ho_{\mathrm{Ar}}^0(\X,\Li^{\otimes d})$ is such that $\mathrm{div}\sigma$ is regular, then in particular it has no singular point of residual characteristic smaller than or equal to $e^{\varepsilon d}$ with constant $\varepsilon$ as in Theorem \ref{mainvar}. So naturally $\mathcal{P}_d\subset \mathcal{P}_{d, p\leq e^{\varepsilon d}}$ and $\mathcal{P}\subset \mathcal{P}_{A,\varepsilon}$. Therefore we have
	\[
		\overline{\mu_{\mathrm{Ar}}}(\mathcal{P})\leq \mu_{\mathrm{Ar}}(\mathcal{P}_{A,\varepsilon})=\zeta_{\X}(1+n)^{-1}.
	\]
\end{proof}
\begin{cor}\label{singR}
	Let $\X$ be a regular projective arithmetic variety of dimension $n$, and let $\Li$ be an ample Hermitian line bundle on $\X$. 
	There exists a constant $c>1$ such that for any $R>1$ we have
	\[
		\lim_{d\rightarrow \infty}\frac{\#\left\{ \sigma\in \ho^0(\X, \Li^{\otimes d})\ ;\ \begin{array}{ll} \lVert \sigma \rVert_{\infty}<R^d,\ \mathrm{Sing}(\mathrm{div}\sigma) \text{ has no point of residual} \\ \text{ characteristic smaller than or equal to } (cR)^{\frac{d}{2}} \end{array}\right\}}{\#\left\{ \sigma\in \ho^0(\X, \Li^{\otimes d})\ ;\ \lVert \sigma \rVert_{\infty}<R^d \right\}}=\zeta_{\X}(1+n)^{-1}.
	\]
\end{cor}

Charles used this notion of density in \cite{Ch17} to prove the analogous Bertini irreducibility theorem for arithmetic varieties, which says that if $\X$ is an irreducible arithmetic variety of dimension at least $2$ and $\Li$ an ample Hermitian line bundle on $\X$, then the set of global sections in $\bigcup_{d\geq 0}\ho^0(\X, \Li^{\otimes d})$ whose divisor is irreducible has density $1$. 
The result of Charles can also be compared to the result of Breuillard and Varjú in \cite{BV19} for polynomials only with coefficients in $0$ and $1$. Breuillard and Varjú showed that if we admit the Riemann hypothesis for the Dedekind zeta function $\zeta_K$ for all number fields of the form $K=\Q(a)$ for some root $a$ of a polynomial with $0$, $1$ coefficients, the density of the subset of irreducible polynomials in the set of polynomials $P(X)$ of $1$ variable with $0$, $1$ coefficients such that $P(0)\not=0$ is $1$. (In each degree $d$, such polynomials are finite in number, so the density can be defined by the limit of proportion when $d$ tends to infinity.) \\

\subsection{Comparison with earlier results}
We compare our result with some existing results.


We first recall the result of Poonen that we already mentioned. In \cite{Po04} Poonen established a density for $\bigcup_{d\geq 0}\ho^0(\pr^n_{\Z}, \mathcal{O}(d))$ on the projective space $\pr^n_{\Z}$. Let $\mathcal{P}=\bigcup_{d\geq 0}\mathcal{P}_d$ be a subset of $\bigcup_{d\geq 0}\ho^0(\pr^n_{\Z}, \mathcal{O}(d))$, where $\mathcal{P}_d\subset \ho^0(\pr^n_{\Z}, \mathcal{O}(d))$. For any $d$, we have a natural $\Z$-basis of $\ho^0(\pr^n_{\Z}, \mathcal{O}(d))$ which is composed of all monomials of degree $d$. For simplicity of notations we denote them by $f_{d,1},\dots, f_{d,h_d}$ where $h_d=h^0(\pr^n_{\Q}, \mathcal{O}(d))$. Any section $f\in \ho^0(\pr^n_{\Z}, \mathcal{O}(d))$ can be written as $f=\sum_{i=1}^{h_d}a_if_{d,i}$ with some $a_i\in \Z$ for each $i$. Poonen defines the upper density of $\mathcal{P}_d$ as	
\[
	\overline{\mu}_{\mathrm{P},d}(\mathcal{P})=\max_{\tau\in \mathfrak{S}_{h_d}}\limsup_{B_{\tau(1)}\rightarrow \infty}\cdots \limsup_{B_{\tau(h_d)}\rightarrow \infty}\frac{\#\left(\mathcal{P}_d\cap \{ \sum_{i=1}^{h_d}a_if_{d,i}\in  \ho^0(\pr^n_{\Z}, \mathcal{O}(d))\ ;\ |a_i|\leq B_i,\ \forall i \}\right)}{\#\{ \sum_{i=1}^{h_d}a_if_{d,i}\in  \ho^0(\pr^n_{\Z}, \mathcal{O}(d))\ ;\ |a_i|\leq B_i,\ \forall i \}}.
\]
Here $\mathfrak{S}_{h_d}$ is the symmetric group of $h_d$ symbols. The upper density of $\mathcal{P}$ is then defined by
\[
	\overline{\mu}_{\mathrm{P}}(\mathcal{P})=\limsup_{d\rightarrow \infty} \overline{\mu}_{\mathrm{P},d}(\mathcal{P}_d).
\]
The lower density of $\mathcal{P}$ is defined similarly and the density of $\mathcal{P}$ exists if its upper and lower density coincide. Using this density, Poonen proved the following theorem:

\begin{thm}[Poonen, Theorem 5.1 of \cite{Po04}]\label{poonen}
	When $\X$ is a regular subscheme of dimension $m$ of $\pr^n_{\Z}$, assuming the \emph{abc} conjecture and a supplementary conjecture which holds at least when $\X$ is projective, the density of the set of sections $f\in \bigcup_{d\geq 0}\ho^0(\pr^n_{\Z}, \mathcal{O}(d))$ such that $\mathrm{div}(f)\cap \X$ is regular of dimension $m-1$ is $\zeta_{\X}(m+1)^{-1}$. 
\end{thm}

\begin{rmq}
	In Poonen's proof, the \emph{abc} conjecture is used to show that for any fixed $d$, the upper density of global sections whose divisor has a singular point on a fiber over a prime number $p\geq M$ with $M>0$ tends to $0$ when $M$ tends to infinity. The proof of this follows the idea of A. Granville in \cite{Gr98} that for a polynomial $f(x)\in \Z[x]$ we can get an asymptotic control of the prime squarefactors of $f(n)$ by the norm of $n\in \Z$.  Poonen generalized this idea to the case of multivariable polynomials in \cite{Po03}. 
\end{rmq}

Essentially, in each degree $d$, we get the density $\overline{\mu}_d$ by taking the limit of coefficients one by one. The action of symmetric group adds the condition that the order of coefficients can be arbitrary. This theorem permits us to find global sections $f\in\ho^0(\pr^n_{\Z}, \mathcal{O}(d))$ such that $\mathrm{div}(f)\cap \X$ is regular of dimension $m-1$ for sufficiently high $d$. 

The density defined by Poonen depends on a choice of coordinates of $\pr^n_{\Z}$. These coordinates determine which global sections are monomials in each $\ho^0(\pr^n_{\Z}, \mathcal{O}(d))$. His method is hard to be applied to a more general case, for example when we consider an arithmetic variety other than $\pr^n_{\Z}$ equipped with an ample line bundle which may not be very ample. Moreover, the size of the global sections with regular divisor cannot be controlled using this method. The global sections having regular divisor may have very large coefficients as polynomials. Finally, the \emph{abc} conjecture is powerfully used in his proof. Without it, the proof can give control of sections whose divisor does not have singular points of finitely many fixed residual characteristics, but cannot give the limit of the proportion of global sections of $\mathcal{O}(d)$ whose divisor has no singular point of residual characteristic smaller than or equal to $e^{\varepsilon d}$ for a constant $\varepsilon $ as we do.
\\

In \cite{BSW16}, Bhargava, Shankar and Wang proved that monic integer polynomials of one variable $f(x)=x^d+a_1x^{d-1}+\cdots+a_d\in V^{\mathrm{mon}}_d(\Z)$ such that $\Z[x]/(f(x))$ is the ring of integers of the field $\Q[x]/(f(x))$ has density $\zeta(2)^{-1}$. Here the density is constructed using the size of the coefficients of polynomials. 

This result can be viewed as a version of Bertini regularity theorem for $\pr^1_{\Z}$. In fact, the condition that $\Z[x]/(f(x))$ is the ring of integers of the field $\Q[x]/(f(x))$ means exactly that 
$\Spec\ \Z[x]/(f(x))$ is regular. When we homogenize $f$ to the global section 
\[
	F(X,Y)=X^d+a_1X^{d-1}Y+\cdots+ a_dY^d\in \ho^0(\pr^1_{\Z}, \mathcal{O}(d)),
\]
this means that $\mathrm{div} (F)$ is a regular divisor of $\pr^1_{\Z}$.

In their paper, fixing the degree $d>1$, they order the monic integer polynomials 
\[
	f(x)=x^d+a_1x^{d-1}+\cdots+a_d
\]
by a height function
\[
	H(f):=\max\{ |a_i|^{\frac{1}{i}} \},
\]
and calculate the density of a subset $\mathcal{P}_d\subset V^{\mathrm{mon}}_d(\Z)$ by
\[
	\mu_{H,d}(\mathcal{P}_d)=\lim_{R\rightarrow \infty} \frac{ \#\left(\mathcal{P}_d\cap \{ f\in V^{\mathrm{mon}}_d(\Z)\ ;\ H(f)\leq R \} \right)}{\#\{ f\in V^{\mathrm{mon}}_d(\Z)\ ;\ H(f)\leq R \}}.
\]
Identifying $V^{\mathrm{mon}}_d(\Z)$ with the set $\{ F\in \ho^0(\pr^1_{\Z}, \mathcal{O}(d))\ ;\ \mathrm{div}\ F\cap \infty_{\Z}=\emptyset \}$ by homogenization, the density of 
$\mathcal{P}_d$ can be understood as
\[
	\mu_{H,d}(\mathcal{P}_d)=\lim_{R\rightarrow \infty} \frac{ \#\left(\mathcal{P}_d\cap \{ F\in \ho^0(\pr^1_{\Z}, \mathcal{O}(d))\ ;\ \mathrm{div}\ F\cap \infty_{\Z}=\emptyset,\  H(F)\leq R \} \right)}{\#\{ F\in \ho^0(\pr^1_{\Z}, \mathcal{O}(d))\ ;\ \mathrm{div}\ F\cap \infty_{\Z}=\emptyset,\  H(F)\leq R \}}.
\]
We can then reformulate \cite[Theorem 1.2]{BSW16} as follows:
\begin{thm}[Theorem 1.2 of \cite{BSW16}]
	For a fixed $d>1$, set 
	\[
		\mathcal{P}_d:=\{ F\in \ho^0(\pr^1_{\Z}, \mathcal{O}(d))\ ;\ \mathrm{div}\ F\cap \infty_{\Z}=\emptyset,\ \mathrm{div}\ F \text{ is regular of dimension }1 \}.
	\]
	Then we have
	\[
		\mu_{H,d}(\mathcal{P}_d)=\zeta(2)^{-1}.
	\]
\end{thm}
\begin{rmq}
	This result is similar to Poonen's theorem. In fact, if we note that $\zeta(2)^{-1}$ can be expressed by values of the zeta function of the affine line over $\Spec\ \Z$
	\[
		\zeta_{\A^1_{\Z}}(s)=\prod_{p}\zeta_{\A^1_{\F_p}}=\prod_{p}\frac{1}{1-p^{1-s}},
	\]
	which is,
	\[
		\zeta(2)^{-1}=\prod_p (1-p^{-2})=\zeta_{\A^1_{\Z}}(3)^{-1},
	\]
	then the theorem tells us that, for any $d>1$, the density of the subset $\mathcal{P}_d$ of 
	\[
		\{ F\in \ho^0(\pr^1_{\Z}, \mathcal{O}(d))\ ;\ \mathrm{div}\ F\cap \infty_{\Z}=\emptyset \}
	\]
	consisting of sections with regular divisor is equal to 
	\[
		\zeta_{\pr^1_{\Z}-\infty_{\Z}}\left(1+\dim(\pr^1_{\Z}-\infty_{\Z}) \right)^{-1}=\zeta_{\A^1_{\Z}}(3)^{-1}.
	\]
	This is a statement similar to Theorem \ref{poonen}. But as we only consider global sections whose divisor is disjoint of $\infty_{\Z}$, we can not recover Poonen's theorem for $\X=\A^1_{\Z}$ in $\pr^1_{\Z}$. 
\end{rmq}

The result of Bhargava, Shankar and Wang surpasses our results for $\pr^1_{\Z}$ in the sense that for any $d>1$, they can actually find global sections of $\ho^0(\pr^1_{\Z}, \mathcal{O}(d))$ whose divisor is regular without auxiliary assumptions. Neither can we get such a strong statement using Poonen's method. But the method of Bhargava, Shankar and Wang is hard to be generalized to other situations. Their proof depends on the monogenicity of the finite $\Z$-algebra $\Z[x]/(f(x))$. They constructed a map from the moduli space of monogenic finite $\Z$-algebras of length $d$ to the space of symmetric $n\times n$ matrices quotient by the action of group $SO(A_0)$, where $A_0$ is the $n\times n$ anti-diagonal matrix. This map is then used in the article to turn the counting of monic polynomials to the counting of special orbits in this quotient space. Due to the construction of this map, it is difficult to release the monic condition in their theorem so as to get a result for all polynomials with coefficients in $\Z$. It is even more difficult to generalize this method to regular arithmetic varieties other than $\pr^1_{\Z}$.\\

In \cite{Au01}, \cite{Au02}, Autissier showed another arithmetic analogue of the Bertini theorems. He proved as a particular case that if $\X$ is an arithmetic variety of dimension $n$ over an integer ring $\mathcal{O}_K$ (where $K$ here is a number field), and $\Li$ a very ample Hermitian line bundle on $\X$, then there exists a finite extension $L$ of $K$ and a section $\sigma\in \ho^0(\X_{\mathcal{O}_L}, \Li)$ such that by writing $g:\Spec\ \mathcal{O}_L\longrightarrow \Spec\ \mathcal{O}_K$ the morphism induced by $\mathcal{O}_K\lhook\joinrel\longrightarrow \mathcal{O}_L$, for any closed point $b\in \Spec\ \mathcal{O}_L$, the fiber $(\mathrm{div}\sigma)_b$ of the divisor $\mathrm{div}\sigma$ is smooth if $\X_{g(b)}$ is smooth. Moreover, we can bound the height of $\mathrm{div}\sigma$ (defined by $\Li$) in terms of the height of $\X$, $\deg_{\li_{\Q}}\X_{\Q}$, $n$ and an effective constant which is only dependent of $\Li$ et $n$.

This result is stronger than ours in the sense where the divisor that he gives comes from a global section of the sheaf $\Li$ but not $\Li^{\otimes d}$ for a large $d$, and moreover the divisor satisfies the smoothness condition rather that the regularity condition. The disadvantage of this result is that it need to pass to finite base changes to find such a global section. In particular, if $\X$ is defined over $\Spec\ \Z$, there's little chance that we can find a divisor satisfying the smoothness condition in the statement which is defined over $\Spec\ \Z$ by Autissier's method. Our result, on the other hand, provides divisors which are defined over $\Spec\ \Z$ if so is the arithmetic variety that we consider.
\bigskip

As we can see, the densities in the above results are not defined in a natural way. We construct them using additional information on polynomials. In particular, the coordinate system on projective spaces are often needed for these constructions. Our objective is to construct a density for arithmetic varieties in a more natural way and get rid of this choice of coordinates. It is well-known that when we study arithmetic varieties, it is usually better to consider vector bundles with Hermitian metric on the complex fiber. In particular, we get good properties as arithmetic ampleness, arithmetic Riemann-Roch theorem, and the set of effective sections behaves well in the Hilbert-Samuel formula. Our construction of the Arakelov density via the set of effective sections of Hermitian line bundles should be a better approach for a more natural notion of density for arithmetic varieties.

\subsection{Method of proof}
The proof of Theorem \ref{mainvar} relies on an effective estimate of proportion of global sections whose divisor has no singular point on one single fiber. This estimate can be reduced to computing, for a projective arithmetic variety $\X$ of dimension $n$ with an ample Hermitian line bundle $\Li$, the proportion of $\sigma\in \ho^0(\X_{p^2}, \li^{\otimes d}|_{\X_{p^2}})$ such that for any closed point $x\in \mathrm{div}\sigma$, $ \dim_{\kappa(x)}\frac{\m_{\mathrm{div}\sigma,x}}{\m^2_{\mathrm{div}\sigma, x}}=n-1$, where 
\[
	\X_{p^2}=\X\times_{\Spec\ \Z}\Spec\ \Z/p^2\Z. 
\]
In fact, for any closed point $x$ of $\X$ on the fiber $\X_p$, and any global section $\sigma_0\in \ho^0(\X, \Li^{\otimes d})$, the divisor $\mathrm{div}\sigma_0$ is singular at $x$ if and only if the restriction of $\sigma_0$ on the first order infinitesimal neighbourhood $x'$ of $x$ is $0$. Note that $x'$ is defined by the square of the maximal ideal of $x$ in $\X$, it is actually a closed subscheme of $\X_{p^2}$ (but not of $\X_p$), and the condition to test whether $\mathrm{div}\sigma_0$ is singular at $x$ depends only on the restriction $\sigma_0|_{\X_p^2}$. As the proportion in the group $\ho^0(\X_{p^2}, \li^{\otimes d}|_{\X_{p^2}})$ is easier to compute than in the set $\ho^0_{\mathrm{Ar}}(\X, \Li^{\otimes d})$, we first estimate the above proportion, and then lift it to the proportion of sections in $\ho^0_{\mathrm{Ar}}(\X, \Li^{\otimes d})$ whose divisor has no singular point on the fiber $\X_p$ via a proportion-lifting result about $\ho^0_{\mathrm{Ar}}(\X, \Li^{\otimes d})\longrightarrow \ho^0(\X_{p^2}, \li^{\otimes d}|_{\X_{p^2}})$.

We first generalize Poonen's Bertini smoothness theorem over finite fields in the appendix, replacing the very ample condition by ampleness. Then the generalized proof can be applied to get the estimate on one single fiber. With a choice of positive integers $r_{p,d},N(p)$, where $r_{p,d}$ depends on $p,d$ and $N(p)$ depends only on $p$, we give estimates of proportion of sections $\sigma\in \ho^0(\X_{p^2}, \li^{\otimes d}|_{\X_{p^2}})$ whose divisor has a singular point of degree smaller than or equal to $r_{p,d}$, between $r_{p,d}$ and $\frac{d}{nN(p)}$ and larger than $\frac{d}{nN(p)}$, respectively. Then we conclude by putting together these three estimates.

The estimate on one single fiber can be easily extended to finitely many fibers. The effective estimates permit us to show that we can gather all fibers over $p$ such that $p\leq d^{\frac{1}{n+1}}$ without ruining the convergence of the proportion of $\sigma \in \ho^0_{\mathrm{Ar}}(\X, \Li^{\otimes d})$ such that $\mathrm{div}\sigma$ has no singular point on all these fibers.

Then we use a different method to show that there exists a constant $c>0$ such that for any prime $ p \leq e^{\varepsilon d}$ with constant $\varepsilon$ satisfying the condition in Theorem \ref{mainvar} such that the fiber over $p$ is smooth and irreducible (these two conditions are satisfied by all but finitely many $p$), the proportion of strictly effective global sections whose divisor has singular points on this fiber is smaller than or equal to $cp^{-2}$. Consequently the proportion of $\sigma \in \ho^0_{\mathrm{Ar}}(\X, \Li^{\otimes d})$ such that $\mathrm{div}\sigma$ has singular points on the fiber $\X_p$ for some $d^{\frac{1}{n+1}}\leq p \leq e^{\varepsilon d}$ is bounded above by 
\[
	\sum_{d^{\frac{1}{n+1}}\leq p \leq e^{\varepsilon d}}cp^{-2}, 
\]
which tends to $0$ when $d$ tends to infinity. This together with the above estimate for $p\leq d^{\frac{1}{n+1}}$ proves Theorem \ref{mainvar}.

\subsection{Organization of the paper}
In Section \ref{arithample} we recall the definition of arithmetic ampleness introduced by Shouwu Zhang in \cite{Zh92} and \cite{Zh95} as well as some properties of ample Hermitian line bundles such as the arithmetic Hilbert-Samuel formula and present some results on restrictions to a subscheme.  In Section \ref{c0} we gather two results on estimates of the convergences of special values of zeta functions.  In Section \ref{effective} we estimate how fast the proportion of strictly effective global sections whose divisor has no singular point on a given special fiber over $p\in \Spec\ \Z$ tends to $\zeta_{\X_p}(1+n)^{-1}$. In Section \ref{smallresidual} we use the results of the previous sections to show that the proportion of strictly effective global sections whose divisor has no singular point of residual characteristic smaller than or equal to $d^{\frac{1}{n+1}}$ tends to $\zeta_{\X}(1+n)^{-1}$ when $d$ tends to infinity. In Section \ref{finalresult}, we prove that the proposition of sections in $\ho^0_{\mathrm{Ar}}(\X, \Li^{\otimes d})$ whose divisor has singular points on a special fiber over $p$ can be bounded above by $cp^{-2}$ with some positive constant $c$ independent of $p$ and $d$ when $d$ and $p$ are large. We use these results to prove Theorem \ref{mainvar} and Corollary \ref{singR} in the same section. In Appendix \ref{finitefield}, we give a proof of a generalized Poonen's Bertini smoothness theorem over finite fields where we allow the sheaf to be ample instead of very ample. This theorem is not directly used in the main part of this paper, but some results in the appendix are used in Section \ref{effective} and \ref{finalresult}.

\subsection{Notation}
\begin{enumerate}
	\item For a finite set $S$, we denote by $\#S$ its cardinality.
	\item For a positive real number $x$, we define
	\begin{eqnarray*}
		\lfloor x\rfloor= \max\{ n\in \Z\ ;\ n\leq x \},\quad
		\lceil x\rceil = \min\{ n\in \Z\ ;\ n\geq x \}.
	\end{eqnarray*}
	\item Let $f,g: \R_{\geq 0}\longrightarrow \R$ be two real continuous functions such that $f(0)=g(0)=0$. We say $f=O(g)$ if there exist $c>0$ and $\varepsilon>0$ such that for any $0<x<\varepsilon$ we have
	\[
		f(x)\leq c\cdot g(x).
	\]
	We say $f\sim g$ if $f=O(g)$ and $g=O(f)$.
	\item For an arithmetic variety $\X$ and a positive integer $N$, we write $\X_N$ for the closed subscheme $\X\times_{\Spec\ \Z}\Spec(\Z/N\Z)$.
	\item For an arithmetic variety $\X$ equipped with an ample Hermitian line bundle $\Li$, if $Y$ is a subscheme of $\X$ such that $Y_{\Q}\not=\emptyset$, we set $\ho^0(Y,\Li):=\ho^0(Y, \Li|_Y)$; if $Y$ is a subscheme of $\X$ such that $Y_{\Q}=\emptyset$, we set $\ho^0(Y,\li):=\ho^0(Y, \li|_Y)$.
\end{enumerate}

\subsection{Acknowledgement}
The author is very grateful to François Charles for introducing him to the subject, for the various discussions they had and for the helpful guidance the author has received. The author thanks Yang Cao, Étienne Fouvry, Salim Tayou for their useful conversations. 
	
This project has received funding from the European Research Council (ERC) under the European Union’s Horizon 2020 research and innovation programme (grant agreement No 715747).

\section{Arithmetic ampleness}\label{arithample}

In this section, we discuss arithmetic ampleness for arithmetic varieties, i.e. integral separated schemes which are flat and of finite type over $\Spec\ \Z$. This notion is established in \cite{Zh92} and \cite{Zh95}.

\subsection{Basic properties}
\begin{defi}
	Let $M$ be a complex analytic space. Let $\overline{L}=(L,\left\lVert \cdot \right\rVert)$ be a Hermitian line bundle on $M$, where $\left\lVert \cdot \right\rVert$ is a continuous Hermitian metric on $L$. Then $\overline{L}$ is said to be \emph{semipositive} if for any section $\sigma$ of $\overline{L}$ on any open subset $U$ of $M$ such that $s$ does not vanish on any point of $U$, the function $-\log \left\lVert \sigma \right\rVert$ is plurisubharmonic on $U$.
\end{defi}
\begin{rmq}
	If $M$ is a complex manifold and the Hermitian metric $\left\lVert \cdot \right\rVert$ on $\overline{L}$ is of differentiability class $C^2$, then saying that $\overline{L}$ is semipositive is equivalent to saying that for any section $s$ on any open subset $U$ of $M$ such that $s$ does not vanish on any point of $U$, we have that
	\[
		\sqrt{-1}\partial \overline{\partial}\left(-\log \left\lVert s \right\rVert\right)
	\]
	is a non-negative $(1,1)$-form.
\end{rmq}
In this article, we always demands that the Hermitian metric on a Hermitian line bundle is smooth.

\begin{defi}
	Let $\X$ be a projective arithmetic variety, which means an arithmetic variety projective over $\Spec\ \Z$. We say that $\Li=(\li, \left\lVert\cdot\right\rVert)$ is a \emph{Hermitian line bundle} on $\X$ if $\li$ is a line bundle on $\X$ and $\left\lVert\cdot\right\rVert$ is a smooth Hermitian metric on $\li|_\X(\bbC)$ which makes $\Li|_{\X(\bbC)}=(\li|_{\X(\bbC)}, \left\lVert\cdot\right\rVert)$ a Hermitian line bundle as defined above. We say that $\Li$ is \emph{ample} on $\X$ if it satisfies the following three conditions:
	\begin{enumerate}[label=\roman*)]
		\item $\li$ is ample over $\Spec\ \Z$;
		\item $\Li$ is semipositive on the complex analytic space $\X(\bbC)$;
		\item for any $d\gg 1$, $\ho^0(\X,\Li^{\otimes d})$ is generated by sections of norm strictly smaller than $1$.
	\end{enumerate}
\end{defi}

For any Hermitian line bundle $\Li=(\li, \left\lVert\cdot\right\rVert)$ and any real number $\delta$, we note $\Li(\delta)$ the Hermitian line bundle $(\li, \left\lVert\cdot\right\rVert e^{-\delta})$. If $\Li$ is ample, it is easy to see that $\Li(\delta)$ is also ample for any $\delta>0$.

A useful result concerning ample Hermitian line bundles is the following proposition, which is a simple version of \cite[Prop. 2.3]{Ch17}:
\begin{prop}\label{epsilon}
	Let $\X$ be a projective arithmetic variety, and let $\Li$ be an ample Hermitian line bundle on $\X$ and $\overline{\mathcal{M}}$ a Hermitian vector bundle of rank $r$ on $\X$. There exists a positive constant $\varepsilon_0$ such that for any large enough integer $d$, $\ho^0(\X, \Li^{\otimes d}\otimes \overline{\mathcal{M}})$ has a basis consisting of sections with norm smaller than $e^{-\varepsilon_0 d}$.
\end{prop}
Now we recall the arithmetic Hilbert-Samuel theorem for ample Hermitian line bundles on projective arithmetic varieties, which is proved by Gillet and Soulé in \cite{GS92} using arithmetic Riemann-Roch theorem for arithmetic varieties with smooth generic fiber and generalized by Zhang in \cite{Zh95} for arithmetic varieties whose generic fiber may be singular. There is also a proof given by Abbes and Bouche in \cite{AB95} without the application of the arithmetic Riemann-Roch theorem. 
\begin{thm}\label{hilbsamu}
	Let $\X$ be a projective arithmetic variety of absolute dimension $n$,  $\Li$ an ample Hermitian line bundle on $\X$ and $\overline{\mathcal{M}}$ a Hermitian vector bundle of rank $r$ on $\X$. As $d$ tends to $\infty$, we have
	\[
		\mathrm{h}^0_{\mathrm{Ar}}(\X,\Li^{\otimes d}\otimes \overline{\mathcal{M}})= \frac{r}{n!}\Li^{n} d^n+o(d^n).
	\]
\end{thm}
\begin{proof}[Proof]
	The statement can be found in \cite[Corollary 2.7 (1)]{Yu08}. We get this by combining two results. The first one is the arithmetic Riemann-Roch formula estimating $\chi_{\mathrm{sup}}(\Li^{\otimes d}\otimes \overline{\mathcal{M}})$, which is the logarithm of the covolume of the lattice $\ho^0(\X, \Li^{\otimes d}\otimes \overline{\mathcal{M}})$ for the sup norm. This result is proved by Gillet-Soulé in \cite{GS92} by combining the Gromov inequality (see for example \cite[Corollary 2.7 (2)]{Yu08}), as
	\[
		\chi_{\mathrm{sup}}(\Li^{\otimes d}\otimes \overline{\mathcal{M}})=\frac{r}{n!}\Li^{n} d^n+o(d^n),
	\]
	when the generic fiber of the arithmetic variety is smooth. A generalization to the case where we drop the smoothness condition on the generic fiber is given by Shouwu Zhang in \cite[Theorem (1.4)]{Zh95}.
	
	The second result, which says that 
	\[
		|\mathrm{h}^0_{\mathrm{Ar}}(\X,\Li^{\otimes d}\otimes \overline{\mathcal{M}})-\chi_{\mathrm{sup}}(\Li^{\otimes d}\otimes \overline{\mathcal{M}})|
		=O(d^{n-1}\log d),
	\]
	is a consequence of \cite[Theorem 2]{GS91}, together with a Theorem of Zhang (\cite[Theorem 4.2]{Zh95}), which implies that $\mathrm{h}^1_{\mathrm{Ar}}(\X,\Li^{\otimes d}\otimes \overline{\mathcal{M}})=0$ when $d$ is large enough.
\end{proof}

\subsection{Restriction modulo \emph{N} of sections}\label{arithample2}
\begin{lem}\label{freemod}
	Let $\X$ be a projective arithmetic variety, and let $\li$ be an ample Hermitian line bundle on $\X$. There is a positive integer $d_0$ such that when $d\geq d_0$, we have 
	\[
		 \ho^0(\X_N,\li^{\otimes d})\simeq\ho^0(\X,\li^{\otimes d})/\big(N\cdot \ho^0(\X,\li^{\otimes d})\big)
	\]
	for any positive integer $N$.
\end{lem}
\begin{proof}[Proof]
	By the definition of $\X_N$, we have $\mathcal{O}_{\X_N}\simeq \mathcal{O}_{\X}/(N\cdot \mathcal{O}_{\X})$. Therefore we have that on $\X$, 
	\[
		\li^{\otimes d}\otimes \mathcal{O}_{\X_N}\simeq \li^{\otimes d}/(N\cdot \li^{\otimes d}), 
	\]
	which induces an exact sequence of sheaves on $\X$
	\[
		0\longrightarrow \li^{\otimes d}\stackrel{N}\longrightarrow \li^{\otimes d}\longrightarrow \li^{\otimes d}\otimes \mathcal{O}_{\X_N}\longrightarrow 0.
	\]
	Since $\ho^0(\X_N,\li^{\otimes d})=\ho^0(\X,\li^{\otimes d}\otimes \mathcal{O}_{\X_N} )$, we may choose $d_0>0$ so that for any $d\geq d_0$, $\ho^1(\X,\li^{\otimes d})=0$. For such $d$, we have the following exact sequence :
	\[
		0\longrightarrow \ho^0(\X,\li^{\otimes d})\stackrel{N}\longrightarrow\ho^0(\X,\li^{\otimes d})\longrightarrow \ho^0(\X_N,\li^{\otimes d})\longrightarrow 0,
	\]
	and the lemma follows from this exact sequence.

\end{proof}
We have two results concerning the restriction modulo $N$ map.
\begin{prop}\label{modN}
	Let $\X$ be a projective arithmetic variety, and let $\Li$ be an ample Hermitian line bundle on $\X$. Let $0<\alpha_0<1$ be a real number. There exists a positive constant $\eta$ such that for any $d\in\Z_{>0}$ large enough, if $N\in \Z_{>0}$ is bounded above by $\exp(d^{\alpha_0})$, then the following holds:
	\begin{enumerate}[label=\roman*)]
		\item the restriction morphism
		\[
			\psi_{d,N}: \ho^0_{\mathrm{Ar}}(\X,\Li^{\otimes d}) \longrightarrow \ho^0(\X_N,\li^{\otimes d})
		\]
		is surjective;
		\item for any two sections $s,s'$ in $ \ho^0(\X_N,\li^{\otimes d})$, we have
		\[
			\frac{| \#\psi_{d,N}^{-1}(s)- \#\psi_{d,N}^{-1}(s') |}{\#\psi_{d,N}^{-1}(s)}\leq e^{-\eta d}.
		\]
	\end{enumerate}
\end{prop}
This proposition is a reformulation of \cite[Proposition 2.14]{Ch17}. Here the constant $\alpha_0$ can be any real number between $0$ and $1$.
\bigskip

\begin{prop}\label{ed}
	Let $\X$ be a projective arithmetic variety, and let $\Li$ be an ample Hermitian line bundle on $\X$. Let $\varepsilon_0$ be a constant as in Proposition \ref{epsilon}. For a positive integer $N$, let 
	\[
		\psi_{d,N}: \ho^0_{\mathrm{Ar}}(\X,\Li^{\otimes d}) \longrightarrow \ho^0(\X_{N},\li^{\otimes d})
	\]
	be the restriction map. When $d$ is large enough, for any $\delta<\varepsilon_0$, any odd integer $N\leq e^{\delta d}$ and any subset $E\subset \ho^0(\X_{N},\li^{\otimes d}) $,
	we have
	\[
		\frac{\#\psi_{d,N}^{-1}(E)}{\# \ho^0_{\mathrm{Ar}}(\X,\Li^{\otimes d})}\leq 4 \frac{\#E}{\#\ho^0(\X_{N},\li^{\otimes d})}.
	\]
\end{prop}
\begin{proof}[Proof]
		Note that $\ho^0(\X,\Li^{\otimes d})$ is a free $\Z$-module for any $d$. For simplicity of notation, we write $h=\mathrm{rk}(\ho^0(\X,\Li^{\otimes d}))$. We may assume that $d$ is large enough so that for any positive integer $N$ we have $\ho^0(\X_N,\li^{\otimes d})\simeq\ho^0(\X,\li^{\otimes d})/\big(N\cdot \ho^0(\X,\li^{\otimes d})\big)$ by Lemma \ref{freemod}. Let $(\sigma_1,\dots, \sigma_h)$ be a $\Z$-basis of $\ho^0(\X,\Li^{\otimes d})$ such that
		\[
			\lVert \sigma_j\rVert <e^{-\varepsilon_0 d}, \ \forall  j\in \{ 1,\dots, h \}.
		\]
		For an odd integer $N$ such that $0<N\leq e^{\delta d}$ with a fixed $\delta<\varepsilon_0$, we set 
		\[
			D_{d,N}=\left\{ \sigma=\sum_{j=1}^{h}\lambda_j \sigma_j\ ;\ |\lambda_j|\leq \frac{N}{2},\ \lambda_j\in \mathbb{R} \right\}\subset \ho^0(\X,\Li^{\otimes d})\otimes_{ \mathbb{Z}}\mathbb{R}.
		\]
		Then we have
		\[
			 D_{d,N}\cap \ho^0(\X,\Li^{\otimes d})= 
    \left\{  \sigma=\sum_{j=1}^{h}\lambda_j \sigma_j\ ;\ -\frac{N-1}{2}\leq \lambda_j\leq \frac{N-1}{2},\ \lambda_j\in \mathbb{Z} \right\},
		\]
		and
		\[
			\ho^0(\X,\Li^{\otimes d})\otimes_{ \mathbb{Z}}\mathbb{R}=N\cdot \ho^0(\X,\Li^{\otimes d}) +D_{d, N}.
		\]
		Moreover, for any $\sigma \in D_{d,N}$, we have a bound for the norm of $\sigma$
		\[
			\lVert \sigma\rVert \leq h\frac{N}{2}e^{-\varepsilon_0 d}.
		\]
		The existence of such a basis is guaranteed by Proposition \ref{epsilon}. In particular, when $d$ is large enough, as $N\leq e^{\delta d}$ with $\delta <\varepsilon_0$, any $\sigma \in D_{d,N}$ satisfies 
		\[
			\lVert \sigma\rVert \leq \frac{h}{2} e^{(\delta-\varepsilon_0) d}<1.
		\]
		Note that by the expression of $D_{d,N}\cap \ho^0(\X,\Li^{\otimes d})$, we have a 1-1 correspondence between elements in $D_{d,N}\cap \ho^0(\X,\Li^{\otimes d}) $ and elements in $\ho^0(\X_N, \li^{\otimes d})$ induced by the restriction modulo $N$ map. So the map
		\[
			\psi_{d,N} : \ho^0_{\mathrm{Ar}}(\X, \Li^{\otimes d}) \longrightarrow \ho^0(\X_N, \li^{\otimes d})
		\]
		is surjective for such $N$. 
		
		For any $R\in \R_+$, we set 
		\[
			B_d(R)=\left\{ \sigma\in \ho^0(\X,\Li^{\otimes d})\otimes_{\mathbb{Z}}\mathbb{R}\ \ ;\ \ \lVert\sigma\rVert<R \right\}. 
		\]
		Then in particular, we have
		\[
			 B_d(1)\cap \ho^0(\X,\Li^{\otimes d})=\ho^0_{\mathrm{Ar}}(\X,\Li^{\otimes d}).
		\]
		
		For any element $\sigma\in \ho^0(\X,\Li^{\otimes d})\otimes_{\Z}\R$, we can find a $\sigma'\in N\cdot\ho^0(\X,\Li^{\otimes d})$ such that $\sigma-\sigma'\in D_{d, N}$. If moreover $\sigma \in B_d(1)$, we have 
		\[
			\lVert \sigma' \rVert \leq \lVert \sigma\rVert + \lVert \sigma-\sigma' \rVert <1+h\frac{N}{2}e^{-\varepsilon_0 d}.
		\]
		Thus we have two inclusions
		\begin{eqnarray*}
			B_d(1)&\subset& \left( B_d(1+\frac{Nh}{2}e^{-\varepsilon_0 d}) \cap N\cdot \ho^0(\X,\Li^{\otimes d}) \right)+D_{d,N};\\
			\ho^0_{\mathrm{Ar}}(\X,\Li^{\otimes d})&\subset& \left( B_d(1+\frac{Nh}{2}e^{-\varepsilon_0 d}) \cap N\cdot \ho^0(\X,\Li^{\otimes d}) \right)+D_{d,N}\cap \ho^0(\X,\Li^{\otimes d}).
		\end{eqnarray*}
		Note that any element in $\ho^0(\X_N,\li^{\otimes d}) $ has exactly one preimage in $D_{d,N}\cap \ho^0(\X,\Li^{\otimes d}) $. The number of sections $\sigma\in D_{d,N}\cap \ho^0(\X,\Li^{\otimes d})$ such that $\psi_{d,N}(\sigma) \in E$ is equal to $\#E$. Then by the above inclusion, we have
		\begin{eqnarray*}
			\#\psi_{d,N}^{-1}(E)&=&\#\{ \sigma\in \ho^0_{\mathrm{Ar}}(\X,\Li^{\otimes d})\ ;\ \psi_{d,N}(\sigma)\in E \} \\
			&\leq& \#\left( B_d(1+\frac{Nh}{2}e^{-\varepsilon_0 d}) \cap N\cdot \ho^0(\X,\Li^{\otimes d}) \right)\cdot \#E.
		\end{eqnarray*}
		
		Now we bound $ \#\left( B_d(1+\frac{Nh}{2}e^{-\varepsilon_0 d}) \cap N\cdot \ho^0(\X,\Li^{\otimes d}) \right)$. If $\sigma\in B_d(1+\frac{Nh}{2}e^{-\varepsilon_0 d}) \cap N\cdot \ho^0(\X,\Li^{\otimes d})$, any $\sigma'\in \sigma+D_{d,N}$ satisfies
		\[
			\lVert \sigma' \rVert \leq \lVert \sigma\rVert +\lVert \sigma-\sigma'\rVert <1+\frac{Nh}{2}e^{-\varepsilon_0 d}+\frac{Nh}{2}e^{-\varepsilon_0 d}=1+Nhe^{-\varepsilon_0 d}.
		\]
		Hence 
		\[
			 \left( B_d(1+\frac{Nh}{2}e^{-\varepsilon_0 d}) \cap N\cdot \ho^0(\X,\Li^{\otimes d}) \right)+D_{d,N}\subset  B_d(1+Nhe^{-\varepsilon_0 d}),
		\]
		and in particular, we have
		\[
			\mathrm{Vol}\left(  \left( B_d(1+\frac{Nh}{2}e^{-\varepsilon_0 d}) \cap N\cdot \ho^0(\X,\Li^{\otimes d}) \right)+D_{d,N} \right)\leq \mathrm{Vol} \left( B_d(1+Nhe^{-\varepsilon_0 d}) \right).
		\]
		If $\sigma_1,\sigma_2$ are two distinct elements in $B_d(1+\frac{Nh}{2}e^{-\varepsilon_0 d}) \cap N\cdot \ho^0(\X,\Li^{\otimes d})$, the intersection $(\sigma_1+D_{d,N})\cap (\sigma_2+D_{d,N})$ is either empty or a subset in $\ho^0(\X,\Li^{\otimes d})\otimes_{\mathbb{Z}}\mathbb{R}$ of dimension smaller than $h$. In particular, the intersection always has volume $0$. Therefore we have
		\begin{eqnarray*}
			& &\mathrm{Vol}\left(  \left( B_d(1+\frac{Nh}{2}e^{-\varepsilon_0 d}) \cap N\cdot \ho^0(\X,\Li^{\otimes d}) \right)+D_{d,N} \right) \\
			&=&\# \left( B_d(1+\frac{Nh}{2}e^{-\varepsilon_0 d}) \cap N\cdot \ho^0(\X,\Li^{\otimes d}) \right)\cdot \mathrm{Vol}(D_{d,N}).
		\end{eqnarray*}
		From this equality, we can bound $\# \left( B_d(1+\frac{Nh}{2}e^{-\varepsilon_0 d}) \cap N\cdot \ho^0(\X,\Li^{\otimes d}) \right)$ as 
		\begin{eqnarray*}
			& &\# \left( B_d(1+\frac{Nh}{2}e^{-\varepsilon_0 d}) \cap N\cdot \ho^0(\X,\Li^{\otimes d}) \right)\\
			&=&\frac{\mathrm{Vol}\left(  \left( B_d(1+\frac{Nh}{2}e^{-\varepsilon_0 d}) \cap N\cdot \ho^0(\X,\Li^{\otimes d}) \right)+D_{d,N} \right)}{\mathrm{Vol}(D_{d,N})}\\
			&\leq& \frac{\mathrm{Vol} \left( B_d(1+Nhe^{-\varepsilon_0 d}) \right)}{\mathrm{Vol}(D_{d,N})}			
		\end{eqnarray*}

		Now set
		\[
			D_{d,1}=\left\{ \sigma=\sum_{j=1}^{h}\lambda_j \sigma_j\ ;\ |\lambda_j|\leq \frac{1}{2},\ \lambda_j\in \mathbb{R} \right\}\subset \ho^0(\X,\Li^{\otimes d})\otimes_{ \mathbb{Z}}\mathbb{R}.
		\]
		We get similarly that for any $\sigma\in D_{d,1}$, $\lVert \sigma\rVert \leq \frac{h}{2}e^{-\varepsilon_0 d}$.
		If $\sigma\in \ho^0(\X,\Li^{\otimes d})\otimes_{ \mathbb{Z}}\mathbb{R}$ satisfies $\lVert\sigma\rVert <1-\frac{h}{2}e^{-\varepsilon_0 d}$, then we can find a section $\sigma'\in \ho^0(\X,\Li^{\otimes d})$ such that $\sigma\in \sigma'+D_{d, 1}$; moreover, we get 
		\[
			\lVert \sigma' \rVert \leq \lVert \sigma \rVert + \lVert \sigma-\sigma' \rVert <1-\frac{h}{2}e^{-\varepsilon_0 d}+\frac{h}{2}e^{-\varepsilon_0 d}=1.
		\]
		Thus we have
		\begin{eqnarray*}
			B_d(1-\frac{h}{2}e^{-\varepsilon_0 d})&\subset& \ho^0_{\mathrm{Ar}}(\X,\Li^{\otimes d})+D_{d,1}.
		\end{eqnarray*}
		So similarly we have
		\[
			\mathrm{Vol} \left( B_d(1-\frac{h}{2}e^{-\varepsilon_0 d}) \right)\leq \#\ho^0_{\mathrm{Ar}}(\X,\Li^{\otimes d})\cdot \mathrm{Vol}(D_{d,1}).
		\]

		Note that for any $R>0$, 
		\[
			\mathrm{Vol}\left( B_d(R)\right)=R^h\mathrm{Vol}\left( B_d(1)\right),
		\]
		and 
		\[
			 \mathrm{Vol}(D_{d,N})=N^{h} \mathrm{Vol}(D_{d,1}).
		\]
		Hence
		\begin{eqnarray*}
			\frac{\mathrm{Vol} \left( B_d(1+Nhe^{-\varepsilon_0 d}) \right)}{\mathrm{Vol}(D_{d,N})}
			&=&\left(\frac{1+Nhe^{-\varepsilon_0 d}}{1-\frac{h}{2}e^{-\varepsilon_0 d}}\right)^h\frac{\mathrm{Vol} \left( B_d(1-\frac{h}{2}e^{-\varepsilon_0 d}) \right)}{N^h\cdot\mathrm{Vol}(D_{d,1})} \\
			&\leq&\left(\frac{1+Nhe^{-\varepsilon_0 d}}{1-\frac{h}{2}e^{-\varepsilon_0 d}}\right)^h\frac{\#\ho^0_{\mathrm{Ar}}(\X,\Li^{\otimes d})\cdot \mathrm{Vol}(D_{d,1})}{N^{h}\cdot\mathrm{Vol}(D_{d,1})}\\
			&=&N^{-h}\left(\frac{1+Nhe^{-\varepsilon_0 d}}{1-\frac{h}{2}e^{-\varepsilon_0 d}}\right)^h\#\ho^0_{\mathrm{Ar}}(\X,\Li^{\otimes d}).
		\end{eqnarray*}
		Since $N\leq e^{\delta d}$, we have
		\begin{eqnarray*}
			(1+Nhe^{-\varepsilon_0 d})^h &=&\exp\left( Nh^2e^{-\varepsilon_0 d}+ O(N^2h^3e^{-2\varepsilon_0 d})\right)\\
			&=& 1+Nh^2e^{-\varepsilon_0 d}+O(N^2h^4e^{-2\varepsilon_0 d}) \\
			&\leq& 1+h^2e^{(\delta-\varepsilon_0) d}+O(h^4e^{2(\delta-\varepsilon_0) d}).
		\end{eqnarray*}
		As $\delta<\varepsilon_0$ and that the rank $h=\mathrm{rk}(\ho^0(\X,\Li^{\otimes d}))$ grows polynomially with $d$, when $d$ is sufficiently large, 
		\[
			(1+Nhe^{-\varepsilon_0 d})^h \leq 1+2h^2e^{(\delta-\varepsilon_0) d}\leq 2.
		\]
		Similarly, we have when $d$ is sufficiently large, 
		\[
			(1-\frac{h}{2}e^{-\varepsilon_0 d})^h =1-\frac{h^2}{2}e^{-\varepsilon_0 d}+O(h^4e^{-2\varepsilon_0 d})\geq 1-h^2e^{-\varepsilon_0 d}\geq \frac{1}{2}.
		\]
		Hence we have
		\[
			\left(\frac{1+Nhe^{-\varepsilon_0 d}}{1-\frac{h}{2}e^{-\varepsilon_0 d}}\right)^h \leq \frac{1+2h^2e^{(\delta-\varepsilon_0) d}}{1-h^2e^{-\varepsilon_0 d}}\leq 4
		\]
		for any $N\leq e^{\delta d}$. Therefore, we can bound $\# \left( B_d(1+\frac{Nh}{2}e^{-\varepsilon_0 d}) \cap N\cdot \ho^0(\X,\Li^{\otimes d}) \right)$ by 
		\begin{eqnarray*}
			& &\# \left( B_d(1+\frac{Nh}{2}e^{-\varepsilon_0 d}) \cap N\cdot \ho^0(\X,\Li^{\otimes d}) \right)\\
			&\leq&  \frac{\mathrm{Vol} \left( B_d(1+Nhe^{-\varepsilon_0 d}) \right)}{\mathrm{Vol}(D_{d,N})} \\
			&\leq& N^{-h}\left(\frac{1+Nhe^{-\varepsilon_0 d}}{1-\frac{h}{2}e^{-\varepsilon_0 d}}\right)^h\#\ho^0_{\mathrm{Ar}}(\X,\Li^{\otimes d}) \\
			&\leq& 4N^{-h}\cdot \#\ho^0_{\mathrm{Ar}}(\X,\Li^{\otimes d}),
		\end{eqnarray*}
		So finally we have
		\begin{eqnarray*}
			\#\psi_{d,N}^{-1}(E)&\leq& \#\left( B_d(1+\frac{Nh}{2}e^{-\varepsilon_0 d}) \cap N\cdot \ho^0(\X,\Li^{\otimes d}) \right)\cdot \#E\\
			&\leq& 4N^{-h}\cdot \left(\#\ho^0_{\mathrm{Ar}}(\X,\Li^{\otimes d})\right)\cdot \#E.
		\end{eqnarray*}
		Note that by Lemma \ref{freemod}, when $d$ is large enough, $\ho^0(\X_N,\li^{\otimes d})\simeq\ho^0(\X,\li^{\otimes d})/\big(N\cdot \ho^0(\X,\li^{\otimes d})\big)$. So for such $d$ we have $\#\ho^0(\X_{N},\li^{\otimes d})=N^{h}$. Hence
		\[
			\frac{\#\psi_{d,N}^{-1}(E)}{\#\ho^0_{\mathrm{Ar}}(\X,\Li^{\otimes d})}\leq 4\frac{\#E}{\#\ho^0(\X_{N},\li^{\otimes d})}
		\]
		and we conclude.
\end{proof}

\begin{rmq}
	When $d$ is large enough, $\ho^0 (\X,\Li^{\otimes d})$ is a free $\Z$-module such that
	\[
		h=\mathrm{rk}\big(\ho^0(\X,\Li^{\otimes d})\big)=\dim_{\Q}\ho^0(\X_{\Q},\li^{\otimes d})=\chi(\X_{\Q}, \li^{\otimes d})
	\]
	as $\Li$ is ample.

	The asymptotic Riemann-Roch Theorem tells us then that
	\[
		h=\frac{\big((\li|_{\X_{\Q}})^{n-1} \big)}{(n-1)!} d^{n-1}+O(d^{n-2}),
	\]
	where $\big((\li|_{\X_{\Q}})^{n-1} \big)$ is the intersection number of $n-1$ copies of $\li |_{\X_\Q}$, where as 
	\[
		h^0_{\mathrm{Ar}}(\X,\Li^{\otimes d}\otimes \overline{\mathcal{M}})= \frac{r}{n!}\Li^{n} d^n+O(d^{n-1}\log d)
	\]
	by Theorem \ref{hilbsamu}.
\end{rmq}

\section{Convergence of special values of zeta functions}\label{c0}
Let $\X$ be a separated scheme of finite type and flat over $\Z$ of absolute dimension $n$. We fix from now on a constant $c_0>0$ such that for any prime integer $p$ and any $e\in \Z_{>0}$,
\[
	\#\X(\F_{p^e})=\#\X_{p}(\F_{p^e})\leq c_0p^{(n-1)e},
\]
where $n$ is the absolute dimension of $\X$ (so $\X_p$ is of dimension $n-1$). Such a constant exists by the Lang-Weil estimates in \cite{LW54}. A good introduction to the function $\#\X(\F_{p^e})$ and its properties is Serre's book \cite{Se12}. 

We know that the zeta function
\[
	\zeta_{\X}(s)= \prod_{x\in |\X|}(1-\#\kappa(x)^{-s})^{-1}
\]
is absolutely convergent for any $s\in \bbC$ satisfying $\mathrm{Re}(s)>n$. Moreover, the zeta function of $\X$ is the product of the zeta function of all its fibers, i.e. we have
\[
	\zeta_{\X}(s)=\prod_{p\text{ prime}} \zeta_{\X_p}(s).
\]
For later use, we calculate in this section the speed of convergence of
\[
	\prod_{x\in |\X_p|,\ \deg x\leq r}(1-\#\kappa(x)^{-(n+1)})
\]
to $\zeta_{\X_p}(n+1)^{-1}$ when $r\rightarrow \infty$, and that of $\prod_{p\leq R} \zeta_{\X_p}(n+1)^{-1}$ to $\zeta_{\X}(n+1)^{-1}$ when $R\rightarrow \infty$.

\begin{lem}\label{boundlog}
	For any prime number $p$ and any positive integer $e\geq 1$, we have
	\[
		- \log( 1-p^{-e} )<2p^{-e}.
	\]
	In particular, for any closed point $x$ on an arithmetic variety $\X$ and any integer $e\geq 1$, we have
	\[
		- \log(1-\#\kappa(x)^{-e})<2\cdot\#\kappa(x)^{-e}.
	\]
\end{lem}
\begin{proof}
	For any real number $0<t\leq \frac{\sqrt{5}-1}{2}$, we have $-\log(1-t)\leq 2\log(1+t)$. Indeed, when $0<t\leq \frac{\sqrt{5}-1}{2}$, we have
	\[
		-1<t^2+t-1=(t+\frac{1}{2})^2-\frac{5}{4}\leq 0.
	\]
	Then
	\[
		(1+t)^2(1-t)=1-(t^3+t^2-t)=1-t(t^2+t-1)\geq 1,
	\]
	which implies 
	\[
		\frac{1}{1-t}\leq (1+t)^2,
	\]
	i.e.
	\[
		-\log(1-t)\leq 2\log(1+t).
	\]
	Since for any $t>0$, $\log(1+t)<t$, we have for $0<t\leq \frac{\sqrt{5}-1}{2}$,
	\[
		-\log(1-t)\leq 2t.
	\]
	As $\frac{\sqrt{5}-1}{2}>\frac{1}{2}$, any prime $p$ and positive integer $r\geq 1$ satisfy $p^{-r}< \frac{\sqrt{5}-1}{2}$. Hence we conclude.
\end{proof}

\begin{lem}\label{zetafinite}
	Let $\X$ be an arithmetic scheme of absolute dimension $n$. For any prime number $p$ and any positive integer $r\geq 1$, we have
	\[
		 \left| \prod_{x\in |\X_{p}|, \deg x\leq r} \left( 1-p^{-(n+1)\deg x} \right) - \zeta_{\X_p}(n+1)^{-1} \right|\leq 4c_0 p^{-2(r+1)}.
	\]
\end{lem}
\begin{proof}
	By Lemma \ref{boundlog}, for any closed point $x$ of $\X$,
	\[
		- \log(1-\#\kappa(x)^{-(n+1)})<2\cdot\#\kappa(x)^{-(n+1)}
	\]
	We have
	\begin{eqnarray*}
		\sum_{x\in |\X_{p}|, \deg x> r}\left(- \log( 1-p^{-(n+1)\deg x} )\right) &<& 2\sum_{x\in |\X_{p}|, \deg x> r} p^{-(n+1)\deg x} \\
		&\leq & 2\sum_{e=r+1}^{\infty}\#\X(\F_{p^e}) p^{-(n+1)e} \\
		&\leq& 2\sum_{e=r+1}^{\infty}c_0p^{(n-1)e}\cdot p^{-(n+1)e} \\
		&=& 2c_0\sum_{e=r+1}^{\infty} p^{-2e}\\
		&=& 2c_0\frac{p^{-2(r+1)}}{1-p^{-2}}\\
		&<& 4c_0p^{-2(r+1)}.
	\end{eqnarray*}
	On the other hand, for any $x\in |\X_{p}|$,
	\[
		 \left( 1-p^{-(n+1)\deg x} \right)<1.
	\]
	Hence
	\begin{eqnarray*}
		& &\left| \prod_{x\in |\X_{p}|, \deg x\leq r} \left( 1-p^{-(n+1)\deg x} \right) - \zeta_{\X_p}(n+1)^{-1} \right| \\
		&=& \prod_{x\in |\X_{p}|, \deg x\leq r} \left( 1-p^{-(n+1)\deg x} \right)\cdot \left( 1- \prod_{x\in |\X_{p}|, \deg x> r} \left( 1-p^{-(n+1)\deg x} \right)\right) \\
		&<&1- \prod_{x\in |\X_{p}|, \deg x> r} \left( 1-p^{-(n+1)\deg x} \right) \\
		&=&1-\exp\left( \sum_{x\in |\X_{p}|, \deg x> r} \log\left( 1-p^{-(n+1)\deg x} \right)\right).
	\end{eqnarray*}
	By the above computation, we have
	\begin{eqnarray*}
		& &\left| \prod_{x\in |\X_{p}|, \deg x\leq r} \left( 1-p^{-(n+1)\deg x} \right) - \zeta_{\X_p}(n+1)^{-1} \right| \\
		&<& 1-\exp\left( -4c_0p^{-2(r+1)} \right) \\
		&<& 4c_0p^{-2(r+1)}
	\end{eqnarray*}
	as for any $t>0$, $e^{-t}>1-t$. Therefore we conclude. 
\end{proof}

\begin{lem}\label{boundXp}
	Let $\X$ be an arithmetic scheme of absolute dimension $n$. For any prime number $p$, we have
	\[
		0<\log\zeta_{\X_p}(n+1)\leq 4c_0p^{-2}.
	\]
\end{lem}
\begin{proof}
	In fact, we have
	\begin{eqnarray*}
		0<\log\zeta_{\X_p}(n+1) &=&\log\left( \prod_{x\in |\X_p|}(1-\#\kappa(x)^{-(n+1)})^{-1} \right) \\
		&=&(-1) \sum_{x\in |\X_p|} \log \left(1-\#\kappa(x)^{-(n+1)}  \right) \\
		&<&\left( \sum_{x\in |\X_p|} 2\#\kappa(x)^{-(n+1)}  \right) \\
		&<& 2\sum_{e=1}^{\infty}\#\X(\mathbb{F}_{p^e})\cdot p^{-e(n+1)},
	\end{eqnarray*}
	where the third line uses Lemma \ref{boundlog}.
	By the choice of $c_0$ at the beginning of this section, for any $e$, 
	\[
		\#\X(\F_{p^e})=\#\X_{p}(\F_{p^e})\leq c_0p^{(n-1)e}.
	\]
	Then we have
	\begin{eqnarray*}
		\sum_{e=1}^{\infty}\#\X(\mathbb{F}_{p^e})\cdot p^{-e(n+1)} &\leq& c_0\sum_{e=1}^{\infty} p^{-2e}\\
		&\leq& \frac{c_0p^{-2}}{1-p^{-2}}<2c_0p^{-2}.
	\end{eqnarray*}
	Hence we conclude.
\end{proof}

\begin{lem}\label{zetaintegral}
	When $R\in \Z_{>0}$ is large enough, we have 
	\[
		\left| \prod_{p\leq R}\zeta_{\X_p}(n+1)^{-1}-\zeta_\X(n+1)^{-1} \right| < 8c_0 \zeta_\X(n+1)^{-1}\cdot R^{-1}. 
	\]
\end{lem}
\begin{proof}
	Since $\zeta_{\X}(s)=\prod_p\zeta_{\X_p}(s)$, for a positive integer $R$ we have
	\begin{eqnarray*}
		\left| \prod_{p\leq R}\zeta_{\X_p}(n+1)^{-1}-\zeta_\X(n+1)^{-1} \right|& =& \zeta_\X(n+1)^{-1}\cdot \left| \prod_{p>R}\zeta_{\X_p}(n+1)-1 \right|\\
		&=&  \zeta_\X(n+1)^{-1}\cdot \left| \exp\left(\sum_{p>R}\log\zeta_{\X_p}(n+1)\right)-1 \right|.
	\end{eqnarray*}
	Since by Lemma \ref{boundXp} $0<\log\zeta_{\X_p}(n+1)\leq 4c_0p^{-2}$, we have
	\begin{eqnarray*}
		0<\sum_{p>R}\log\zeta_{\X_p}(n+1) &<& 4c_0\sum_{p>R}p^{-2}\\
		&<&4c_0\sum_{k>R}k^{-2} \\
		&<& 4c_0\int_{R}^{\infty}x^{-2}\mathrm{d}x \\
		& =& 4c_0 R^{-1}.
	\end{eqnarray*}
	When $t\in \R$ is sufficiently small, we have $e^t-1<2t$. Therefore when $R$ is sufficiently large, 
	\begin{eqnarray*}
		\left| \prod_{p\leq R}\zeta_{\X_p}(n+1)^{-1}-\zeta_\X(n+1)^{-1} \right|	&=&  \zeta_\X(n+1)^{-1}\cdot \left| \exp\left(\sum_{p>R}\log\zeta_{\X_p}(n+1)\right)-1 \right|\\
		&<& \zeta_\X(n+1)^{-1}\cdot \left| \exp\left(4c_0 R^{-1}\right)-1 \right|\\
		&<&8c_0R^{-1}\zeta_\X(n+1)^{-1}.
	\end{eqnarray*}
\end{proof}

\section{Effective computations on a single fiber}\label{effective}

In this section, for a regular projective arithmetic variety $\X$ of dimension $n$ equipped with an ample Hermitian line bundle $\Li$, we calculate the density of the set of global sections in $\ho^0_{\mathrm{Ar}}(\X, \Li^{\otimes d})$ whose divisor has no singular point lying on a fiber $\X_p$ for a fixed prime integer $p$ when $d\rightarrow \infty$. Note that this density differs from the density of sections in $\ho^0(\X_p, \li^{\otimes d})$ whose divisor is smooth over $\F_p$. This is because when a global section $\sigma \in \ho^0(\X, \Li^{\otimes d})$ is such that $\mathrm{div}\sigma$ has no singular point on $\X_p$, it is still possible that its image by restriction map $\phi_{d,p}: \ho^0(\X, \Li^{\otimes d})\rightarrow \ho^0(\X_p, \li^{\otimes d})$ is such that $\mathrm{div}\phi_{d,p}(\sigma)$ is singular.

Indeed, let $x$ be a closed point on the fiber $\X_p$ with maximal ideal $\mathfrak{m}_x$ as a closed subscheme of $\X$. We may assume that $\X_p$ is smooth over $\F_p$. The maximal ideal of $x$ as a closed point of $\X_p$ is $\mathfrak{m}_{\X_p,x}=\mathfrak{m}_x/(p\cdot\mathcal{O}_x)$. For any $\sigma\in \ho^0(\X,\Li^{\otimes d})$, its divisor $\mathrm{div}\sigma$ is singular at $x$ if and only if $\sigma$ is contained in $\ho^0(\X, \Li^{\otimes d}\otimes \mathfrak{m}_{x}^2)$, where we identify $\ho^0(\X, \Li^{\otimes d}\otimes \mathfrak{m}_{x}^2)$ with a sub-$\Z$-module of $\ho^0(\X,\Li^{\otimes d})$ by regarding $\Li^{\otimes d}\otimes  \mathfrak{m}_{x}^2$ as a subsheaf of $\Li^{\otimes d}$. This is equivalent to the condition that, denoting by $x'$ the closed subscheme of $\X$ defined by the ideal sheaf $\mathfrak{m}_x^2$, the image of $\sigma$ by the restriction map $\ho^0(\X, \Li^{\otimes d})\rightarrow \ho^0(x', \li^{\otimes d})$ is $0$. Similarly, denoting by $x''$ the closed subscheme of $\X_p$ defined by $\mathfrak{m}_{\X_p,x}^2$, $\mathrm{div}\phi_{d,p}(\sigma)$ is singular at $x$ if and only if the image of $\sigma$ by the restriction map $\ho^0(\X, \Li^{\otimes d})\longrightarrow \ho^0(x'', \li^{\otimes d})$ is $0$. Note that as $x$ is a regular point of $\X$,
\[
	\#\ho^0(x', \li^{\otimes d})=\#\ho^0(x', \mathcal{O}_{x'})=\big(\#\kappa(x)\big)^{1+n}.
\]
Similarly, since $\X_p$ is smooth,
\[
	\#\ho^0(x'', \li^{\otimes d})=\#\ho^0(x'', \mathcal{O}_{x'})=\big(\#\kappa(x)\big)^{1+(n-1)}=\big(\#\kappa(x)\big)^n.
\]
Moreover, by the definition of $x'$ and $x''$, the restriction map $\ho^0(\X, \Li^{\otimes d})\longrightarrow \ho^0(x'', \li^{\otimes d})$ factors through
\[
	\#\ho^0(x', \li^{\otimes d})\longrightarrow \ho^0(x'', \li^{\otimes d}).
\]
Therefore we have a strict inclusion 
\[
	\Ker\left( \ho^0(\X, \Li^{\otimes d})\longrightarrow \ho^0(x', \li^{\otimes d}) \right) \subsetneq \Ker\left( \ho^0(\X, \Li^{\otimes d})\longrightarrow \ho^0(x'', \li^{\otimes d})\right)
\]
which implies that it is possible that $\mathrm{div}\phi_{d,p}(\sigma)$ is singular at $x$ while $\mathrm{div}\sigma$ is regular at $x$.

\begin{expl}
	Consider $\pr^2_{\Z}$ together with the ample line bundle $\mathcal{O}(1)$ on it. Then $X^2+5Y^2-Z^2$ is a global section in $\ho^0(\pr^2_{\Z}, \mathcal{O}(2))$. The restriction $\phi_{2,5}(X^2+5Y^2-Z^2)$ in $\ho^0(\pr^2_{\F_5}, \mathcal{O}(2))$ is equal to $X^2-Z^2$. So $\mathrm{div}(\phi_{2,5}(X^2+5Y^2-Z^2))$ has a singular point $P=\overline{[0,1,0]}\in \pr^2_{\F_5}$. But $P$ is not a singular point of $\mathrm{div}(X^2+5Y^2-Z^2)$. Indeed, consider the open affine neighbourhood $\A^2_{\Z}=\pr^2_{\Z}-\mathrm{div}(Y)$ of $P$. The ideal sheaf $\mathfrak{m}_{\A^2_{\Z},P}$ of $P$ in $\A^2_{\Z}$ is generated by $\frac{X}{Y},\frac{Z}{Y},5\in \ho^0(\A^2_{\Z}, \mathcal{O}_{\A^2_{\Z}})$. Then $\mathfrak{m}^2_{\A^2_{\Z},P}$ is generated by $\frac{X^2}{Y^2}, \frac{Z^2}{Y^2}, 25, \frac{XZ}{Y^2}, \frac{5X}{Y},\frac{5Z}{Y}$. Let $P'$ be the first order infinitesimal neighbourhood of $P$ in $\pr^2_{\Z}$. Then $P'$ can be regarded as a closed subscheme of $\A^2_{\Z}$ defined by $\mathfrak{m}^2_{\A^2_{\Z},P}$. Note that
	\begin{eqnarray*}
		X^2+5Y^2-Z^2 = (\frac{X^2}{Y^2}-\frac{Z^2}{Y^2}+5)Y^2.
	\end{eqnarray*}
	The image of $X^2+5Y^2-Z^2$ in $\ho^0(P', \mathcal{O}(2))$ is $5\cdot Y^2$, which is non-zero. So $P$ is a singular point of $\mathrm{div}(\phi_{2,5}(X^2+5Y^2-Z^2))$, but it is not a singular point of $\mathrm{div}(X^2+5Y^2-Z^2)$.
\end{expl}

The above argument shows that for a section $\sigma \in \ho^0(\X,\Li^{\otimes d})$, the condition that the divisor $\mathrm{div}\sigma$ has a singular point on the fiber $\X_p$ is stronger than that $\mathrm{div}\phi_{d,p}(\sigma)$ is singular. However, we can test whether $\mathrm{div}\sigma$ has a singular point on the fiber $\X_p$ through the restriction $\phi_{d,p^2}(\sigma)$, where $\phi_{d,p^2}:  \ho^0(\X, \Li^{\otimes d})\rightarrow \ho^0(\X_{p^2}, \li^{\otimes d})$ is the restriction morphism to $\X_{p^2}$. In fact, for a closed point $x$ on the fiber $\X_p$, as $p\cdot\mathcal{O}_x$ is contained in $\mathfrak{m}_x$, we have naturally $p^2\cdot\mathcal{O}_x\in \mathfrak{m}^2_x$ and that $x'$ is a closed subscheme of $\X_{p^2}$. This implies that the image of $\sigma$ in $\ho^0(x', \li^{\otimes d})$ is the same as the image of $\phi_{d,p^2}(\sigma)$ via $\ho^0(\X_{p^2}, \li^{\otimes d})\longrightarrow \ho^0(x', \li^{\otimes d})$, and also that when $x\in \mathrm{div}\sigma$,
\[
	\frac{\m_{\mathrm{div}\phi_{d,p^2}(\sigma),x}}{\m^2_{\mathrm{div}\phi_{d,p^2}(\sigma), x}}=\frac{\m_{\mathrm{div}\sigma,x}/p^2\mathcal{O}_x}{\m^2_{\mathrm{div}\sigma, x}/p^2\mathcal{O}_x}\simeq\frac{\m_{\mathrm{div}\sigma,x}}{\m^2_{\mathrm{div}\sigma, x}}.
\]
So when $x\in\mathrm{div} \sigma$, $\mathrm{div} \sigma$ is regular at $x$ if and only if the image of $\phi_{d,p^2}(\sigma)$ in $\ho^0(x', \li^{\otimes d})$ is not $0$, and this condition is equivalent to the condition that 
\[
	\dim_{\kappa(x)}\frac{\m_{\mathrm{div}\phi_{d,p^2}(\sigma),x}}{\m^2_{\mathrm{div}\phi_{d,p^2}(\sigma), x}}=\dim_{\kappa(x)}\frac{\m_{\mathrm{div}\sigma,x}}{\m^2_{\mathrm{div}\sigma, x}}=n-1.
\]
Therefore we may study whether $\mathrm{div}\sigma$ has a singular point on the fiber $\X_p$ via the study of $\mathrm{div}\phi_{d,p^2}(\sigma)$, although the latter is not regular itself.

\subsection{Main results}
We write
\[
	\mathcal{P}_{d,p}:=\left\{ \sigma\in \ho^0(\X, \Li^{\otimes d})\ ;\ \mathrm{div}\ \sigma \text{ has no singular point on } \X_p \right\}.
\]

\begin{thm}\label{fiber}
	Let $\X$ be a regular projective arithmetic variety of absolute dimension $n$, and let $\Li$ be an ample Hermitian line bundle on $\X$. 
	There exists a constant $C>1$ such that for any large enough integer $d$ and any prime number $p$ satisfying $Cnp^n<d$, we have
	\[
		\left| \frac{\#\left(\mathcal{P}_{d,p} \cap \ho^0_{\mathrm{Ar}}(\X, \overline{\mathcal{L}}^{\otimes d})\right)}{\#\ho^0_{\mathrm{Ar}}(\X, \overline{\mathcal{L}}^{\otimes d})} -\zeta_{\X_p}(n+1)^{-1} \right|=O\left( \Big(\frac{d}{p}\Big)^{-\frac{2}{n}} \right),
	\]
	where the constant involved in big $O$ is independent of $d,p$.
\end{thm}

To prove this result, it suffices to prove the following proposition :

\begin{prop}\label{fib}
	Define
	\[
		\mathcal{P}'_{d,p^2}:=\left\{ \sigma'\in \ho^0(\X_{p^2}, \li^{\otimes d})\ ;\ \forall x\in |\mathrm{div}\sigma'|,\ \dim_{\kappa(x)}\frac{\m_{\mathrm{div}\sigma',x}}{\m^2_{\mathrm{div}\sigma', x}}=n-1 \right\}.
	\]
	Then there exists a constant $C>1$ such that for any prime number $p$ satisfying $Cnp^n<d$, we have
	\[
		\left| \frac{\#\mathcal{P'}_{d,p^2}}{\#\ho^0(\X_{p^2}, \li^{\otimes d})} -\zeta_{\X_p}(n+1)^{-1} \right|=O\left( \Big(\frac{d}{p}\Big)^{-\frac{2}{n}} \right),
	\] 
	where the constant involved in big $O$ is independent of $d,p$.
\end{prop}

\begin{proof}[Proof of Theorem \ref{fiber}]
	Assuming this proposition, by Proposition \ref{modN} for any $\alpha_0$ such that $0<\alpha_0<1$ we can find a constant $\eta>0$ such that when $d$ is large enough, for any prime number $p$ such that $p^2\leq \exp(d^{\alpha_0})$ and any $\sigma'_1,\sigma'_2\in \ho^0(\X_{p^2}, \li^{\otimes d})$, 
	\begin{eqnarray*}
		\frac{\left| \#\left(\phi_{d,p^2}^{-1}(\sigma'_1)\cap \ho^0_{\mathrm{Ar}}(\X, \Li^{\otimes d}) \right)-\#\left(\phi_{d,p^2}^{-1}(\sigma'_2)\cap \ho^0_{\mathrm{Ar}}(\X, \Li^{\otimes d}) \right)   \right|}{\#\left(\phi_{d,p^2}^{-1}(\sigma'_1)\cap \ho^0_{\mathrm{Ar}}(\X, \Li^{\otimes d}) \right)}  \leq e^{-\eta d} ,
	\end{eqnarray*}
	where $\phi_{d,p^2} :  \ho^0(\X, \Li^{\otimes d}) \longrightarrow \ho^0(\X_{p^2}, \li^{\otimes d})$ is the restriction map.
	When we take the sum over all $\sigma'_2\in \ho^0(\X_{p^2}, \li^{\otimes d})$, as 
	\[
		\bigcup_{\sigma'_2\in \ho^0(\X_{p^2}, \li^{\otimes d})} \left(\phi_{d,p^2}^{-1}(\sigma'_2)\cap \ho^0_{\mathrm{Ar}}(\X, \Li^{\otimes d}) \right)= \ho^0_{\mathrm{Ar}}(\X, \Li^{\otimes d}),
	\]
	we get 
	\begin{eqnarray*}
		& &\qquad\left| \#\left(\phi_{d,p^2}^{-1}(\sigma'_1)\cap \ho^0_{\mathrm{Ar}}(\X, \Li^{\otimes d}) \right)\cdot \left( \# \ho^0(\X_{p^2}, \li^{\otimes d}) \right)-\# \ho^0_{\mathrm{Ar}}(\X, \Li^{\otimes d})   \right| \\
		&\leq& \sum_{\sigma'_2\in  \ho^0(\X_{p^2}, \li^{\otimes d})}\left| \#\left(\phi_{d,p^2}^{-1}(\sigma'_1)\cap \ho^0_{\mathrm{Ar}}(\X, \Li^{\otimes d}) \right)-\#\left(\phi_{d,p^2}^{-1}(\sigma'_2)\cap \ho^0_{\mathrm{Ar}}(\X, \Li^{\otimes d}) \right)   \right|\\
		&\leq&  \sum_{\sigma'_2\in  \ho^0(\X_{p^2}, \li^{\otimes d})} e^{-\eta d}\cdot \#\left(\phi_{d,p^2}^{-1}(\sigma'_1)\cap \ho^0_{\mathrm{Ar}}(\X, \Li^{\otimes d}) \right)\\
		&=& e^{-\eta d}\cdot \#\left(\phi_{d,p^2}^{-1}(\sigma'_1)\cap \ho^0_{\mathrm{Ar}}(\X, \Li^{\otimes d}) \right)\cdot \#\ho^0(\X_{p^2}, \li^{\otimes d}).
	\end{eqnarray*}
	Dividing both side of the inequality by $\#\ho^0(\X_{p^2}, \li^{\otimes d})$, we get
	\[
		\left| \#\left(\phi_{d,p^2}^{-1}(\sigma'_1)\cap \ho^0_{\mathrm{Ar}}(\X, \Li^{\otimes d}) \right)-\frac{\#\ho^0_{\mathrm{Ar}}(\X, \Li^{\otimes d})}{\#\ho^0(\X_{p^2}, \li^{\otimes d})}   \right| \leq e^{-\eta d}\cdot  \#\left(\phi_{d,p^2}^{-1}(\sigma'_1)\cap \ho^0_{\mathrm{Ar}}(\X, \li^{\otimes d}) \right).
	\]
	$ $

	Note that for any $\sigma\in \ho^0(\X, \Li^{\otimes d})$, $\sigma \in \mathcal{P}_{d,p}$ if and only if for any closed point $x\in \mathrm{div}\sigma\cap \X_p$, we have $\dim_{\kappa(x)}\frac{\m_{\mathrm{div}\sigma,x}}{\m^2_{\mathrm{div}\sigma, x}}=n-1$. This is equivalent to the condition that for any closed point $x\in \mathrm{div}\phi_{d,p^2}(\sigma)$, we have $\dim_{\kappa(x)}\frac{\m_{\mathrm{div}\phi_{d,p^2}(\sigma),x}}{\m^2_{\mathrm{div}\phi_{d,p^2}(\sigma), x}}=n-1$, i.e. $\phi_{d,p^2}(\sigma)\in \mathcal{P}'_{d,p^2}$. Hence $\mathcal{P}_{d,p}\cap \ho^0_{\mathrm{Ar}}(\X, \Li^{\otimes d})$ is exactly the preimage of $\mathcal{P'}_{d,p^2}$ in $\ho^0_{\mathrm{Ar}}(\X, \Li^{\otimes d})$. Summing up over all $\sigma'_1\in \mathcal{P'}_{d,p^2}$, we get 
	\begin{eqnarray*}
		& &\left| \#\left(\mathcal{P}_{d,p}\cap \ho^0_{\mathrm{Ar}}(\X, \Li^{\otimes d})\right)  -\frac{\#\ho^0_{\mathrm{Ar}}(\X, \Li^{\otimes d})}{\#\ho^0(\X_{p^2}, \li^{\otimes d})}\cdot\#\mathcal{P'}_{d,p^2}    \right| \\
		&=&\left| \#\bigcup_{\sigma'_1\in \mathcal{P'}_{d,p^2}}\left(\phi_{d,p^2}^{-1}(\sigma'_1)\cap \ho^0_{\mathrm{Ar}}(\X, \Li^{\otimes d}) \right)  -\frac{\#\ho^0_{\mathrm{Ar}}(\X, \Li^{\otimes d})}{\#\ho^0(\X_{p^2}, \li^{\otimes d})} \cdot \#\mathcal{P'}_{d,p^2}  \right| \\
		&\leq& \sum_{\sigma'_1\in \mathcal{P'}_{d,p^2}}  \left| \#\left(\phi_{d,p^2}^{-1}(\sigma'_1)\cap \ho^0_{\mathrm{Ar}}(\X, \Li^{\otimes d}) \right)-\frac{\#\ho^0_{\mathrm{Ar}}(\X, \Li^{\otimes d})}{\#\ho^0(\X_{p^2}, \li^{\otimes d})}   \right| \\
		&\leq& \sum_{\sigma'_1\in \mathcal{P'}_{d,p^2}} e^{-\eta d}\cdot  \#\left(\phi_{d,p^2}^{-1}(\sigma'_1)\cap \ho^0_{\mathrm{Ar}}(\X, \Li^{\otimes d}) \right) \\
		&=& e^{-\eta d}\cdot \#\left(\mathcal{P}_{d,p}\cap \ho^0_{\mathrm{Ar}}(\X, \Li^{\otimes d})\right)\leq e^{-\eta d}\cdot\#\ho^0_{\mathrm{Ar}}(\X, \Li^{\otimes d}).
	\end{eqnarray*}
	Dividing this inequality by $\#\ho^0_{\mathrm{Ar}}(\X, \Li^{\otimes d})$, we get 
	\[
		\left| \frac{\#\left(\mathcal{P}_{d,p}\cap \ho^0_{\mathrm{Ar}}(\X, \Li^{\otimes d})\right)}{\#\ho^0_{\mathrm{Ar}}(\X, \Li^{\otimes d})}  -\frac{\#\mathcal{P'}_{d,p^2}   }{\#\ho^0(\X_{p^2}, \li^{\otimes d})} \right|\leq e^{-\eta d}.
	\]
	Therefore assuming Proposition \ref{fib}, we have 
	\begin{eqnarray*}
		& &\left| \frac{\#\left(\mathcal{P}_{d,p}\cap \ho^0_{\mathrm{Ar}}(\X, \Li^{\otimes d})\right)}{\#\ho^0_{\mathrm{Ar}}(\X, \Li^{\otimes d})} -\zeta_{\X_p}(n+1)^{-1} \right| \\
		&\leq& \left| \frac{\#\left(\mathcal{P}_{d,p}\cap \ho^0_{\mathrm{Ar}}(\X, \Li^{\otimes d})\right)}{\#\ho^0_{\mathrm{Ar}}(\X, \Li^{\otimes d})}  -\frac{\#\mathcal{P'}_{d,p^2}   }{\#\ho^0(\X_{p^2}, \li^{\otimes d})} \right|+\left| \frac{\#\mathcal{P'}_{d,p^2}}{\#\ho^0(\X_{p^2}, \li^{\otimes d})} -\zeta_{\X_p}(n+1)^{-1} \right| \\
		&=&O(e^{-\eta d})+O\left( \Big(\frac{d}{p}\Big)^{-\frac{2}{n}} \right)=O\left( \Big(\frac{d}{p}\Big)^{-\frac{2}{n}} \right).
	\end{eqnarray*}
\end{proof}

The proof of Proposition \ref{fib} follows the method of Poonen for his proof of the Bertini theorem over finite fields in \cite{Po04}. 

We will prove Proposition \ref{fib} through the following steps.
\begin{enumerate}
	\item In Section \ref{step1} we will calculate the proportion of $\sigma\in \ho^0(\X_{p^2}, \li^{\otimes d})$ such that $ \dim_{\kappa(x)}\frac{\m_{\mathrm{div}\sigma,x}}{\m^2_{\mathrm{div}\sigma, x}}=n-1$ for any closed point $x$ of degree $\leq r$ for an integer $r$. This proportion equals to 
	\[
		\prod_{x\in |\X|,\ \deg x\leq r}(1-\#\kappa(x)^{-(1+n)})
	\]
	for $r$ not too big. We will give a bound $r_d$ for $r$ depending on $d$ where this proportion is valid for any $r$ such that $0<r\leq r_d$.
	\item Then in Section \ref{step2}, we will show that for some integer constant $N$, the proportion of $\sigma\in \ho^0(\X_{p^2}, \li^{\otimes d})$ such that there exists a closed point $x$ of degree between $r_d$ and $\frac{d}{nN}$ where the condition $ \dim_{\kappa(x)}\frac{\m_{\mathrm{div}\sigma,x}}{\m^2_{\mathrm{div}\sigma, x}}=n-1$ is not satisfied tends to $0$ when $d$ tends to infinity. 
	\item In Section \ref{step3}, we will show the following : there exists a constant $N(p)$ depending on $p$ such that the proportion of $\sigma\in \ho^0(\X_{p^2}, \li^{\otimes d})$ which satisfy the condition that there exists a closed point $x$ of degree strictly larger than $\frac{d}{nN(p)}$ where we have $ \dim_{\kappa(x)}\frac{\m_{\mathrm{div}\sigma,x}}{\m^2_{\mathrm{div}\sigma, x}}\not=n-1$ tends to $0$ when $d$ tends to infinity. 
	\item In Section \ref{step4}, we will put these three estimates together to get an effective estimate of proportion of global sections whose divisor has no singular point on one single fiber.

\end{enumerate}


In the following, we need a relative version of Lemma \ref{ample}:
\begin{lem}\label{relample}
	Let $\li$ be an ample line bundle on a projective scheme $\mathcal{Y}$ flat over an open subscheme $\mathcal{S}=\Spec\ R$ of $\Spec\ \Z$. 
	Then there exists a positive integer $N$ such that
	\begin{enumerate}[label=\roman*)]
		\item $\li^{\otimes d}$ is very ample for all $d\geq N$;
		\item for any $a, b\geq N$, the natural morphism of $R$-modules
		\[
			\ho^0(\mathcal{Y},\li^{\otimes a})\otimes \ho^0(\mathcal{Y},\li^{\otimes b})\longrightarrow \ho^0(\mathcal{Y},\li^{\otimes (a+b)})
		\]
		is surjective.
	\end{enumerate} 
\end{lem}
\begin{proof}
	It suffices to take the integer $N$ such that Lemma \ref{ample} holds for the generic fiber $\mathcal{Y}_{\Q}$ and that $\ho^0(\mathcal{Y},\li^{\otimes d})$ is torsion free for any $d\geq N$.
\end{proof}

\subsection{Singular points of small degree}\label{step1}

We need a lemma:
\begin{lem}\label{sur}
	Let $Z$ be a closed subscheme of $\X_{p^2}$ of dimension $0$, and let $N$ be a positive integer such that $\li^{\otimes N}$ is very ample.
	The restriction morphism
	\[
		\ho^0(\X_{p^2}, \li^{\otimes d})\longrightarrow \ho^0(Z,  \li^{\otimes d}) 
	\]
	is surjective when $d\geq Nh_Z$, where $h_Z=\dim_{\F_p}\left(\ho^0(Z, \mathcal{O}_{Z})\otimes_{\Z/p^2\Z} \F_p\right)$.
\end{lem}
\begin{proof}
	Let $C_d$ be the cokernel of the restriction map. Then $C_d\otimes_{\Z/p^2\Z}\F_{p}$ is the cokernel of 
	\[
		\ho^0(\X_{p}, \li^{\otimes d})\longrightarrow \ho^0(Z_p,  \li^{\otimes d}) =\ho^0(Z,  \li^{\otimes d}) \otimes_{\Z/p^2\Z} \F_p.
	\]
	When $d\geq Nh_Z$, by Lemma \ref{surject}, we have $C_d\otimes_{\Z/p^2\Z}\F_{p}=0$. Then by the short exact sequence
	\[
		0\longrightarrow pC_d \longrightarrow C_d \longrightarrow C_d\otimes_{\Z/p^2\Z}\F_{p} \longrightarrow 0,
	\]
	we get $pC_d=C_d$. Applying Nakayama's lemma to $C_d$, considered as a $\Z/p^2\Z$-module, we get $C_d=0$. Thus the surjectivity of the restriction map in the lemma holds when $d\geq Nh_Z$. 
\end{proof}

\begin{lem}\label{smalldegree}
	Set
	\[
		\mathcal{P}'_{d,p^2,\leq r}:=\left\{ \sigma'\in \ho^0(\X_{p^2}, \li^{\otimes d})\ ;\ \forall x\in |\mathrm{div}\sigma'|\text{ of degree}\leq r, \ \dim_{\kappa(x)}\frac{\m_{\mathrm{div}\sigma',x}}{\m^2_{\mathrm{div}\sigma', x}}=n-1 \right\}.
	\]
	For any positive integer $r$ satisfying $2c_0Nnrp^{(n-1)r}\leq d$, we have 
	\[
		\frac{\#\mathcal{P}'_{d,p^2,\leq r}}{\#\ho^0(\X_{p^2}, \li^{\otimes d})}=\prod_{x\in |\X_{p}|, \deg x\leq r} \left( 1-p^{-(n+1)\deg x} \right).
	\]
\end{lem}
\begin{proof}
	For any closed point $x\in \X_{p^2}$, let $x'$ be the closed subscheme of $\X_{p^2}$ defined by the square of the maximal ideal of $x$. Then $x'$ is the first order infinitesimal neighborhood of $x$ in $\X_{p^2}$. We have $x'\simeq \Spec (\Oc_{\X_{p^2,},x}/\mathfrak{m}_{\X_{p^2,},x}^2)$. A section $\sigma'\in \ho^0(\X_{p^2}, \li^{\otimes d})$ is such that $\mathrm{div}\sigma'$ contains $x$ and that $ \dim_{\kappa(x)}\frac{\m_{\mathrm{div}\sigma',x}}{\m^2_{\mathrm{div}\sigma', x}}=n$ if and only if the restriction map
	\[
		\ho^0(\X_{p^2}, \li^{\otimes d}) \longrightarrow \ho^0(x', \li^{\otimes d})
	\]
	sends $\sigma'$ to $0$.  For a positive integer $r$, let $\X'_{p^2,\leq r}$ be the disjoint union 
$$\X'_{p^2,\leq r}=\coprod_{x\in |\X_{p^2}|, \deg x\leq r}x'.$$ 
	Then we have a natural isomorphism 
	\[
		\ho^0(\X'_{p^2,\leq r}, \li^{\otimes d})\simeq \prod_{x\in |\X_{p^2}|, \deg x\leq r}\ho^0(x', \li^{\otimes d}).
	\]
	A section $\sigma'\in \ho^0(\X_{p^2}, \li^{\otimes d})$ is such that $\mathrm{div}\sigma'$ is regular at all closed points $x$ of degree $\leq r$ if and only if its image in $\ho^0(\X'_{p^2,\leq r}, \li^{\otimes d})$ lies in the subset which by the above natural isomorphism corresponds to
	\[
		\prod_{x\in |\X_{p^2}|, \deg x\leq r}\left( \ho^0(x', \li^{\otimes d})-\{0\}\right).
	\]
	
	To get the result, we need to study the surjectivity of the restriction map
	\[
		\ho^0(\X_{p^2}, \li^{\otimes d})\longrightarrow \ho^0(\X'_{p^2,\leq r}, \li^{\otimes d})\simeq\ho^0(\X_{p^2}, \mathcal{O}_{\X'_{p^2,\leq r}}\otimes \li^{\otimes d}).
	\]
	When $n=1$, the number of closed points of $|\X_{p^2}|$ is bounded above by $\#\X(\bbC)$. We have 
	\begin{eqnarray*}
		\dim_{\F_p}\ho^0(\X'_{p^2,\leq r}, \Oc_{\X'_{p^2,\leq r}})\otimes_{\Z/p^2\Z} \F_p	&=&\dim_{\F_p}\ho^0\Big(\X_{p^2}, \prod_{\deg x\leq r}\Oc_{x'}\Big)\otimes_{\Z/p^2\Z} \F_p\\
		&=& \sum_{\deg x\leq r}\dim_{\F_p}\ho^0(x', \Oc_{x'})\otimes_{\Z/p^2\Z} \F_p\\
		&=& \sum_{\deg x\leq r} {\big((1-1)+1\big)\deg x}\\
		&\leq& \sum_{x\in |\X_{p^2}|}\deg x<\infty
	\end{eqnarray*}
	for any $r>0$. So the restriction map is always surjective when $d$ is sufficiently large by Lemma \ref{sur}.
	When $n>1$, we have
	\begin{eqnarray*}
		\dim_{\F_p}\ho^0(\X'_{p^2,\leq r}, \Oc_{\X'_{p^2,\leq r}})\otimes_{\Z/p^2\Z} \F_p	&=&\dim_{\F_p}\ho^0(\X_{p^2}, \prod_{\deg x\leq r}\Oc_{x'})\otimes_{\Z/p^2\Z} \F_p\\
		&=& \sum_{\deg x\leq r}\dim_{\F_p}\ho^0(x', \Oc_{x'})\otimes_{\Z/p^2\Z} \F_p\\
		&=& \sum_{\deg x\leq r} {((n-1)+1)\deg x}\\
		&\leq & n \sum_{e=1}^r \#\X_{p^2}(\F_{p^e})e \\
		&\leq & n \sum_{e=1}^r c_0p^{(n-1)e}e  \\
		&\leq & n c_0r\cdot \sum_{e=1}^r p^{(n-1)e} \\
		&\leq &  n c_0r \frac{p^{(n-1)(r+1)}-1}{p^{n-1}-1}.
	\end{eqnarray*}
	By Lemma \ref{sur}, when $d\geq N\cdot(n c_0r\frac{p^{(n-1)(r+1)}-1}{p^{n-1}-1})$ with $N$ as in the lemma, the restriction map
	\[
		\ho^0(\X_{p^2}, \li^{\otimes d})\longrightarrow \ho^0(\X_{p^2}, \mathcal{O}_{\X'_{p^2,\leq r}}\otimes \li^{\otimes d}) 
	\]
	is surjective. In particular, since we have
	\begin{eqnarray*}
		n c_0r \frac{p^{(n-1)(r+1)}-1}{p^{n-1}-1}\leq nc_0r\frac{p^{(n-1)(r+1)}}{\frac{1}{2}p^{n-1}}=2nc_0rp^{(n-1)r},
	\end{eqnarray*}
	the surjectivity of the restriction holds for $r,d$ satisfying 
	\[
		2c_0Nnrp^{(n-1)r}\leq d.
	\]	
	For such $r,d$, we have 
	\[
		\frac{\#\mathcal{P}'_{d,p^2,\leq r}}{\#\ho^0(\X_{p}, \li^{\otimes d})}=\prod_{x\in |\X_{p}|, \deg x\leq r} \left( 1-p^{-(n+1)\deg x} \right).
	\]
\end{proof}

\subsection{Singular points of medium degree}\label{step2}

Let $N$ be an integer satisfying Lemma \ref{relample}. Set
\[
	\mathcal{Q}^{\mathrm{med}}_{d,p^2,r}:=\left\{ \sigma'\in \ho^0(\X_{p^2}, \li^{\otimes d})\ ;\ \exists x\in |\mathrm{div}\sigma'|\text{ of degree } r<\deg x\leq \frac{d}{nN}, \ \dim_{\kappa(x)}\frac{\m_{\mathrm{div}\sigma',x}}{\m^2_{\mathrm{div}\sigma', x}}=n \right\}.
\]
Note that a section $\sigma\in \ho^0(\X, \li^{\otimes d})$ is such that $\mathrm{div}\sigma$ has a singular point $x$ on the fiber $\X_p$ of degree $r<\deg x\leq \frac{d}{nN}$ if and only if for such an $x$, $\dim_{\kappa(x)}\frac{\m_{\mathrm{div}\sigma,x}}{\m^2_{\mathrm{div}\sigma, x}}=n$, and hence if and only if $\phi_{d,p^2}(\sigma)\in \mathcal{Q}^{\mathrm{med}}_{d,p^2,r}$.

\begin{lem}\label{medi}
	We have for $r\geq 1$,
	\[
		\frac{\#\mathcal{Q}^{\mathrm{med}}_{d,p^2,r}}{\#\ho^0(\X_{p^2}, \li^{\otimes d})}<2c_0p^{-2(r+1)},
	\]
	where the constant $c_0$ is as defined in Section \ref{c0}.
\end{lem}
\begin{proof}
	For any closed point $x$ in $\X_{p^2}$, applying Lemma \ref{sur} to the first order infinitesimal neighborhood $x'$ of $x$ in $\X_{p^2}$, we get that the restriction morphism
	\[
		 \ho^0(\X_{p^2}, \li^{\otimes d})\longrightarrow \ho^0(\X_{p^2}, \Oc_{x'}\otimes \li^{\otimes d})
	\]
	is surjective when
	\[
		N(n\deg x)\leq d,
	\]
	which is when $\deg x\leq \frac{d}{Nn}$.
	We can then estimate the proportion of elements in $\mathcal{Q}^{\mathrm{med}}_{d,p^2,r}$ by
	\begin{eqnarray*}
		\frac{\#\mathcal{Q}^{\mathrm{med}}_{d,p^2,r}}{\#\ho^0(\X_{p^2}, \li^{\otimes d})}&\leq& \sum_{r<\deg x\leq \lfloor\frac{d}{Nn}\rfloor}  \frac{\#\Ker\left( \ho^0(\X_{p^2}, \li^{\otimes d})\longrightarrow \ho^0(\X_{p^2}, \Oc_{x'}\otimes \li^{\otimes d})  \right) }{\#\ho^0(\X_{p^2}, \li^{\otimes d})}\\
		&\leq& \sum_{e=r+1}^{\lfloor\frac{d}{Nn}\rfloor}\#\X(\F_{p^e}) p^{-(n+1)e} \\
		&\leq& \sum_{e=r+1}^{\infty}c_0p^{(n-1)e} p^{-(n+1)e} \\
		&\leq&c_0 \sum_{e=r+1}^{\infty}p^{-2e}\\
		&=&\frac{c_0p^{-2(r+1)}}{1-p^{-2}}\\
		&<&2c_0p^{-2(r+1)}.
	\end{eqnarray*}
\end{proof}


\subsection{Singular points of large degree}\label{step3}
	
\begin{prop}\label{high}
	Fix a constant $0<\alpha<1$.
	There exist positive integers $N_0,N_1$ only depending on $\X$ and $\overline{\li}$ such that for any $p\leq d^{\alpha} $, denoting
	\[
		N(p)= (N_0+1)(N_1+p-1)+p,  
	\]
	and
	\[
		\mathcal{Q}^{\mathrm{high}}_{d,p^2}:=\left\{ \sigma'\in \ho^0(\X_{p^2}, \li^{\otimes d})\ ;\ \exists x\in |\mathrm{div}\sigma'|\text{ of degree}\geq \frac{d}{nN(p)}, \ \dim_{\kappa(x)}\frac{\m_{\mathrm{div}\sigma',x}}{\m^2_{\mathrm{div}\sigma', x}}=n \right\},
	\]
	we have
	\[
		\frac{\#\mathcal{Q}_{d,p^2}^{\mathrm{high}}}{\#\ho^0(\X_{p^2}, \li^{\otimes d})}=O\left( d^{n}p^{-c_1\frac{d}{p}} \right),
	\]
	where $c_1$ and the constant involved in big $O$ are independent of $d,p$ and $\alpha$.
	
	In particular, we have
	\[
		\lim_{d\rightarrow \infty}\frac{\#\mathcal{Q}_{d,p^2}^{\mathrm{high}}}{\#\ho^0(\X_{p^2}, \li^{\otimes d})}=0.
	\]
\end{prop}
\begin{rmq}
	When $d$ is large enough, for any $p\leq d^{\alpha}$, 
	\[
		d^{n}p^{-c_1\frac{d}{p}}\leq d^{n}d^{-c_1\alpha d^{1-\alpha}}=d^{n-c_1\alpha d^{1-\alpha}}.
	\]
	As $\alpha<1$, $d^{n-c_1\alpha d^{1-\alpha}}$ tends to $0$ when $d$ tends to infinity. So the above proportion $\frac{\#\mathcal{Q}_{d,p^2}^{\mathrm{high}}}{\#\ho^0(\X_{p^2}, \li^{\otimes d})}$ is always near $0$ for any $p\leq d^{\alpha}$ when $d$ is large enough.
\end{rmq}

When $\X$ is regular, so is its generic fiber $\X_{\Q}$, which is equivalent to say that $\X_{\Q}$ is smooth over $\Q$. This implies that we can find an open subset $\mathcal{S}$ of $\Spec\ \Z$ such that $\X_{\mathcal{S}}$ is smooth over $\mathcal{S}$. We will give a uniform control of the proportion of $\mathcal{Q}_{d,p^2}^{\mathrm{high}}$ for primes $p\in \mathcal{S}$ such that $2c_0N(p)np^{n-1}\leq d$. As $\Spec\ \Z-\mathcal{S}$ is a finite scheme, the set of primes $p$ where the proportion of $\mathcal{Q}_{d,p^2}^{\mathrm{high}}$ is not controlled is finite in number. We then give independent control of the proportion of $\mathcal{Q}_{d,p^2}^{\mathrm{high}}$ for each fiber with constants possibly depending on $p$. The finiteness of such $p$ permits us to get a uniform control for all primes $p$ satisfying $2c_0N(p)np^{n-1}\leq d$.
	
Thus Proposition \ref{high} is implied by the following two propositions :

\begin{prop}\label{hism}
	Fix a constant $0<\alpha<1$. For any prime $p\leq d^{\alpha}$ such that $\X_p$ is smooth over $\F_p$ 
	we have
	\[
		\frac{\#\mathcal{Q}_{d,p^2}^{\mathrm{high}}}{\#\ho^0(\X_{p^2}, \li^{\otimes d})}=O\left( d^{n-1}p^{-c_1\frac{d}{p}} \right),
	\]
	where $c_1$ and the constant involved in big $O$ are independent of $d,p$ and $\alpha$.
\end{prop}

\begin{prop}\label{hising}
	Fix a constant $0<\alpha<1$. For any prime number $p\leq d^{\alpha}$ with possibly singular $\X_p$, we have
	\[
		\frac{\#\mathcal{Q}_{d,p^2}^{\mathrm{high}}}{\#\ho^0(\X_{p^2}, \li^{\otimes d})}=O\left( d^{n}p^{-c'_1\frac{d}{p}} \right),
	\]
	where $c'_1$ and the constant involved in big $O$ are independent of $d,\alpha$, but may depend on $p$.
\end{prop}

We prove first the smooth case. Before proving Proposition \ref{hism}, we need some preparation.
\begin{lem}\label{opensub}
	Let $\mathcal{S}$ be an open subscheme of $\Spec\ \Z$ such that $\X_{\mathcal{S}}$ is smooth over $\mathcal{S}$. We can find a finite cover of $\X_{\mathcal{S}}$ by open subschemes $U$ satisfying the following conditions: 
\begin{enumerate}[label=\arabic*)]
	\item we can find $t_{1},\dots t_{n-1}\in \ho^0(U, \mathcal{O}_{U})$ such that
	\[
		\Omega_{U/\mathcal{S}}^1\simeq \bigoplus_{i=1}^{n-1} \mathcal{O}_{U}\mathrm{d} t_{i};
	\]
	\item for any positive integer $M$, we can choose a constant $N_0\geq M$ satisfying Lemma \ref{relample} such that there exists a $\tau_0 \in \ho^0(\X_{\mathcal{S}}, \li^{\otimes(N_0+1)})$ such that $\X_{\mathcal{S}}-\mathrm{div}\tau_0=U$;
	\item with the same $N_0$ as above, there exist $\tau_{1},\dots \tau_{k}\in \ho^0(\X_{\mathcal{S}}, \li^{\otimes N_0})$ such that $U=\bigcup_{1\leq j\leq k}(\X_{\mathcal{S}}-\mathrm{div}\tau_{j})$.
\end{enumerate}
\end{lem}

\begin{proof}[Proof]
	As $\Omega^1_{\X_{\mathcal{S}/\mathcal{S}}}$ is a locally free sheaf of rank $n-1$, 
	we can find a finite open covering $\{U_{\beta}\}$ of $\X_{\mathcal{S}}$ and sections $t_{\beta,1},\dots t_{\beta,n-1}\in \ho^0(U_{\beta}, \mathcal{O}_{U_{\beta}})$ such that
	\[
		\Omega^1_{U_{\beta}/\mathcal{S}}\simeq \bigoplus_{i=1}^{n-1} \mathcal{O}_{U_{\beta}}\mathrm{d} t_{\beta,i}.
	\]
	Hence we may assume the condition (1) for $U$. Moreover, if the condition (1) is satisfied by $U$, then it is also satisfied by any open subscheme of $U$.
	\bigskip
	
	If $U$ is an open subscheme of $\X_{\mathcal{S}}$ satisfying the condition (1), we will show that it can be covered by finitely many open subschemes $\{U_{\alpha}\}$ of $U$ which satisfy the conditions (2) and (3). This will complete the proof of the lemma.

	Since $\Li$ is ample, we may take a positive integer $N'_0>0$ satisfying Lemma \ref{relample} such that for any $d\geq N'_0$, the sheaf $\mathcal{I}_{\X_{\mathcal{S}}-U}\otimes \li_{\mathcal{S}}^{\otimes d}$ is globally generated, where $\mathcal{I}_{\X_{\mathcal{S}}-U}$ is the ideal sheaf of $\X_{\mathcal{S}}-U$ with the induced reduced structure. We may then choose non-zero sections 
	\[
		\tau'_{1},\dots, \tau'_s\in \ho^0(\X_{\mathcal{S}}, \mathcal{I}_{\X_{\mathcal{S}}-U}\otimes \li^{\otimes N'_0})\subset \ho^0(\X_{\mathcal{S}}, \li^{\otimes N'_0}),
	\]
	generating $\mathcal{I}_{\X_{\mathcal{S}}-U}\otimes \li^{\otimes N'_0}$. This means that set theoretically, we have $\X_{\mathcal{S}}-U=\bigcap_{i}\mathrm{div}(\tau'_i)$. In other words, we get a finite cover of $U$ :
	\[
		U=\bigcup_{i}\big( \X_{\mathcal{S}}-\mathrm{div}(\tau'_i) \big),
	\]
	where $\X_{\mathcal{S}}-\mathrm{div}(\tau'_i)$ are open subschemes of $\X_{\mathcal{S}}$. Note that for each $i$, $\X_{\mathcal{S}}-\mathrm{div}(\tau'_i)$ satisfies the condition (1) of the lemma. Without loss of generality, we may replace $U$ by one of the subschemes $\X_{\mathcal{S}}-\mathrm{div}(\tau'_i)$, i.e. we assume that there is a section $\tau'_0\in \ho^0(\X_{\mathcal{S}}, \li^{\otimes N'_0})$ such that $U=Y-\mathrm{div}(\tau'_0)$. We denote $\mathrm{div}(\tau'_0)$ by $D$. 
		
	Now set $N_0=r N'_0-1$ for some positive integer $r$ such that $N_0\geq M$ and that the sheaf $\mathcal{I}_D\otimes \li^{\otimes N_0}$ is globally generated. Then in particular we can find sections $\tau_1,\dots, \tau_k\in \ho^0(\X_{\mathcal{S}}, \mathcal{I}_D\otimes \li^{\otimes N_0})\subset \ho^0(\X_{\mathcal{S}}, \li^{\otimes N_0})$ such that $D=\bigcap_{j=1}^k\mathrm{div}(\tau_j)$ set theoretically. We also set $\tau_0=(\tau'_0)^r\in \ho^0(\X_{\mathcal{S}}, \li^{\otimes (N_0+1)})$. This means exactly $U=\bigcup_{1\leq j\leq k}(\X_{\mathcal{S}}-\mathrm{div}\tau_{j})$. In this situation we still have $D=\mathrm{div}(\tau_0)$ set theoretically.  The section $\tau_0$ and sections $\tau_1,\dots, \tau_k$ are then what we need for conditions (2) and (3) in the lemma.
\end{proof}

\begin{lem}\label{opensubworks}
	For an open subscheme $U$ of $\X_\mathcal{S}$, set
	\[
		\mathcal{Q}^{\mathrm{high}}_{d,p,U}=\left\{ \sigma\in \ho^0(\X_{p}, \li^{\otimes d})\ ;\ \mathrm{div}\sigma \text{ has a singular point in } U\cap \mathrm{div}\sigma \text{ of degree}\geq \frac{d}{nN(p)} \right\}.
	\]
	Then Proposition \ref{high} holds if we have
	\[
		\frac{\#\mathcal{Q}_{d,p,U}^{\mathrm{high}}}{\#\ho^0(\X_{p}, \li^{\otimes d})}=O\left( d^{n-1}p^{-c_{U,1}\frac{d}{p}} \right),
	\]
	for all $U$ satisfying the conditions in Lemma \ref{opensub}, where the constant $c_{U,1}$ only depends on $\X,\li$ and $U$.
\end{lem}

\begin{proof}[Proof]
	Let $\{U_{\alpha}\}$ be a finite cover of $\X_{\mathcal{S}}$ where all the $U_{\alpha}$ are open subschemes of $\X_{\mathcal{S}}$ satisfying the conditions in Lemma \ref{opensub}. Then for any $p\in \mathcal{S}$, we get a finite open cover $\{U_{\alpha,p^2}\}$ of $\X_{p^2}$. For an open set $U$ of $\X_{\mathcal{S}}$ flat over $\mathcal{S}$, set
	\[
		\mathcal{Q}^{\mathrm{high}}_{d,p^2,U}=\left\{ \sigma'\in \ho^0(\X_{p^2}, \li^{\otimes d})\ ;\ \exists x\in |\mathrm{div}\sigma'\cap U|\text{ of degree}\geq \frac{d}{nN(p)}, \ \dim_{\kappa(x)}\frac{\m_{\mathrm{div}\sigma',x}}{\m^2_{\mathrm{div}\sigma', x}}=n \right\}.
	\]
	To bound $\mathcal{Q}^{\mathrm{high}}_{d,p^2}$, it suffices to bound $\mathcal{Q}^{\mathrm{high}}_{d,p^2,U_{\alpha}}$ for all $U_{\alpha}$ in the covering. 
	
	Note that for any $\sigma\in \ho^0(\X_{p^2}, \li^{\otimes d})$ and any $x\in |\mathrm{div}\sigma|$, we have an exact sequence
	\[
		p\mathcal{O_{\mathrm{div}\sigma}} \longrightarrow \frac{\m_{\mathrm{div}\sigma,x}}{\m^2_{\mathrm{div}\sigma, x}} \longrightarrow \frac{\m_{\mathrm{div}\overline{\sigma},x}}{\m^2_{\mathrm{div}\overline{\sigma}, x}} \longrightarrow 0
	\]
	where $\overline{\sigma}=\sigma\ \mathrm{mod}\ p$ is the restriction in $\ho^0(\X_{p}, \li^{\otimes d})$ and $\mathrm{div}\overline{\sigma}$ is the divisor in $\X_p$. Therefore
	\[
		\dim_{\kappa(x)}\frac{\m_{\mathrm{div}\overline{\sigma},x}}{\m^2_{\mathrm{div}\overline{\sigma}, x}} \geq \dim_{\kappa(x)}\frac{\m_{\mathrm{div}\sigma,x}}{\m^2_{\mathrm{div}\sigma, x}}-1
	\]
	
	In particular, if $\dim_{\kappa(x)}\frac{\m_{\mathrm{div}\sigma,x}}{\m^2_{\mathrm{div}\sigma, x}}=n=\dim \X_{\mathcal{S}}$, then as 
	\[
	 	\dim_{\kappa(x)}\frac{\m_{\mathrm{div}\sigma,x}}{\m^2_{\mathrm{div}\sigma, x}}-1\leq \dim_{\kappa(x)}\frac{\m_{\mathrm{div}\overline{\sigma},x}}{\m^2_{\mathrm{div}\overline{\sigma}, x}} \leq \dim \X_p=n-1,
	\]
	we have $ \dim_{\kappa(x)}\frac{\m_{\mathrm{div}\overline{\sigma},x}}{\m^2_{\mathrm{div}\overline{\sigma}, x}}=n-1$, which means that $x$ is a singular point of $\mathrm{div}\overline{\sigma}$.

	Then for a section $\sigma\in \ho^0(\X_{p^2}, \li^{\otimes d})$, $\sigma \in \mathcal{Q}^{\mathrm{high}}_{d,p^2,U}$ implies $\overline{\sigma}\in \mathcal{Q}^{\mathrm{high}}_{d,p,U}$. Thus
	\[
		\frac{\#\mathcal{Q}_{d,p^2,U}^{\mathrm{high}}}{\#\ho^0(\X_{p^2}, \li^{\otimes d})}\leq \frac{\#\{ \sigma\in \ho^0(\X_{p^2}, \li^{\otimes d})\ ;\ \overline{\sigma}\in \mathcal{Q}^{\mathrm{high}}_{d,p,U}  \} }{\#\ho^0(\X_{p^2}, \li^{\otimes d})} = \frac{\#\mathcal{Q}_{d,p,U}^{\mathrm{high}}}{\#\ho^0(\X_{p}, \li^{\otimes d})}.
	\]
	It suffices then to bound $\#\mathcal{Q}_{d,p,U_{\alpha}}^{\mathrm{high}}$ for $U_{\alpha}$ in the covering. Here all $U_{\alpha}$ satisfies the conditions in Lemma \ref{opensub}. Such a finite covering exists by Lemma \ref{opensub}. If for any $\alpha$, the estimate
	\[
		\frac{\#\mathcal{Q}_{d,p,U_{\alpha}}^{\mathrm{high}}}{\#\ho^0(\X_{p}, \li^{\otimes d})}=O\left( d^{n-1}p^{-c_{U_{\alpha},1}\frac{d}{p}} \right)
	\]
	holds, then setting $c_1=\min_{\alpha}\{ c_{U_{\alpha},1} \}>0$, we have
	\[
		\frac{\#\mathcal{Q}_{d,p^2}^{\mathrm{high}}}{\#\ho^0(\X_{p^2}, \li^{\otimes d})}\leq \sum_{\alpha}\frac{\#\mathcal{Q}_{d,p,U_{\alpha}}^{\mathrm{high}}}{\#\ho^0(\X_{p}, \li^{\otimes d})}=O\left( d^{n-1}p^{-c_1\frac{d}{p}} \right).
	\]
\end{proof}

Now for an open subscheme $U$ as in Lemma \ref{opensub}, we get morphisms
\begin{eqnarray*}
	\Phi_j :  \ho^0(\X_{\mathcal{S}}, \li^{\otimes d}) &\longrightarrow & \ho^0(U, \mathcal{O}_U) \\
	\sigma &\longmapsto & \quad \frac{\sigma\cdot \tau_j^d}{\tau_0^d}
\end{eqnarray*}
for any $d\in \Z_{>0}$ and $1\leq j\leq k$.

On the other hand, Lemma \ref{extend} tells us that there exists a positive integer $N_1$ such that for any $\sigma \in  \ho^0(\X_{p}, \li^{\otimes d})$, any $1\leq i\leq n-1$, the section $\left(\partial_i \Phi_j(\sigma) \right)\cdot \tau_{0,p}^{d+\delta}$ extends to a global section in $\ho^0(\X_{p}, \li^{\otimes (N_0+1)(d+\delta)})$ for any $\delta \geq N_1$. Here $\tau_{0,p}$ is the restriction of $\tau_0$ modulo $p$ in $\ho^0(\X_p, \li^{\otimes (N_0+1)})$.
		
In fact, we may choose $N_1$ to be $N_1=N'_1+d_0+1$ where $N'_1$ is such that when $\delta\geq N'_1$, for any $1\leq i\leq n$ the section $\partial_i\cdot \tau_0^{\delta}\in \ho^0\left( U, \mathcal{H}om_{\mathcal{O}_{\X_\mathcal{S}}}(\Omega^1_{\X_\mathcal{S}/\mathcal{S}}, \mathcal{O}_{\X_\mathcal{S}})\otimes \li^{(N_0+1)\delta} \right)$ can be extended to a global section in 
\[
	\ho^0\left( \X_\mathcal{S}, \mathcal{H}om_{\mathcal{O}_{\X_\mathcal{S}}}(\Omega^1_{\X_\mathcal{S}/\mathcal{S}}, \mathcal{O}_{\X_\mathcal{S}})\otimes \li^{(N_0+1)\delta} \right)\simeq \Hom(\Omega^1_{\X_\mathcal{S}/\mathcal{S}}, \li^{\otimes(N_0+1)\delta}),
\]
and where $d_0$ is such that for any $d\geq d_0$, the restriction morphism 
\[
	\ho^0(\pr\Big( \ho^0(\X_{\mathcal{S}}, \li^{\otimes (N_0+1)}) \Big),\mathcal{O}(d)) \longrightarrow \ho^0(\X_{\mathcal{S}}, \li^{\otimes (N_0+1)d})
\]
is surjective. So $N_1$ is again independent of $d$ and $p$.

We enlarge $N_0$ if necessary so that it satisfies the following conditions:
\begin{enumerate}[label=\arabic*)]
	\item $N_0+1$ is a power of a prime number ;
	\item for any $d\geq N_0$, $(\li |_{\X_\mathcal{S}})^{\otimes d}$ is very ample;
	\item for any $a,b\geq N_0$, we have a surjective morphism
	\[
		\ho^0(\X_{\mathcal{S}}, \li^{\otimes a}) \otimes_{\ho^0(\mathcal{S}, \mathcal{O}_{\mathcal{S}})}\ho^0(\X_{\mathcal{S}}, \li^{\otimes b})\longrightarrow \ho^0(\X_{\mathcal{S}}, \li^{\otimes (a+b)}).
	\]
\end{enumerate}
	
We prove the following result :
\begin{lem}\label{U}
	For any prime $p\in \mathcal{S}$, take 
	\[
		N(p)=(N_0+1)(N_1+p-1)+p=p(N_0+2)+(N_0+1)(N_1-1).
	\] 
	With notation as in Proposition \ref{high}, if $p$ and $N(p)$ satisfy $2c_0N(p)np^{n-1}\leq d$, then we have
	\[
		\frac{\#\mathcal{Q}_{d,p,U}^{\mathrm{high}}}{\#\ho^0(\X_{p}, \li^{\otimes d})}=O\left( d^{n-1}p^{-c_1\frac{d}{p}} \right),
	\]
	where $c_1$ and the constant involved in big $O$ are independent of $d,p$.
\end{lem}
By Lemma \ref{opensubworks}, this implies Proposition \ref{hism}.
\begin{proof}[Proof of Lemma \ref{U}]	
	For each $p\in \mathcal{S}$, if $d\geq N(p)$, $d$ has a unique decomposition
	\[
		d=pk_{p,d}+(N_0+1)l_{p,d}
	\]
	with $N_1\leq l_{p,d}< N_1+p$. We have a surjective map
	\begin{eqnarray*}
		\ho^0(\X_p,\mathcal{L}^{\otimes d})\times \left( \prod_{i=1}^{n-1} \ho^0(\X_p,\mathcal{L}^{\otimes k_{p,d}}) \right) \times \ho^0(\X_p, \mathcal{L}^{\otimes k_{p,d}})  \longrightarrow   \ho^0(\X_p, \mathcal{L}^{\otimes d})
	\end{eqnarray*}
	which sends $\big(\sigma_{0}, (\beta_1,\dots, \beta_{n-1}), \gamma\big)$ to
	\[
		\sigma=\sigma_{0}+\sum_{i=1}^{n-1} \beta_i^pt_i\tau_{0,p}^{l_d}+\gamma^p\tau_{0,p}^{l_d},
	\]
	where $\tau_{0,p}$ is the restriction of $\tau_0$ modulo $p$ in $\ho^0(\X_p, \li^{\otimes (N_0+1)})$. Thus
	\[
		\Phi_{j,p}(\sigma)=\Phi_{j,p}(\sigma_0)+\sum_{i=1}^{n-1} \Phi_{j,p}(\beta_i)^pt_i\Phi_{j,p}(\tau_{0,p})^{l_d}+\Phi_{j,p}(\gamma)^p\Phi_{j,p}(\tau_{0,p})^{l_d}
	\]
	in $\ho^0(U\cap \X_p,\mathcal{O}_{U\cap \X_p})$, where $\Phi_{j.p}=\Phi_j|_{U\cap \X_p}$. As $\tau_{0,p}$ is nowhere zero on $U$, we have
	\[
		\mathrm{Sing}(\mathrm{div}\sigma)\cap (U-\mathrm{div}\tau_{j,p}) = \mathrm{Sing}(\mathrm{div}\Phi_{j,p}(\sigma))\cap  (U-\mathrm{div}\tau_{j,p})
	\]
	and hence
	\[
		\mathrm{Sing}(\mathrm{div}\sigma)\cap U \subset  \bigcup_{j=1}^k \mathrm{Sing}(\mathrm{div}\Phi_{j,p}(\sigma)).
	\]
	Since
	\[
		\partial_i[\Phi_{j,p}(\beta_i)^p t_i \Phi_{j,p}(\tau_{0,p})^{l_d}]=\Phi_{j,p}(\beta_i)^p \Phi_{j,p}(\tau_{0,p})^{l_d} + l_d \Phi_{j,p}(\beta_i)^p t_i \Phi_{j,p}(\tau_{0,p})^{l_d-1}\cdot \partial_i \Phi_{j,p}(\tau_{0,p}),
	\]
	and for any $i'\not=i$,
	\[
		\partial_i[\Phi_{j,p}(\beta_{i'})^p t_{i'} \Phi_{j,p}(\tau_{0,p})^{l_d}]= l_d \Phi_{j,p}(\beta_{i'})^p t_{i'} \Phi_{j,p}(\tau_{0,p})^{l_d-1}\cdot \partial_i \Phi_{j,p}(\tau_{0,p}),
	\]
	we have
	\begin{eqnarray*}
		\partial_i\Phi_{j,p}(\sigma) &=& \partial_i \Phi_{j,p}(\sigma_0)+\sum_{i'=1}^{n-1} \partial_i\left[\Phi_{j,p}(\beta_{i'})^p t_{i'}\Phi_{j,p}(\tau_{0,p})^{l_d}\right]+\partial_i\left[\Phi_{j,p}(\gamma)^p\Phi_{j,p}(\tau_{0,p})^{l_d}\right]\\
		&=& \left[ \sum_{i'=1}^{n-1} l_d\Phi_{j,p}(\beta_{i'})^p t_{i'}\Phi_{j,p}(\tau_{0,p})^{l_d-1}+l_d \Phi_{j,p}(\gamma)^p\Phi_{j,p}(\tau_{0,p})^{l_d-1} \right]\cdot \partial_i \Phi_{j,p}(\tau_{0,p})\\
		& & \qquad \qquad\qquad\qquad\qquad\qquad+\partial_i\Phi_{j,p}(\sigma_0)+\Phi_{j,p}(\beta_i)^p\Phi_{j,p}(\tau_{0,p})^{l_d} \\\\
		&=& \partial_i\Phi_{j,p}(\sigma_0)+ \frac{l_d(\Phi_{j,p}(\sigma)-\Phi_{j,p}(\sigma_0))}{\Phi_{j,p}(\tau_{0,p})}\partial_i\Phi_{j,p}(\tau_{0,p})+ \Phi_{j,p}(\beta_i)^p \Phi_{j,p}(\tau_{0,p})^{l_d}.
	\end{eqnarray*}
	
	Now set
	\[
		g_{p, j,i}(\sigma_0,\beta_i)=\partial_i \Phi_{j,p}(\sigma_0)-\frac{l_d\Phi_{j,p}(\sigma_0)}{\Phi_{j,p}(\tau_{0,p})} \partial_i\Phi_{j,p}(\tau_{0,p})+\Phi_{j,p}(\beta_i)^p\Phi_{j,p}(\tau_{0,p})^{l_d},
	\]
	and
	\[
		W_{p,j,i}:=\X_p\cap U\cap \{ g_{p,j,1}=\cdots=g_{p,j,i}=0 \}.
	\]
	Then for any $\sigma=\sigma_0+\sum_{i=1}^{n-1} \beta_i^pt_i\tau_{0,p}^{l_d}+\gamma^p\tau_{0,p}^{l_d}$, comparing the expressions of $g_{p,j,i}$ and $\partial_i\Phi_{j,p}$ we have
	\[
		g_{p,j,i}(\sigma_0,\beta_i)=\partial_i\Phi_{j,p}(\sigma)-\frac{l_d\Phi_{j,p}(\sigma)}{\Phi_{j,p}(\tau_{0,p})} \partial_i\Phi_{j,p}(\tau_{0,p}),
	\]
	and hence
	\[
		g_{p,j,i}(\sigma_0,\beta_i)|_{ \mathrm{div\ } \Phi_{j,p}(\sigma)}=\partial_i\Phi_{j,p}(\sigma)|_{ \mathrm{div\ } \Phi_{j,p}(\sigma)}.
	\]
	Moreover, any section $g_{p,j,i}(\sigma_0,\beta_i)\cdot \tau_{0,p}^{d+\delta}\in \ho^0(\X_p\cap U, \li^{\otimes (N_0+1)(d+\delta)})$ can be extended to a global section in $\ho^0(\X_p, \li^{\otimes (N_0+1)(d+\delta)})$ for any $\delta\geq N_1+1$. In fact, we know already that the section $\partial_i\Phi_{j,p}(\sigma)\cdot \tau_{0,p}^{d+\delta}\in \ho^0(\X_p\cap U, \li^{\otimes (N_0+1)(d+\delta)})$ can be extended to a global section in $\ho^0(\X_p, \li^{\otimes (N_0+1)(d+\delta)})$ for any $\delta\geq N_1$. On $\X_p\cap U$ we have
	\begin{eqnarray*}
		\frac{l_d\Phi_{j,p}(\sigma)}{\Phi_{j,p}(\tau_{0,p})}&=& l_d\frac{\sigma \cdot \tau_{j,p}^d}{\tau_{0,p}^d} \cdot\left(\frac{\tau \cdot \tau_{j,p}^{N_0+1}}{\tau_{0,p}^{N_0+1}}\right)^{-1}  \\
		&=&l_d\frac{\sigma \cdot \tau_{j,p}^{d-N_0-1}}{\tau_{0,p}^{d-N_0}}.
	\end{eqnarray*}
	By Lemma \ref{extend}, for any $\delta\geq N_1$, the section $ \partial_i\Phi_{j,p}(\tau_{0,p}) \cdot \tau_{0,p}^{(N_0+1)+\delta}$ can be extended to a global section in $\ho^0(\X_p, \li^{\otimes (N_0+1)((N_0+1)+\delta)})$, so the section
	\begin{eqnarray*}
		\frac{l_d\Phi_{j,p}(\sigma)}{\Phi_{j,p}(\tau_{0,p})} \partial_i\Phi_{j,p}(\tau_{0,p})\cdot \tau_{0,p}^{d+\delta} &=& \left( \frac{l_d\Phi_{j,p}(\sigma)}{\Phi_{j,p}(\tau_{0,p})}\cdot \tau_{0,p}^{d-N_0}\right)\cdot \left( \partial_i\Phi_{j,p}(\tau_{0,p}) \cdot \tau_{0,p}^{(N_0+1)+(\delta-1)} \right)
	\end{eqnarray*}
	extends to a global section of $\li^{\otimes (N_0+1)(d+\delta)}$ for any $\delta \geq N_1+1$. 
	Therefore the section 
	\[
		g_{j,i}(\sigma_0,\beta_i)\cdot \tau^{d+\delta}=\left(\partial_i\Phi_{j}(\sigma)-\frac{l_d\Phi_j(\sigma)}{\Phi_j(\tau)}\partial_i\Phi_j(\tau)\right)\cdot \tau^{d+\delta}\in \ho^0(U, \li^{\otimes (N_0+1)(d+\delta)})
	\]
	can be extended to a global section in $\ho^0(\X_p, \li^{\otimes (N_0+1)(d+\delta)})$ for any $\delta\geq N_1+1$.
	
	\begin{lem}\label{5.8}
		When $d$ is sufficiently large, the proportion of 
		\[
			\big(\sigma_0, (\beta_1,\dots, \beta_{n-1}), \gamma\big) \in \ho^0(\X_p,\mathcal{L}^{\otimes d})\times \left( \prod_{i=1}^{n-1} \ho^0(\X_p,\mathcal{L}^{\otimes k_{p,d}}) \right) \times \ho^0(\X_p, \mathcal{L}^{\otimes k_{p,d}}) 
		\]
		such that for $\sigma=\sigma_{0}+\sum_{i=1}^{n-1} \beta_i^pt_i\tau_{0,p}^{l_d}+\gamma^p\tau_{0,p}^{l_d}$,
		\[
			\mathrm{div}\sigma\cap W_{p,n-1,j}\cap \Big\{ x\in |\X_p|\ ;\ \deg x\geq\frac{d}{nN(p)} \Big\}=\emptyset,
		\]
		is
		\[
			1-O\left( d^{n-1}p^{-c_1\frac{d}{p}} \right),
		\]
		with a constant $c_1$ depending only on $N_0,N_1$ and the dimension $n$.
	\end{lem}
	\begin{proof}[Proof]
		Apply Lemma \ref{reddim} to the case $Y=\X_p$ and $X=U\cap \X_p$. We obtain that for $0\leq i\leq n-2$, with a fixed choice of $\sigma_0,\beta_1,\dots, \beta_i$ such that $\dim W_{p,i,j}\leq n-1-i$, the proportion of $\beta_{i+1}$ in $\ho^0(\X_p,\mathcal{L}^{\otimes k_{p,d}})$ such that $\dim W_{p,i+1,j} \leq n-2-i$ is $1-O\big(d^i\cdot p^{2-\frac{d}{(N_0+1)N_1p}}\big)$, where the constant involved depends only on the degree of $\X_p$ when embedded in $\pr(\ho^0(\li_p^{\otimes(N_0+1)})^{\vee})$ (this degree is independent of $p$), hence is independent of $d,p$.
		In particular, the proportion of $\big(\sigma_0,(\beta_1,\dots, \beta_{n-1})\big)$ such that $W_{p,n-1,j}$ is finite is 
		\[
			\prod_{i=0}^{n-2}\left( 1-O\big(d^i\cdot p^{2-\frac{d}{(N_0+1)N_1p}}\big) \right)=1-O\Big(d^{n-2}\cdot p^{2-\frac{d}{(N_0+1)N_1p}}\Big).
		\]
		And then Lemma \ref{finiteset} tells us that for fixed $\big(\sigma_0,(\beta_1,\dots, \beta_{n-1})\big)$ making $W_{p,n-1,j}$ finite, the proportion of $\gamma\in \ho^0(\X_p,\mathcal{L}^{\otimes k_{p,d}})$ such that for $\sigma=\sigma_{0}+\sum_{i=1}^{n-1} \beta_i^pt_i\tau_{0,p}^{l_d}+\gamma^p\tau_{0,p}^{l_d}$,
		\[
			\mathrm{div}\sigma\cap W_{p,n-1,j}\cap \left\{ x\in |\X_p|\ ;\ \deg x\geq\frac{d}{nN(p)} \right\}=\emptyset,
		\]
		is
		\[
			1-O(d^{n-1}p^{-\frac{d}{nN(p)}}),
		\]
		where the constant involved is independent of $d,p$.

		Therefore for large enough $d$, the proportion of 
		\[
			\big(\sigma_0, (\beta_1,\dots, \beta_{n-1}), \gamma\big)\in \ho^0(\X_p,\mathcal{L}^{\otimes d})\times \left( \prod_{i=1}^{n-1} \ho^0(\X_p,\mathcal{L}^{\otimes k_{p,d}}) \right) \times \ho^0(\X_p, \mathcal{L}^{\otimes k_{p,d}})
		\]
		such that 
		\[
			\mathrm{div}\sigma\cap W_{p,n-1,j}\cap \left\{ x\in |\X_p|\ ;\ \deg x\geq\frac{d}{nN(p)} \right\}=\emptyset,
		\]
		where $\sigma=\sigma_0+\sum_{i=1}^{n-1} \beta_i^pt_i\tau_{0,p}^{l_d}+\gamma^p\tau_{0,p}^{l_d}$, is
		\begin{eqnarray*}
			& &\left( \prod_{i=0}^{n-2}\Big( 1-O\big(d^i\cdot p^{2-\frac{d}{(N_0+1)N_1p}}\big) \Big)  \right) \cdot \left( 1-O\big(d^{n-1}p^{-\frac{d}{nN(p)}}\big) \right) \\
			&=&\left( 1-O\big(d^{n-2}\cdot p^{2-\frac{d}{(N_0+1)N_1p}} \big)\right)\cdot  \left( 1-O\big(d^{n-1}p^{-\frac{d}{nN(p)}}\big) \right) \\
			&=&\left( 1-O\big(d^{n-2}\cdot p^{2-\frac{d}{(N_0+1)N_1p}} \big)\right)\cdot  \left( 1-O\big(d^{n-1}p^{-\frac{d}{nN(p)}}\big) \right) \\
			&=&1-O\left(\max\Big(d^{n-2}p^{2-\frac{d}{(N_0+1)N_1p}}, d^{n-1}p^{-\frac{d}{nN(p)}}\Big) \right),
		\end{eqnarray*}
		As $N(p)=p(N_0+2)+(N_0+1)(N_1-1)$, when $d$ is sufficiently large, we have
		\begin{eqnarray*}
			\frac{d}{nN(p)}&=& \frac{d}{n\big[p(N_0+2)+(N_0+1)(N_1-1)\big]}\\
			&\geq& \frac{d}{n\big[(N_0+2)(N_1+p-1)\big]} \\
			&\geq& \frac{d}{n(N_0+2)N_1 p},
		\end{eqnarray*}
		and 
		\[
			\frac{d}{(N_0+1)N_1p}-2\geq \frac{d}{2n(N_0+2)N_1 p}.
		\]
		Therefore when $d$ tends to infinity,
		\begin{eqnarray*}
			& &\left( \prod_{i=0}^{n-2}\Big( 1-O\big(d^i\cdot p^{2-\frac{d}{(N_0+1)N_1p}}\big) \Big)  \right) \cdot \left( 1-O\big(d^{n-1}p^{-\frac{d}{nN(p)}}\big) \right) \\
			&=&1-O\left(\max\Big(d^{n-2}p^{2-\frac{d}{(N_0+1)N_1p}}, d^{n-1}p^{-\frac{d}{nN(p)}}\Big) \right)\\
			&=&1-O\left( d^{n-1}p^{-\frac{d}{2n(N_0+2)N_1 p}} \right)\\
			&=&1-O\left( d^{n-1}p^{-c_1\frac{d}{p}} \right)
		\end{eqnarray*}
		with constant $c_1=\frac{1}{2n(N_0+2)N_1 }$ which then depends only on $N_0,N_1$ and the dimension $n$.
	\end{proof}
		
	On the other hand, for such $\sigma=\sigma_0+\sum_{i=1}^{n-1} \beta_i^pt_i\tau_{0,p}^{l_d}+\gamma^p\tau_{0,p}^{l_d}\in \ho^0(\X_p,\mathcal{L}^{\otimes d})$, we have
	\[
		\mathrm{Sing}(\mathrm{div}\sigma)\cap U \subset  \bigcup_{j=1}^k \mathrm{Sing}(\mathrm{div}\Phi_{j,p}(\sigma))
	\]
	and
	\begin{eqnarray*}
		\mathrm{Sing}(\mathrm{div}\Phi_{j,p}(\sigma))
		&=& \mathrm{div}\Phi_{j,p}(\sigma)\cap \{ \partial_1\Phi_{j,p}(\sigma)=\cdots =\partial_{n-1}\Phi_{j,p}(\sigma)=0 \} \\
		&=& \mathrm{div}\Phi_{j,p}(\sigma)\cap \{ g_{p,j,1}(\sigma_0,\beta_1)=\cdots = g_{p,j,n-1}(\sigma_0,\beta_{n-1})=0 \} \\
		&=& \mathrm{div}\Phi_{j,p}(\sigma)\cap W_{p,n-1,j} \\
		&=& \mathrm{div}\sigma\cap W_{p,n-1,j}
	\end{eqnarray*}
	as
	\[
		g_{p,j,i}(\sigma_0,\beta_i)|_{ \mathrm{div\ } \Phi_{j,p}(\sigma)}=\partial_i\Phi_{j,p}(\sigma)|_{ \mathrm{div\ } \Phi_{j,p}(\sigma)}.
	\]
	Since the homomorphism of groups
	\begin{eqnarray*}
		\ho^0(\X_p,\mathcal{L}^{\otimes d})\times \left( \prod_{i=1}^{n-1} \ho^0(\X_p,\mathcal{L}^{\otimes k_{p,d}}) \right) \times \ho^0(\X_p, \mathcal{L}^{\otimes k_{p,d}})  \longrightarrow   \ho^0(\X_p, \mathcal{L}^{\otimes d})
	\end{eqnarray*}
	sending $\big(\sigma_0, (\beta_1,\dots, \beta_{n-1}), \gamma\big)$ to $\sigma=\sigma_0+\sum_{i=1}^{n-1} \beta_i^pt_i\tau_{0,p}^{l_d}+\gamma^p\tau_{0,p}^{l_d}$ is surjective, Lemma \ref{5.8} implies that
	\[
		\frac{\#\left\{\sigma\in \ho^0(\X_p, \mathcal{L}^{\otimes d})\ ;\  \mathrm{Sing}(\mathrm{div}\sigma)\cap U\cap \big\{ x\in |\X_p|\ ;\ \deg x\geq\frac{d}{nN(p)} \big\}=\emptyset \right\}}{\# \ho^0(\X_p, \mathcal{L}^{\otimes d})}= 1-O\left( d^{n-1}p^{-c_1\frac{d}{p}} \right).
	\]
	which means that the proportion of $\sigma\in \ho^0(\X_p, \li^{\otimes d})$ such that $\mathrm{div}\sigma$ has no singular point of degree strictly larger than $\frac{d}{nN(p)}$, that is, elements not contained in $\mathcal{Q}_{d,p,U}^{\mathrm{high}}$, is $1-O\left( d^{n-1}p^{-c_1\frac{d}{p}} \right)$ with a constant $c_1$ depending only on $N_0$, $N_1$ and $n$. We therefore conclude that
	\[
		\frac{\#\mathcal{Q}_{d,p,U}^{\mathrm{high}}}{\#\ho^0(\X_p, \li^{\otimes d})}=O\left( d^{n-1}p^{-c_1\frac{d}{p}}\right)
	\]
	with a possibly smaller $c_1$. \\

	Now that we have proved Lemma \ref{U} except for the only prime number $p_0$ dividing $N_0+1$, we can run the same process with another constant $N'_0> N_0$ such that  $(N'_0+1,N_0+1)=1$. We get a control for the proportion of $\mathcal{Q}_{d,p_0,U}^{\mathrm{high}}$ with different constants as we have
	\[
		\mathcal{Q}_{d,p_0,U}^{\mathrm{high}}\subset \left\{ \sigma'\in \ho^0(\X_{p_0^2}, \li^{\otimes d})\ ;\ \exists x\in |\mathrm{div}\sigma'|\text{ of degree}\geq \frac{d}{nN'(p_0)}, \ \dim_{\kappa(x)}\frac{\m_{\mathrm{div}\sigma',x}}{\m^2_{\mathrm{div}\sigma', x}}=n \right\},
	\]
	where
	\[
		N'(p_0)=(N'_0+1)(N_1+p_0-1)+p_0>(N_0+1)(N_1+p_0-1)+p_0=N(p_0).
	\]
	So by modifying the constant $c_1$ and the constant involved in the big $O$, this case can be included in the uniform control. Therefore we proved Lemma \ref{U}.
\end{proof}

Now we proceed to prove Proposition \ref{hising}.
\begin{proof}[Proof of Proposition \ref{hising}]
	The main problem for controlling the proportion of $\mathcal{Q}_{d,p^2}^{\mathrm{high}}$ for $p\not\in \mathcal{S}$ is that $\X_{p}$ over $p$ might be singular. We decompose $\X_p$ into regular and singular part:
	\[
		\X_p=U_p\cup Z_p
	\]
	where $Z_p=\mathrm{Sing}(\X_p)$ is the singular locus of $\X_p$ and $U_p=\X_p-Z_p$. As $\X$ is regular, for a closed point $x$ in $\X_p$, if $x\in U_p$, $\dim_{\kappa(x)}\mathfrak{m}_{\X_p, x}/\mathfrak{m}^2_{\X_p, x}=n-1$; if $x\in Z_p$, $\dim_{\kappa(x)}\mathfrak{m}_{\X_p, x}/\mathfrak{m}^2_{\X_p, x}=\dim_{\kappa(x)}\mathfrak{m}_{\X, x}/\mathfrak{m}^2_{\X, x}=n$. Set
	\begin{eqnarray*}
		\mathcal{Q}_{d,U_p}^{\mathrm{high}}&=&  \left\{ \sigma\in \ho^0(\X_{p}, \li^{\otimes d})\ ;\ \exists x\in |\mathrm{div}\sigma\cap U_p|\text{ of degree}\geq \frac{d}{nN(p)}, \ \dim_{\kappa(x)}\frac{\m_{\mathrm{div}\sigma,x}}{\m^2_{\mathrm{div}\sigma, x}}=n-1 \right\},\\
		\mathcal{Q}_{d,Z_p}^{\mathrm{high}}&=&  \left\{ \sigma\in \ho^0(\X_{p}, \li^{\otimes d})\ ;\ \exists x\in |\mathrm{div}\sigma\cap Z_p|\text{ of degree}\geq \frac{d}{nN(p)}, \ \dim_{\kappa(x)}\frac{\m_{\mathrm{div}\sigma,x}}{\m^2_{\mathrm{div}\sigma, x}}=n \right\}.
	\end{eqnarray*}
	For a section $\sigma\in \ho^0(\X_{p^2}, \li^{\otimes d})$, assume that $\mathrm{div}\sigma$ contains a closed point $x$ with $\dim_{\kappa(x)}\frac{\m_{\mathrm{div}\sigma,x}}{\m^2_{\mathrm{div}\sigma, x}}=n$. Let $\overline{\sigma}=\sigma \ \mathrm{mod}\ p$ be its image in $\ho^0(\X_{p}, \li^{\otimes d})$. Then if $x\in U_p$, $x$ is also a singular point of $\mathrm{div}\overline{\sigma}\cap U_p$, i.e. $\dim_{\kappa(x)}\frac{\m_{\mathrm{div}\overline{\sigma},x}}{\m^2_{\mathrm{div}\overline{\sigma}, x}}=n-1$; if $x\in Z_p$, we have then $\dim_{\kappa(x)}\frac{\m_{\mathrm{div}\overline{\sigma},x}}{\m^2_{\mathrm{div}\overline{\sigma}, x}}=n$. So we have
	\[
		\{ \overline{\sigma}\ ;\ \sigma \in \mathcal{Q}_{d,p^2}^{\mathrm{high}} \}\subset \mathcal{Q}_{d,U_p}^{\mathrm{high}}\cup \mathcal{Q}_{d,Z_p}^{\mathrm{high}},
	\]
	hence
	\[
		\frac{\#\mathcal{Q}_{d,p^2}^{\mathrm{high}}}{\#\ho^0(\X_{p^2}, \li^{\otimes d})}\leq \frac{\#\mathcal{Q}_{d,U_p}^{\mathrm{high}}}{\#\ho^0(\X_{p}, \li^{\otimes d})}+\frac{\#\mathcal{Q}_{d,Z_p}^{\mathrm{high}}}{\#\ho^0(\X_{p}, \li^{\otimes d})}.
	\]
	We can bound the first term $\frac{\#\mathcal{Q}_{d,U_p}^{\mathrm{high}}}{\#\ho^0(\X_{p}, \li^{\otimes d})}$ by exactly the same method as in the proof of Lemma \ref{U}. The second term can be bounded by a slightly different way. 
		
	As now $(\Omega^1_{\X_p/\F_p})|_{Z_p}$ is locally free of rank $n$, we cover an open neighbourhood of $Z_p$ by open subschemes $V_{Z_p,\alpha}$ where we can find $t_{\alpha,1},\dots,t_{\alpha,n}\in \ho^0(V_{Z_p,\alpha},\mathcal{O}_{V_{Z_p,\alpha}})$ such that the image $\overline{\mathrm{d}t_i}$ of ${\mathrm{d}t_i}$ in $\left(\Omega^1_{V_{Z_p,\alpha}/\F_p}\right)\Big|_{V_{Z_p,\alpha}\cap Z_p}$ satisfies
	\[
		\left(\Omega^1_{V_{Z_p,\alpha}/\F_p}\right)\Big|_{V_{Z_p,\alpha}\cap Z_p}\simeq \bigoplus_{i=1}^n\mathcal{O}_{V_{Z_p,\alpha}\cap Z_p}\overline{\mathrm{d}t_i}.
	\]
	Then choosing convenient constants $N'_0\geq N_0,\ N'_1\geq N_1$ and setting
	\[
		N'(p)=(N'_0+1)(N'_1+p-1)+p,
	\]
	the same process as in the proof of Lemma \ref{U} gives us that the proportion of $\mathcal{Q}_{d,Z_p}^{\mathrm{high}}$, being a subset of
	\[
		\left\{ \sigma\in \ho^0(\X_{p}, \li^{\otimes d})\ ;\ \exists x\in |\mathrm{div}\sigma\cap Z_p|\text{ of degree}\geq \frac{d}{nN'(p)}, \ \dim_{\kappa(x)}\frac{\m_{\mathrm{div}\sigma,x}}{\m^2_{\mathrm{div}\sigma, x}}=n \right\}
	\]
	is bounded by 
	\begin{eqnarray*}
		& &1-\left( \prod_{i=0}^{n}( 1-O(d^i\cdot p^{2-\frac{d}{(N'_0+1)N'_1p}}) )  \right) \cdot \left( 1-O(d^{n}p^{-\frac{d}{nN'(p)}}) \right) \\
		&=&1-\left( 1-O(d^{n}\cdot p^{2-\frac{d}{(N'_0+1)N'_1p}} )\right)\cdot  \left( 1-O(d^{n}p^{-\frac{d}{nN'(p)}}) \right) \\
		&=&O\left( d^{n}p^{-c'_1\frac{d}{p}} \right).
	\end{eqnarray*}
	Thus Proposition \ref{hising} is proved.
\end{proof}

\subsection{Proof of Proposition \ref{fib}}\label{step4}

\begin{proof}
	As in the previous section, let $N_0$ be an integer satisfying the following conditions:
	\begin{enumerate}[label=\arabic*)]
		\item $N_0+1$ is a power of a prime number ;
		\item for any $d\geq N_0$, $(\li|_{\X_\mathcal{S}})^{\otimes d}$ is relatively very ample;
		\item for any $a,b\geq N_0$, we have a surjective morphism
		\[
			\ho^0(\X_{\mathcal{S}}, \li^{\otimes a}) \otimes_{\ho^0(\mathcal{S}, \mathcal{O}_{\mathcal{S}})}\ho^0(\X_{\mathcal{S}}, \li^{\otimes b})\longrightarrow \ho^0(\X_{\mathcal{S}}, \li^{\otimes (a+b)}).
		\]
	\end{enumerate}
	
	Let $N_1$ be as in Proposition \ref{high}.
	For each $p\in \mathcal{S}$, take 
	\[
		N(p)=(N_0+1)(N_1+p-1)+p=p(N_0+2)+(N_0+1)(N_1-1).
	\]
	In particular, $N(p)$ also satisfies the conditions (2) and (3) above.
	By Lemma \ref{smalldegree}, for any positive integer $r$ which satisfies $2c_0N(p)nrp^{(n-1)r}\leq d$, we have 
	\[
		\frac{\#\mathcal{P}'_{d,p^2,\leq r}}{\#\ho^0(\X_{p}, \li^{\otimes d})}=\prod_{x\in |\X_{p}|, \deg x\leq r} \left( 1-p^{-(n+1)\deg x} \right).
	\] 
	Let $r_{p,d}$ be the largest $r$ satisfying this condition. In order to have $r_{p,d}\geq 1$, we need 
	\[
		2c_0N(p)np^{n-1}\leq d.
	\]
	Set $C=2c_0(N_0+3)$. When $d$ is larger than $2c_0n(N_0+3)(N_0+1)^n(N_1-1)^n$, if $p$ satisfies $Cnp^n<d$, then either $p<(N_0+1)(N_1-1)$, in which case we have
	\begin{eqnarray*}
		2c_0N(p)np^{n-1} &=& 2c_0\Big[ p(N_0+2)+(N_0+1)(N_1-1)\Big]np^{n-1} \\
		&<&2c_0\Big[ (N_0+2)(N_0+1)(N_1-1)+(N_0+1)(N_1-1)\Big]np^{n-1}\\
		&=&2c_0(N_0+3)(N_0+1)(N_1-1)np^{n-1}\\
		&<&2c_0n(N_0+3)(N_0+1)^n(N_1-1)^n\leq d ;
	\end{eqnarray*}
	or $(N_0+1)(N_1-1)\leq p <Cnp^n<d$, so that
	\begin{eqnarray*}
		2c_0N(p)np^{n-1}&=& 2c_0[p(N_0+2)+(N_0+1)(N_1-1)]np^{n-1} \\
		&\leq& 2c_0(N_0+3)pnp^{n-1}\\
		&=&Cnp^n<d.
	\end{eqnarray*}
	So the above condition is satisfied, hence $r_{p,d}\geq 1$.

	Since
	\[
		\mathcal{P}'_{d,p^2}\subset \mathcal{P}'_{d,p^2,\leq r_{p,d}}\subset \mathcal{P}'_{d,p^2}\cup \mathcal{Q}^{\mathrm{med}}_{d,p^2,r_{p,d}}\cup \mathcal{Q}^{\mathrm{high}}_{d,p^2},
	\]
	we have
	\begin{eqnarray*}
		&&\left| \frac{\#\mathcal{P}'_{d,p^2}}{\#\ho^0(\X_{p^2}, \li^{\otimes d})}-\zeta_{\X_p}(n+1)^{-1} \right|\\
		&\leq& \left| \frac{\#\mathcal{P}'_{d,p^2}}{\#\ho^0(\X_{p^2}, \li^{\otimes d})}-\frac{\#\mathcal{P}'_{d,p^2,\leq r_{p,d}}}{\#\ho^0(\X_{p^2}, \li^{\otimes d})} \right| +\left| \frac{\#\mathcal{P}'_{d,p^2,\leq r_{p,d}}}{\#\ho^0(\X_{p^2}, \li^{\otimes d})}-\zeta_{\X_p}(n+1)^{-1} \right| \\
		&\leq & \frac{\#\mathcal{Q}^{\mathrm{med}}_{d,p^2,r_{p,d}}}{\#\ho^0(\X_{p^2}, \li^{\otimes d})} +\frac{\#\mathcal{Q}^{\mathrm{high}}_{d,p^2}}{\#\ho^0(\X_{p^2}, \li^{\otimes d})} + \left| \zeta_{\X_p}(n+1)^{-1}-\prod_{x\in |\X_{p^2}|, \deg x\leq r_{p,d}} \left( 1-p^{-(n+1)\deg x} \right) \right|.
	\end{eqnarray*}
	By Lemma \ref{medi},
	\[
		 \frac{\#\mathcal{Q}^{\mathrm{med}}_{d,p^2,r_{p,d}}}{\#\ho^0(\X_{p^2}, \li^{\otimes d})}<2c_0p^{-2(r_{p,d}+1)}.
	\]
	$ $

	By the choice of $r_{p,d}$, we have
	\[
		2c_0N(p)n(r_{p,d}+1) \cdot p^{(n-1)(r_{p,d}+1)}>d.
	\]
	So
	\begin{eqnarray*}
		p^{-(r_{p,d}+1)}&=&\left( p^{n(r_{p,d}+1)} \right)^{-\frac{1}{n}}\\
		&<&\left( (r_{p,d}+1)p^{(n-1)(r_{p,d}+1)} \right)^{-\frac{1}{n}}\\
		&<& \left(\frac{d}{2c_0N(p)}\right)^{-\frac{1}{n}}\\
		&=&O\left( \Big(\frac{d}{p}\Big)^{-\frac{1}{n}} \right).
	\end{eqnarray*}
	Therefore we have 
	\[
		 \frac{\#\mathcal{Q}^{\mathrm{med}}_{d,p^2,r_{p,d}}}{\#\ho^0(\X_{p^2}, \li^{\otimes d})}=O\left( \Big(\frac{d}{p}\Big)^{-\frac{2}{n}} \right),
	\]
	where the coefficient involved in is independent of $d,p$.

	By Proposition \ref{high}, we have
	\[
		\frac{\#\mathcal{Q}_{d,p^2}^{\mathrm{high}}}{\#\ho^0(\X_{p^2}, \li^{\otimes d})}=O\left( d^{n}p^{-c_1\frac{d}{p}} \right),
	\]
	where again the coefficient involved in is independent of $d,p$.

	Note that Lemma \ref{zetafinite} shows
	\[
		\left| \prod_{x\in |\X_{p}|, \deg x\leq r_{p,d}} \left( 1-p^{-(n+1)\deg x} \right) - \zeta_{\X_p}(n+1)^{-1} \right|\leq 4c_0 p^{-2(r_{p,d}+1)}=O\left( \Big(\frac{d}{p}\Big)^{-\frac{2}{n}}\right).
	\]
	Finally, by putting together all these three inequalities, we get
	\begin{eqnarray*}
		& & \left| \frac{\#\mathcal{P}'_{d,p^2}}{\#\ho^0(\X_{p^2}, \li^{\otimes d})}-\zeta_{\X_p}(n+1)^{-1} \right|\\
		&\leq&\frac{\#\mathcal{Q}^{\mathrm{med}}_{d,p^2,r_{p,d}}}{\#\ho^0(\X_{p^2}, \li^{\otimes d})} +\frac{\#\mathcal{Q}^{\mathrm{high}}_{d,p^2}}{\#\ho^0(\X_{p^2}, \li^{\otimes d})} + \left| \zeta_{\X_p}(n+1)^{-1}-\prod_{x\in |\X_{p^2}|, \deg x\leq r_{p,d}} \left( 1-p^{-(n+1)\deg x} \right) \right| \\
		&=& O\left( \Big(\frac{d}{p}\Big)^{-\frac{2}{n}} \right)+O\left( d^{n}p^{-c_1\frac{d}{p}} \right)+O\left( \Big(\frac{d}{p}\Big)^{-\frac{2}{n}}\right)\\
		&=& O\left( \Big(\frac{d}{p}\Big)^{-\frac{2}{n}}\right),
	\end{eqnarray*}
	where the coefficient involved in is independent of $d,p$, which is what we need to show.
\end{proof}

\section{Singular points of small residual characteristic}\label{smallresidual}

In this section, we will show the following result : 

\begin{prop}\label{small}
	Let $\X$ be a regular projective arithmetic variety of absolute dimension $n$, and let $\Li$ be an ample Hermitian line bundle on $\X$. Set
	\[
		\mathcal{P}_{d, p\leq d^{\frac{1}{n+1}}}:=\left\{ \sigma\in \ho^0(\X,\Li^{\otimes d})\ ;\ \begin{array}{ll}
		\mathrm{div}\sigma \text{ has no singular point of residual} \\
		\text{characteristic smaller than or equal to } d^{\frac{1}{n+1}}
		\end{array}\right\}.
	\]
	When $d$ is sufficiently large, we have
	\[
		\left|  \frac{\#\left(\mathcal{P}_{d, p\leq d^{\frac{1}{n+1}}}\cap \ho^0_{\mathrm{Ar}}(\X, \overline{\mathcal{L}}^{\otimes d})\right)}{\#\ho^0_{\mathrm{Ar}}(\X, \overline{\mathcal{L}}^{\otimes d})} -\zeta_\X(n+1)^{-1} \right| = O(d^{-\frac{1}{n+1} }).
	\]
	Here the constant involved in the big $O$ depends only on $\X$.
	
	In particular, denoting $\mathcal{P}_B=\bigcup_{d>0}\mathcal{P}_{d, p\leq d^{\frac{1}{n+1}}}$, we have
	\[
		\mu_{\mathrm{Ar}}(\mathcal{P}_B)=\zeta_\X(n+1)^{-1}.
	\]
\end{prop}

\subsection{Union of a finite number of fibers}

Let $p,q$ be two different prime numbers. We have
\[
	\X_{p^2q^2} =\X\times_{\Spec\ \Z}\Spec(\Z/p^2q^2\Z)\simeq \X_{p^2}\amalg \X_{q^2}.
\]
For any $d\geq 0$, we have an isomorphism
\[
	\lambda_{p^2q^2}: \ho^0(\X_{p^2q^2},\li^{\otimes d}) \stackrel{\sim}\longrightarrow  \ho^0(\X_{p^2},\li^{\otimes d}) \times  \ho^0(\X_{q^2},\li^{\otimes d}).
\]
For any $\sigma \in \ho^0_{\mathrm{Ar}}(\X, \overline{\mathcal{L}}^{\otimes d})$, $\sigma\in \mathcal{P}_{p,d}\cap \mathcal{P}_{q,d}$ if and only if the restriction map
\[
	\psi_{d,p^2q^2}: \ho^0_{\mathrm{Ar}}(\X,\Li^{\otimes d}) \longrightarrow \ho^0(\X_{p^2q^2},\li^{\otimes d})
\]
sends $\sigma$ to an element in the set $\lambda^{-1}_{p^2q^2}(\mathcal{P}'_{p^2,d}\times \mathcal{P}'_{q^2,d})$. Therefore applying Proposition \ref{modN}, we have
\[
	\lim_{d\rightarrow \infty}\frac{\#\left(\mathcal{P}_{p,d}\cap \mathcal{P}_{q,d}\cap \ho^0_{\mathrm{Ar}}(\X, \overline{\mathcal{L}}^{\otimes d})\right)}{\#\ho^0_{\mathrm{Ar}}(\X, \overline{\mathcal{L}}^{\otimes d})}=\zeta_{\X_p}(n+1)^{-1}\zeta_{\X_q}(n+1)^{-1}.
\]
More generally, for any finite set $I$ of prime numbers $p$, we have
\[
	\lim_{d\rightarrow \infty}\frac{\#\left(\bigcap_{p\in I}\mathcal{P}_{p,d}\cap \ho^0_{\mathrm{Ar}}(\X, \overline{\mathcal{L}}^{\otimes d})\right)}{\#\ho^0_{\mathrm{Ar}}(\X, \overline{\mathcal{L}}^{\otimes d})}=\prod_{p\in I}\zeta_{\X_p}(n+1)^{-1}.
\]

By Lemma \ref{freemod}, we may only consider $d>0$ such that for any positive integer $N$, we have
\[
	\ho^0(\X_N, \li^{\otimes d})\simeq \ho^0(\X, \li^{\otimes d})/(N\cdot\ho^0(\X, \li^{\otimes d})).
\]
Fix a positive constant $\alpha_0$ such that $\frac{3}{4}<\alpha_0<1$. By Proposition \ref{modN}, when $d$ is large enough, for any $N<e^{d^{\alpha_0}}$, the map
\[
	\psi_{d,N}: \ho^0_{\mathrm{Ar}}(\X,\Li^{\otimes d}) \longrightarrow \ho^0(\X_N,\li^{\otimes d})
\]
is surjective and there exists a positive constant $\eta$ with 
\[
	\frac{| \#\psi^{-1}(\sigma)- \#\psi^{-1}(\sigma') |}{\#\psi^{-1}(\sigma)}\leq e^{-\eta d}
\]
for any two sections $\sigma,\sigma'$ in $ \ho^0(\X_N,\li^{\otimes d})$.

For a positive integer $r$, take $N_r=\prod_{p\leq r}p^2$. 
\begin{lem}\label{many}
	Let $C$ be the constant in Theorem \ref{fiber}.For any large enough integer $d$, and for any integer $r$ satisfying $Cnr^n<d$ with $n=\dim \X$ and $N_r<e^{d^{\alpha_0}}$, we have
	\[
		\left|  \frac{\#\left( \bigcap_{p\leq r}\mathcal{P}_{d,p}\cap \ho^0_{\mathrm{Ar}}(\X, \overline{\mathcal{L}}^{\otimes d})\right)}{\#\ho^0_{\mathrm{Ar}}(\X, \overline{\mathcal{L}}^{\otimes d})} -\prod_{p\leq r}\zeta_{\X_p}(n+1)^{-1}  \right| =O\left( \frac{\sum_{p\leq r}p^{\frac{2}{n}}}{d^{\frac{2}{n}}} \right).
	\]
\end{lem}
\begin{proof}
	The Chinese remainder theorem implies that
	\begin{eqnarray*}
		\ho^0(\X_{N_r}, \li^{\otimes d}) &\simeq& \ho^0(\X, \li^{\otimes d})\otimes_{\mathbb{Z}}\mathbb{Z}/N_r\mathbb{Z} \\
		& \simeq& \ho^0(\X, \li^{\otimes d})\otimes_{\mathbb{Z}}\left( \prod_{p\leq r}\mathbb{Z}/p^2\mathbb{Z}\right)\\
		&\simeq & \prod_{p\leq r}\ho^0(\X_{p^2}, \li^{\otimes d}).
	\end{eqnarray*}
	Moreover we have $\X_{N_r}=\coprod_{p\leq r}\X_{p^2}$. Set
	\[
		E_{d,r}:=\{ \sigma\in \ho^0(\X_{N_r}, \li^{\otimes d})\ ;\ \forall x\in |\mathrm{div}\sigma|,\ \dim_{\kappa(x)}\frac{\m_{\mathrm{div}\sigma,x}}{\m^2_{\mathrm{div}\sigma, x}}=n-1 \}.
	\]
	Then a section $\sigma\in \ho^0(\X_{N_r}, \li^{\otimes d})$ is contained in $E_{d,r}$ if and only if for any $p\leq r$, its restriction $\sigma|_{\X_{p^2}}$ is contained in $\mathcal{P}'_{p^2,d}$. In particular, a section $\sigma \in \ho^0(\X, \li^{\otimes d})$ satisfies $\psi_{d,N_r}(\sigma)\in E_{d,r}$ if and only if $\psi_{d,p^2}(\sigma)\in \mathcal{P}'_{p^2,d}$, for all $p\leq r$, which means exactly that this $\sigma$ is contained in $ \bigcap_{p\leq r}\mathcal{P}_{d,p}$. On the other hand, still by the Chinese remainder theorem, 
	\[
		\frac{\#E_{d,r}}{\#\ho^0(\X_{N_r}, \li^{\otimes d})}=\prod_{p\leq r} \frac{\#\mathcal{P}'_{p^2,d}}{\#\ho^0(\X_{p^2}, \li^{\otimes d} )}.
	\]

	As said above, with some positive constant $\eta$, we have
	\[
		\frac{| \#\psi^{-1}(\sigma)- \#\psi^{-1}(\sigma') |}{\#\psi^{-1}(\sigma)}\leq e^{-\eta d},
	\]
	with $\sigma,\sigma'$ in $ \ho^0(\X_{N_r},\li^{\otimes d})$. Fixing one $\sigma$, we can sum up for all $\sigma'\in E_{d,r}$ and get
	\begin{eqnarray*}
		& &\left| \#E_{d,r}-\frac{\#\left(\bigcap_{p\leq r}\mathcal{P}_{d,p}\right)}{\#\psi^{-1}(\sigma)} \right| \\
		&=&\frac{| \left(\#E_{d,r}\cdot \#\psi^{-1}(\sigma)\right)- \#\psi^{-1}(E_{d,r}) |}{\#\psi^{-1}(\sigma)} \\
		&\leq& \sum_{\sigma'\in E_{d,r}}\frac{| \#\psi^{-1}(\sigma)- \#\psi^{-1}(\sigma') |}{\#\psi^{-1}(\sigma)} \\
		&\leq& \#E_{d,r}\cdot e^{-\eta d},
	\end{eqnarray*}
	where the last inequality follows from Proposition \ref{modN}. This can also be written as
	\[
		\left| \#\left(\bigcap_{p\leq r}\mathcal{P}_{d,p}\cap \ho^0_{\mathrm{Ar}}(\X, \overline{\mathcal{L}}^{\otimes d})\right)-\left(\#E_{d,r}\cdot\#\psi^{-1}(\sigma)\right) \right|\leq \left(\#E_{d,r}\cdot\#\psi^{-1}(\sigma)\right) e^{-\eta d}.
	\]
	Now we take the sum for all $\sigma \in \ho^0(\X_{N_r}, \li^{\otimes d})$ and get
	\begin{eqnarray*}
		& &\left| \left(\#\left(\bigcap_{p\leq r}\mathcal{P}_{d,p}\cap \ho^0_{\mathrm{Ar}}(\X, \Li^{\otimes d})\right)\right)\cdot \#\ho^0(\X_{N_r}, \li^{\otimes d}) -\left(\#E_{d,r}\cdot\#\ho^0_{\mathrm{Ar}}(\X, \Li^{\otimes d})\right) \right|  \\
		&=& \left| \sum_{\sigma \in \ho^0(\X_{N_r}, \li^{\otimes d})}\left( \#\left(\bigcap_{p\leq r}\mathcal{P}_{d,p}\cap \ho^0_{\mathrm{Ar}}(\X, \Li^{\otimes d})\right)-\left(\#E_{d,r}\cdot\#\psi^{-1}(\sigma)\right) \right)\right| \\
		&\leq& \sum_{\sigma \in \ho^0(\X_{N_r}, \li^{\otimes d})}\left| \#\left(\bigcap_{p\leq r}\mathcal{P}_{d,p}\cap \ho^0_{\mathrm{Ar}}(\X, \Li^{\otimes d}) \right)-\left(\#E_{d,r}\cdot\#\psi^{-1}(\sigma)\right) \right| \\
		&\leq& \sum_{\sigma \in \ho^0(\X_{N_r}, \li^{\otimes d})}\left(\#E_{d,r}\cdot\#\psi^{-1}(\sigma)\right) e^{-\eta d}\\
		&=& \left(\#E_{d,r}\cdot\#\ho^0_{\mathrm{Ar}}(\X, \Li^{\otimes d})\right) e^{-\eta d} \\
		&\leq& \left(\#\ho^0(\X_{N_r}, \li^{\otimes d})\cdot\#\ho^0_{\mathrm{Ar}}(\X, \Li^{\otimes d})\right) e^{-\eta d}.
	\end{eqnarray*}
	Dividing both side by $\#\ho^0(\X_{N_r}, \li^{\otimes d})\cdot\#\ho^0_{\mathrm{Ar}}(\X, \Li^{\otimes d})$, we get
	\[
		\left| \frac{\#\left(\bigcap_{p\leq r}\mathcal{P}_{d,p}\cap \ho^0_{\mathrm{Ar}}(\X, \Li^{\otimes d})\right)}{\#\ho^0_{\mathrm{Ar}}(\X, \Li^{\otimes d})} -\frac{\#E_{d,r}}{\#\ho^0(\X_{N_r}, \li^{\otimes d})} \right|\leq e^{-\eta d}.
	\]
	Since we already know that 
	\[
		\frac{\#E_{d,r}}{\#\ho^0(\X_{N_r}, \li^{\otimes d})}=\prod_{p\leq r} \frac{\#\mathcal{P}'_{p^2,d}}{\#\ho^0(\X_{p^2}, \li^{\otimes d} )},
	\]
	the inequality can be written as
	\[
		\left| \frac{\#\left(\bigcap_{p\leq r}\mathcal{P}_{d,p}\cap \ho^0_{\mathrm{Ar}}(\X, \Li^{\otimes d})\right)}{\#\ho^0_{\mathrm{Ar}}(\X, \Li^{\otimes d})} -\prod_{p\leq r} \frac{\#\mathcal{P}'_{p^2,d}}{\#\ho^0(\X_{p^2}, \li^{\otimes d} )} \right|\leq e^{-\eta d}.
	\]
	
	Thus to finish the proof, it suffices to show the following lemma : 
	
	\begin{lem}\label{manymany}
	Under the condition of Lemma \ref{many}, we have
	\[
		\left| \prod_{p\leq r} \frac{\#\mathcal{P}'_{p^2,d}}{\#\ho^0(\X_{p^2}, \li^{\otimes d} )} - \prod_{p\leq r} \zeta_{\X_p}(1+n)^{-1} \right| =O\left( \frac{\sum_{p\leq r}p^{\frac{2}{n}}}{d^{\frac{2}{n}}} \right).
	\]
	\end{lem}
	\begin{proof}
		By the Lemma \ref{fib}, for any prime number $p$ satisfying $Cnp^n<d$, we have
		\[
			\frac{\#\mathcal{P}'_{p^2,d}}{\#\ho^0(\X_{p^2}, \li^{\otimes d} )}=\zeta_{\X_p}(1+n)^{-1} + O\left( \Big(\frac{d}{p}\Big)^{-\frac{2}{n}} \right)
		\]	
		with the constant involved in big $O$ independent of $p$ and $d$. Therefore we can calculate the product as
		\begin{eqnarray*}
			\prod_{p\leq r} \frac{\#\mathcal{P}'_{p^2,d}}{\#\ho^0(\X_{p^2}, \li^{\otimes d} )} &=& \prod_{p\leq r} \left( \zeta_{\X_p}(1+n)^{-1} + O\left( \Big(\frac{d}{p}\Big)^{-\frac{2}{n}} \right) \right) \\
			&=& \prod_{p\leq r} \zeta_{\X_p}(1+n)^{-1} +\sum_{p\leq r} O\left( \Big(\frac{p}{d}\Big)^{\frac{2}{n}} \right) \\
			&=& \prod_{p\leq r} \zeta_{\X_p}(1+n)^{-1} + O\left( \frac{\sum_{p\leq r} p^{\frac{2}{n}}}{d^{\frac{2}{n}}} \right).
		\end{eqnarray*}
	\end{proof}
	The above lemma shows the result, as
	\begin{eqnarray*}
		& &\left|  \frac{\#\left( \bigcap_{p\leq r}\mathcal{P}_{d,p}\cap \ho^0_{\mathrm{Ar}}(\X, \overline{\mathcal{L}}^{\otimes d})\right)}{\#\ho^0_{\mathrm{Ar}}(\X, \overline{\mathcal{L}}^{\otimes d})} -\prod_{p\leq r}\zeta_{\X_p}(n+1)^{-1}  \right| \\
		&\leq& \left| \frac{\#\left(\bigcap_{p\leq r}\mathcal{P}_{d,p}\cap \ho^0_{\mathrm{Ar}}(\X, \overline{\mathcal{L}}^{\otimes d})\right)}{\#\ho^0_{\mathrm{Ar}}(\X, \overline{\mathcal{L}}^{\otimes d})} -\prod_{p\leq r} \frac{\#\mathcal{P}'_{p^2,d}}{\#\ho^0(\X_{p^2}, \li^{\otimes d} )} \right| \\
		& &\qquad\qquad\qquad\qquad+\left| \prod_{p\leq r} \frac{\#\mathcal{P}'_{p^2,d}}{\#\ho^0(\X_{p^2}, \li^{\otimes d} )} - \prod_{p\leq r} \zeta_{\X_p}(1+n)^{-1} \right|  \\
		&=&O\left( e^{-\eta d} \right)+O\left( \frac{\sum_{p\leq r}p^{\frac{2}{n}}}{d^{\frac{2}{n}}} \right)\\
		&=&O\left( \frac{\sum_{p\leq r}p^{\frac{2}{n}}}{d^{\frac{2}{n}}} \right).
	\end{eqnarray*}
\end{proof}

\subsection{Bound on number of fibers}

Now we prove that we can choose $r=d^{\frac{1}{n+1}}$. As 
\[
	\mathcal{P}_{d, p\leq d^{\frac{1}{n+1}}}=\bigcap_{p\leq d^{\frac{1}{n+1}}}\mathcal{P}_{d,p},
\]
we will in fact prove the following :
\begin{lem}\label{br}
	For large enough integer $d$, we have
	\[
		\left|  \frac{\#\left( \mathcal{P}_{d, p\leq d^{\frac{1}{n+1}}}\cap \ho^0_{\mathrm{Ar}}(\X, \overline{\mathcal{L}}^{\otimes d})\right)}{\#\ho^0_{\mathrm{Ar}}(\X, \overline{\mathcal{L}}^{\otimes d})} -\prod_{p\leq d^{\frac{1}{n+1}}}\zeta_{\X_p}(n+1)^{-1}  \right|=O(d^{-\frac{1}{n+1} }).
	\]
\end{lem}
\begin{proof}
	Note that 
	\[
		N_r=\prod_{p\leq r}p^2\leq \prod_{k=1}^r k^2=(r!)^2
	\]
	and that $r!< r^r=\exp(r\log r)$. If $2r\log r<d^{\alpha_0}$, we get $N_r< \exp(d^{\alpha_0})$. Then by Proposition \ref{modN} the restriction morphism
	\[
		\psi_{d,N_r}: \ho^0_{\mathrm{Ar}}(\X,\Li^{\otimes d}) \longrightarrow \ho^0(\X_{N_r},\li^{\otimes d})
	\]
	is surjective. If moreover $r$ satisfies $Cnr^n<d$ as in Theorem \ref{fiber}, then by Lemma \ref{many}, we have
	\[
		\left|  \frac{\#\left( \bigcap_{p\leq r}\mathcal{P}_{d,p}\cap \ho^0_{\mathrm{Ar}}(\X, \Li^{\otimes d})\right)}{\#\ho^0_{\mathrm{Ar}}(\X, \Li^{\otimes d})} -\prod_{p\leq r}\zeta_{\X_p}(n+1)^{-1}  \right| =O\left( \frac{\sum_{p\leq r}p^{\frac{2}{n}}}{d^{\frac{2}{n}}} \right).
	\]
	
	Now as above, 
	\begin{eqnarray*}
		\sum_{p\leq r}p^{\frac{2}{n}} \leq \sum_{k=1}^r k^{\frac{2}{n}} < r\cdot  r^{\frac{2}{n}}= r^{\frac{n+2}{n}}.
	\end{eqnarray*}
	Thus for $r=d^{\frac{1}{n+1}}$ we have
	\[
		\sum_{p\leq d^{\frac{1}{n+1}}}p^{\frac{2}{n}}<d^{\frac{n+2}{n(n+1)}}=O(d^{\frac{n+2}{n(n+1)}}).
	\]
	It's easy to see that $r=d^{\frac{1}{n+1}}$ also satisfies conditions $2r\log r<d^{\alpha_0}$, $Cnr^n<d$ for large $d$. For this $r$, we have
	\begin{eqnarray*}
		\left|  \frac{\#\left( \bigcap_{p\leq d^{\frac{1}{n+1}}}\mathcal{P}_{d,p}\cap \ho^0_{\mathrm{Ar}}(\X, \overline{\mathcal{L}}^{\otimes d})\right)}{\#\ho^0_{\mathrm{Ar}}(\X, \overline{\mathcal{L}}^{\otimes d})} -\prod_{p\leq d^{\frac{1}{n+1}}}\zeta_{\X_p}(n+1)^{-1}  \right| &=&O\left( \frac{\sum_{p\leq d^{\frac{1}{n+1}}}p^{\frac{2}{n}}}{d^{\frac{2}{n}}} \right)\\
		&=&O(d^{\frac{n+2}{n(n+1)}-\frac{2}{n}})=O(d^{-\frac{1}{n+1} }).
	\end{eqnarray*}
\end{proof}

\subsection{Proof of Proposition \ref{small}}
	
\begin{proof}
	Apply Lemma \ref{zetaintegral} and take $R=d^{\frac{1}{n+1}}$. We get
	\begin{eqnarray*}
		\left| \prod_{p\leq d^{\frac{1}{n+1}}}\zeta_{\X_p}(n+1)^{-1}-\zeta_\X(n+1)^{-1} \right| =O(d^{-\frac{1}{n+1}}).
	\end{eqnarray*}
	Combining this with Lemma \ref{br}, we get
	\begin{eqnarray*}
		& &\left|  \frac{\#\left( \mathcal{P}_{d, p\leq d^{\frac{1}{n+1}}}\cap \ho^0_{\mathrm{Ar}}(\X, \overline{\mathcal{L}}^{\otimes d})\right)}{\#\ho^0_{\mathrm{Ar}}(\X, \overline{\mathcal{L}}^{\otimes d})} -\zeta_\X(n+1)^{-1} \right| \\
		&\leq &\left|  \frac{\#\left( \mathcal{P}_{d, p\leq d^{\frac{1}{n+1}}}\cap \ho^0_{\mathrm{Ar}}(\X, \overline{\mathcal{L}}^{\otimes d})\right)}{\#\ho^0_{\mathrm{Ar}}(\X, \overline{\mathcal{L}}^{\otimes d})} -\prod_{p\leq d^{\frac{1}{n+1}}}\zeta_{\X_p}(n+1)^{-1}  \right|\\
		& &\qquad\qquad\qquad\qquad\qquad+\left| \prod_{p\leq d^{\frac{1}{n+1}} }\zeta_{\X_p}(n+1)^{-1}-\zeta_\X(n+1)^{-1} \right| \\
		&=&O(d^{-\frac{1}{n+1} })+O(d^{-\frac{1}{n+1} })=O(d^{-\frac{1}{n+1} }),
	\end{eqnarray*}
	which proves Proposition \ref{small}.
\end{proof}

\section{Final step}\label{finalresult}

In this section, we prove Theorem \ref{mainvar}. The main step is to show the following proposition :
\begin{prop}\label{singmed}
	Let $\X$ be a regular projective arithmetic variety of dimension $n$, and let $\li$ be an ample line bundle on $\X$. Then there exists a constant $c>0$ such that for any $d\gg 1$ and any prime number $p$ such that $\X_p$ is smooth and irreducible, denoting
	\[
		\mathcal{Q}_{d,p^2}:=\left\{ \sigma\in \ho^0(\X_{p^2}, \li^{\otimes d})\ ;\ \exists x\in |\X_{p^2}|,\ \dim_{\kappa(x)}\frac{\mathfrak{m}_{\mathrm{div}\sigma,x}}{\mathfrak{m}_{\mathrm{div}\sigma,x}^2}=n \right\},
	\]
	we have
	\[
		\frac{\#\mathcal{Q}_{d,p^2}}{\#\ho^0(\X_{p^2}, \li^{\otimes d})}\leq c\cdot p^{-2}.
	\]
\end{prop}

\subsection{Divisors with higher dimensional singular locus}

\begin{lem}\label{singnum}
	Let $\X$ be an irreducible projective scheme of dimension $n$ over $\Spec\ \Z$. Let $\li$ be an ample line bundle on $\X$. For any large enough $d$ and any prime number $p$ such that $\X_p$ is smooth over $\F_p$, if $\sigma\in \ho^0(\X_p, \li^{\otimes d})$ is such that $\mathrm{Sing}(\mathrm{div}\sigma)$ is finite, then 
	\[
		\#\mathrm{Sing}(\mathrm{div}\sigma)=O(d^{n-1}),
	\]
	where the constant involved does not depend on $d$ or $p$. 
\end{lem}

\begin{proof}
	We take the construction in Section \ref{step3}. Let $\mathcal{S}$ be the maximal open subscheme of $\Spec\ \Z$ such that $\X_{\mathcal{S}}=\X\times_{\Spec\ \Z}\mathcal{S}$ is smooth over $\mathcal{S}$. So $\mathcal{S}$ contains all prime numbers $p$ such that $\X_p$ is smooth of dimension $n-1$ over $\F_p$. Applying Lemma \ref{opensub}, we may assume that there exists a positive integer $N$ and an open cover of $\X_{\mathcal{S}}$ by $\X_{\mathcal{S}}=\bigcup_{ \alpha\in A}U_{\alpha}$ making the following conditions valid:
	\begin{enumerate}
		\item the sheaf $\li^{\otimes d}$ is very ample for any $d\geq N$ ; 
		\item there exists $\tau_{\alpha}\in \ho^0(\X_{\mathcal{S}}, \li^{\otimes (N+1)})$ such that
		\[
			\X_{\mathcal{S}}-U_{\alpha}=\mathrm{div}\tau_{\alpha}\ ;
		\]
		\item there exist $\tau_{\alpha,1},\dots, \tau_{\alpha,k_{\alpha}}\in \ho^0(\X_{\mathcal{S}}, \li^{\otimes N})$ such that
		\[
			U_{\alpha}=\bigcup_{1\leq j\leq k_{\alpha}}\left( \X_{\mathcal{S}}-\mathrm{div}\tau_{\alpha}\right);
		\]
		\item for any $\alpha\in A$, there exist $t_{\alpha, 1},\dots,t_{\alpha, n-1}\in \ho^0(U_{\alpha}, \mathcal{O}_{U_{\alpha}})$ such that
		\[
			\Omega^1_{U_{\alpha}/\mathcal{S}}\simeq \bigoplus_{i=1}^{n-1}\mathcal{O}_{U_{\alpha}}\mathrm{d}t_{\alpha,i}.
		\]  
	\end{enumerate}
	We denote by $\partial_{\alpha,i}\in \mathrm{Der}_{\mathcal{O}_{\mathcal{S}}}(\mathcal{O}_{U_{\alpha}},\mathcal{O}_{U_{\alpha}})\simeq \mathrm{Hom}(\Omega^1_{U_{\alpha}/\mathcal{S}}, \mathcal{O}_{U_{\alpha}})$ the dual of $\mathrm{d}t_{\alpha, i}$.

	Now we take one arbitrary $U$ among the $U_{\alpha}$'s in the open cover, and we drop the subscript $\alpha$ for simplicity of notation. For any $1\leq j\leq k$, we have a morphism
	\[
		\Phi_{j}: \ho^0(\X_{\mathcal{S}}, \li^{\otimes d}) \longrightarrow \ho^0(U,\mathcal{O}_{U})
	\]
	sending $\sigma$ to $\frac{\sigma\tau_{j}^{d}}{\tau^{d}}$. Then for any $\sigma\in \ho^0(\X_{\mathcal{S}}, \li^{\otimes d})$, $\partial_{i}\Phi_{j}(\sigma)$ is a section in $\ho^0(U,\mathcal{O}_{U})$. For any $p\in \mathcal{S}$, $\Phi_{j}$ induces 
	\[
		\Phi_{p,j}: \ho^0(\X_{p}, \li^{\otimes d}) \longrightarrow \ho^0(U_p,\mathcal{O}_{U_p})
	\] such that for any $\sigma'\in \ho^0(\X_{p}, \li^{\otimes d})$, $\partial_{i}\Phi_{p,j}(\sigma')$ is a section in $\ho^0(U_p,\mathcal{O}_{U_p})$. Here $U_p=U\cap \X_p$. Then Lemma \ref{extend} tells that we can find a positive integer $N_1$ such that for any $\delta\geq N_1$, any $p\in \mathcal{S}$ and any $\sigma\in \ho^0(\X_{p}, \li^{\otimes d})$,  the section $\left(\partial_i \Phi_{p,j}(\sigma) \right)\cdot \tau^{d+\delta}$ extends to a global section in $\ho^0(\X_p, \li^{\otimes (N+1)(d+\delta)})$. Since $\li^{\otimes (N+1)}$ is very ample, $\left(\partial_i \Phi_{p,j}(\sigma) \right)\cdot \tau^{d+\delta}$ can also be regarded as a global section of 
	\[
		\ho^0\left(\pr(\ho^0(\X_{p}, \li^{\otimes (N+1)})^{\vee}), \mathcal{O}(d+\delta)\right).
	\]
	Then since $\tau$ is nowhere $0$ on $U\supset \left( \X_{\mathcal{S}}-\mathrm{div}\tau_{\alpha,j}\right)$, for any $\sigma\in \ho^0(\X_{p}, \li^{\otimes d})$, we have
	\begin{eqnarray*}
		& &\left( \X_{\mathcal{S}}-\mathrm{div}\tau_{\alpha,j}\right)\cap \mathrm{Sing}(\mathrm{div}\sigma)\\ 
		&\subset&\mathrm{Sing}(\mathrm{div}\Phi_j(\sigma)) \\
		&=&\mathrm{div}\Phi_j(\sigma)\cap \left( \bigcap_{i=1}^{n-1}\mathrm{div}\left(\partial_i \Phi_j(\sigma) \right)\right) \\
		&=& \left( \X_{\mathcal{S}}-\mathrm{div}\tau_{\alpha,j}\right)\cap \mathrm{div}(\sigma^{N+1})\cap \left( \bigcap_{i=1}^{n-1}\mathrm{div}\left(\partial_i \Phi_j(\sigma) \cdot \tau^{d+N_1} \right)\right).
	\end{eqnarray*}
	On the other hand, we have
	\[
		\sigma^{N+1}\in \ho^0(\pr(\ho^0(\X_p, \li^{\otimes (N+1)})^{\vee}), \mathcal{O}(d)).
	\]
	Denote the degree of $\X_p$ as a closed subscheme of $\pr(\ho^0(\X_p, \li^{\otimes (N+1)})^{\vee})$ by $\deg_{\li^{\otimes N+1}}(\X_p)$. If $\sigma\in \ho^0(\X_{p}, \li^{\otimes d})$ is such that $\mathrm{Sing}(\mathrm{div}\sigma)$ is finite, then $\left( \X_{\mathcal{S}}-\mathrm{div}\tau_{\alpha,j}\right)\cap \mathrm{Sing}(\mathrm{div}\sigma)$ is finite and we can find $n-1$ divisors among the $n$ ones appeared in the above intersection such that the intersection of these $n-1$ divisors and $ \X_{\mathcal{S}}-\mathrm{div}\tau_{\alpha,j}$ is finite. Obviously this intersection contains $\left( \X_{\mathcal{S}}-\mathrm{div}\tau_{\alpha,j}\right)\cap \mathrm{Sing}(\mathrm{div}\sigma)$. Applying refined Bézout's theorem  \cite[Theorem 12.3]{Fu84}, we get
	\[
		\#\left( \mathrm{Sing}(\mathrm{div}\sigma)\cap \left( \X_{\mathcal{S}}-\mathrm{div}\tau_{\alpha,j}\right) \right)\leq (\deg_{\li^{\otimes (N+1)}}(\X_p) )(d+N_1)^{n-1}=O(d^{n-1}),
	\]
	where coefficients involved in $O(d^i)$ is independent of $p$ when $d$ is large enough ($\deg_{\li^{\otimes (N+1)}}(\X_p)$ is independent of $p$). Therefore we have
	\begin{eqnarray*}
		\# \mathrm{Sing}(\mathrm{div}\sigma) &\leq& \sum_{\alpha\in A}\sum_{j=1}^{k_{\alpha}} \#\left( \mathrm{Sing}(\mathrm{div}\sigma)\cap \left( \X_{\mathcal{S}}-\mathrm{div}\tau_{\alpha,j}\right) \right)=O(d^{n-1}),
	\end{eqnarray*}
	with coefficients involved in $O(d^{n-1})$ independent of $p$ when $d$ is large enough.
\end{proof}

The following lemma is a generalization of Lemma 5.9 in \cite{Po04}, where Poonen shows that for an integral quasi-projective scheme $X$ generically smooth over $\Z$ equipped with a very ample line bundle inducing an immersion $X\lhook\joinrel\rightarrow \pr^n_{\Z}$ for some $n>0$, if the generic fiber $\overline{X}_{\Q}$ of the Zariski closure $\overline{X}$ of $X$ in $\pr^n_{\Z}$ has at most isolated singular points, then there exists $c>0$ such that if $d,p$ are sufficiently large, then
\[
	\frac{\#\left\{ \sigma\in \ho^0(\pr^n_{p},\mathcal{O}(d)) \ ; \ \dim\Big(\mathrm{Sing}(\mathrm{div}\sigma|_{X_p})\Big)>0 \right\}}{\#\ho^0(\pr^n_{p},\mathcal{O}(d))}<\frac{c}{p^2}.
\]
We prove the same conclusion for the case when $X$ is projective and equipped with an ample line bundle, in place of a very ample line bundle.
\begin{lem}\label{posdivsing}
	Let $\X$ be a integral scheme of dimension $n$ which is projective and generically smooth over $\Spec\ \Z$, and let $\li$ be an ample line bundle on $\X$. Then there exists a constant $c_H>0$ such that for any $d\gg 1$ and any prime number $p$ such that $\X_p$ is smooth and irreducible, we have
	\[
		\frac{\#\left\{ \sigma\in \ho^0(\X_p, \li^{\otimes d})\ ;\  \dim\big(\mathrm{Sing(\mathrm{div}\sigma})\big)>0 \right\}}{\#\ho^0(\X_p, \li^{\otimes d})}\leq c_H\cdot p^{-2}.
	\]
\end{lem}
\begin{rmq}
	If the prime $p$ is fixed, Corollary \ref{hidim} tells us that there exists a constant $c>0$ such that 
	\[
		\frac{\#\left\{ \sigma\in \ho^0(\X_p, \li^{\otimes d})\ ;\  \dim\big(\mathrm{Sing(\mathrm{div}\sigma})\big)>0 \right\}}{\#\ho^0(\X_p, \li^{\otimes d})}=O(d^{n-1}\cdot p^{-c\frac{d}{p}}).
	\]
	When $d$ is sufficiently large, we deduce from it that 
	\[
		\frac{\#\left\{ \sigma\in \ho^0(\X_p, \li^{\otimes d})\ ;\  \dim\big(\mathrm{Sing(\mathrm{div}\sigma})\big)>0 \right\}}{\#\ho^0(\X_p, \li^{\otimes d})}\leq p^{-2}.
	\]
	So Corollary \ref{hidim} gives a better bound on the proportion of sections whose divisor has positive dimensional singular locus. But this bound is well-behaved only when $d$ is much larger than the prime $p$. In this lemma, the bound we give is independent of the choice of $p$. In particular, it is valid even when $p$ is much bigger than $d$. 
\end{rmq}
\begin{proof}
	We choose a constant $N\in \Z_{>0}$ satisfying Lemma \ref{relample}. If $\X$ is of dimension $1$, then $\dim(\X_p)=0$ for any prime $p$ and the conclusion holds automatically. When $\X$ is of dimension $2$, for $p$ such that $\X_p$ is smooth and irreducible of dimension $1$ (which is satisfied for all but finitely many $p$), if $\sigma\in \ho^0(\X_p, \li^{\otimes d})$ is such that $\dim\big(\mathrm{Sing(\mathrm{div}\sigma})\big)>0$, then $\mathrm{Sing(\mathrm{div}\sigma})=\X_p$, which is impossible unless $\sigma=0$. This means in the case of dimension $2$, when $d$ is large enough we always have
	\[
		\frac{\#\left\{ \sigma\in \ho^0(\X_p, \li^{\otimes d})\ ;\  \dim\big(\mathrm{Sing(\mathrm{div}\sigma})\big)>0 \right\}}{\#\ho^0(\X_p, \li^{\otimes d})}=\frac{1}{\#\ho^0(\X_p, \li^{\otimes d})} =p^{-h^0(\X_{\Q},\li^{\otimes d} )}<p^{-2}.
	\]

	So the lemma is true when $\dim \X\leq 2$. We prove the higher dimensional case by induction. Assume that for any scheme $\mathcal{Y}$ of dimension smaller than $n$ which is projective over $\Spec\ \Z$ with $\mathcal{Y}_{\mathcal{S}_{\mathcal{Y}}}$ irreducible and smooth over some open subscheme $\mathcal{S}_{\mathcal{Y}}$ of $\Spec\ \Z$, and which is equipped with an ample line bundle $\mathcal{M}$, there exists a constant $c_{\mathcal{Y},\mathcal{M}}>0$ such that for any $d\gg 1$ and any prime $p$ such that $\mathcal{Y}_p$ is smooth and irreducible of dimension $\dim\mathcal{Y}-1$, we have 
	\[
		\frac{\#\left\{ \sigma\in \ho^0(\mathcal{Y}_p, \mathcal{M}^{\otimes d})\ ;\  \dim\big(\mathrm{Sing(\mathrm{div}\sigma})\big)>0 \right\}}{\#\ho^0(\mathcal{Y}_p, \mathcal{M}^{\otimes d})}\leq c_{\mathcal{Y},\mathcal{M}}\cdot p^{-2}.
	\]
	By the classical Bertini theorem over $\Q$, we can find a section $\sigma_{\mathcal{D}}\in \ho^0(\X_{\Q}, \li^{\otimes N})$ whose divisor $\mathcal D_{\Q}$ is a smooth and irreducible divisor of $\X_{\Q}$. By possibly replacing $\sigma_{\mathcal D}$ by a multiple of it, we may assume that $\sigma_{\mathcal D}$ is in fact a section of $ \ho^0(\X, \li^{\otimes N})$. Then the divisor $\mathcal{D}=\mathrm{div}\sigma_{\mathcal{D}}$ on $\X$ has no singular point on the generic fiber $\X_{\Q}$. Let $\mathcal{S}$ be an open subscheme of $\Spec\ \Z$ such that $\X_{\mathcal{S}}$ is smooth over $\mathcal{S}$ and irreducible of dimension $n$. By restricting to a smaller $\mathcal{S}$ we may assume that $\X_{\mathcal{S}}$ does not contain any vertical component of $\mathcal{D}$, and that $\mathcal{D}_{\mathcal{S}}=\mathcal{D}\cap \X_{\mathcal{S}}$ is smooth over $\mathcal{S}$. We may assume moreover that for any $p\in \mathcal{S}$, $\X_p$ and $\mathcal{D}_p$ are both smooth and irreducible such that $\dim\X_p=n-1$ and $\dim\mathcal{D}_p=n-2$.  For the rest of the proof, we fix the divisor $\mathcal{D}$. Note that $\mathcal{D}$ together with the open subscheme $\mathcal{S}$ of $\Spec\ \Z$ and the restriction sheaf $\li|_{\mathcal{D}}$ also satisfies the assumption of the lemma.

	For any prime $p\in \Spec\ \Z-\mathcal{S}$ such that $\X_p$ is smooth and irreducible of dimension $n-1$, we can find a constant $c$ by Corollary \ref{hidim}, such that 
	\[
		\frac{\#\left\{ \sigma\in \ho^0(\X_p, \li^{\otimes d})\ ;\  \dim\big(\mathrm{Sing(\mathrm{div}\sigma})\big)>0 \right\}}{\#\ho^0(\X_p, \li^{\otimes d})}=O(d^{n-1}\cdot p^{-c\frac{d}{p}}).
	\]
	So when $d$ is sufficiently large, the right side can be bounded above by $p^{-2}$. Since $\Spec\ \Z-\mathcal{S}$ is a finite scheme, when $d$ is sufficiently large, for any $p\in \Spec\ \Z-\mathcal{S}$ such that $\X_p$ is smooth and irreducible, we have
	\[
		\frac{\#\left\{ \sigma\in \ho^0(\X_p, \li^{\otimes d})\ ;\  \dim\big(\mathrm{Sing(\mathrm{div}\sigma})\big)>0 \right\}}{\#\ho^0(\X_p, \li^{\otimes d})}\leq p^{-2}.
	\]
	Hence it suffices to prove the lemma for primes $p\in \mathcal{S}$. 
	\\
	
	Now let $p\in \mathcal{S}$. If a section $\sigma\in \ho^0(\X_p, \li^{\otimes d})$ is such that
	\[
		\dim\big( \mathrm{Sing}(\mathrm{div}\sigma) \big)>0,
	\]
	then as $\X_p$ is irreducible and projective by the assumption on $\mathcal{S}$, we have
	\[
		\mathrm{Sing}(\mathrm{div}\sigma)\cap \mathcal{D}_p\not=\emptyset.
	\]
	By induction hypothesis, we know that there exists a constant $c_\mathcal{D}>0$ such that if $d$ is sufficiently large, then for any $p\in \mathcal{S}$, we have
	\[
		\frac{\#\left\{ \sigma\in \ho^0(\mathcal{D}_p, \li^{\otimes d})\ ;\  \dim\big(\mathrm{Sing(\mathrm{div}\sigma})\big)>0 \right\}}{\#\ho^0(\mathcal{D}_p, \li^{\otimes d})}\leq c_\mathcal{D}\cdot p^{-2}.
	\]
	As $\li$ is ample on $\X$, when $d$ is large enough, the restriction map
	\[
		\ho^0(\X, \li^{\otimes d}) \longrightarrow  \ho^0(\mathcal{D}, \li^{\otimes d})
	\]
	is surjective. So for such $d$, the morphism
	\[
		\ho^0(\X_p, \li^{\otimes d}) \longrightarrow  \ho^0(\mathcal{D}_p, \li^{\otimes d})
	\]
	is surjective for any $p\in \mathcal{S}$, and hence
	\[
		\frac{\#\left\{ \sigma\in \ho^0(\X_p, \li^{\otimes d})\ ;\  \dim\big(\mathrm{Sing(\mathrm{div}}\sigma\cap \mathcal{D}_p)\big)>0 \right\}}{\#\ho^0(\X_p, \li^{\otimes d})}\leq c_\mathcal{D}\cdot p^{-2}.
	\]

	We need to bound sections $\sigma\in \ho^0(\X_p, \li^{\otimes d})$ such that $\mathrm{Sing}(\mathrm{div}\sigma)\cap \mathcal{D}_p$ is finite and non-empty. Since $\mathcal{D}$ is of dimension $n-1$, let $c'_0>0$ be a constant such that
	\[
		\#\mathcal{D}(\F_{p^e}) \leq c'_0 p^{(n-2)e}
	\]
	for any prime number $p$ and any integer $e\geq 1$. For any closed point $x\in |\mathcal{D}_p|$ of degree $e\leq \frac{d}{Nn}$, we have by Lemma \ref{singpt} that the proportion of $\sigma\in \ho^0(\X_p, \li^{\otimes d})$ such that $\mathrm{div}\sigma$ is singular at $x$ is $p^{-ne}$. Then we have
	\begin{eqnarray*}
		& &\frac{\#\left\{ \sigma\in \ho^0(\X_p, \li^{\otimes d})\ ;\ \exists x\in |\mathrm{Sing}(\mathrm{div}\sigma)\cap \mathcal{D}_p|,\ \deg x\leq \lfloor \frac{d}{Nn} \rfloor \right\}}{\#\ho^0(\X_p, \li^{\otimes d})} \\
		&\leq & \sum_{x\in |\mathcal{D}_p|,\ \deg x\leq \lfloor \frac{d}{Nn} \rfloor} p^{-n\deg x} \\
		&<& \sum_{e=1}^{\lfloor \frac{d}{Nn} \rfloor} \#\mathcal{D}(\F_{q^e}) p^{-ne} \\
		&\leq & \sum_{e=1}^{\lfloor \frac{d}{Nn} \rfloor}c'_0 q^{(n-2)e}\cdot p^{-ne}\\
		&=&\frac{c'_0p^{-2}}{1-p^{-2}}< 2c'_0p^{-2}.
	\end{eqnarray*}
	
	If a section $ \sigma\in \ho^0(\X_p, \li^{\otimes d})$ whose divisor has positive dimensional singular locus is not included in the above two cases, then it satisfies the two following conditions at the same time :
	\begin{enumerate}[$\bullet$]
		\item $\mathrm{Sing}(\mathrm{div}\sigma\cap \mathcal{D}_p)$ is a finite set;
		\item if $x$ is a closed point of $\mathrm{Sing}(\mathrm{div}\sigma)\cap \mathcal{D}_p\subset \mathrm{Sing}(\mathrm{div}\sigma\cap \mathcal{D}_p)$, then $\deg x\geq \frac{d}{Nn} $.
	\end{enumerate}
	Then to finish the proof, it suffices to show that we can find a constant $c_3>0$ such that when $d$ is large enough, for any $p\in \mathcal{S}$
	\[
		\frac{\#\left\{ \sigma\in \ho^0(\X_p, \li^{\otimes d})\ ;\ \mathrm{Sing}(\mathrm{div}\sigma\cap \mathcal{D}_p) \text{ finite,}\  \exists x\in\mathrm{Sing}(\mathrm{div}\sigma)\cap \mathcal{D}_p,\  \deg x\geq \frac{d}{Nn} \right\}}{\#\ho^0(\X_p, \li^{\otimes d})}< c_3p^{-2}.
	\]
	\bigskip
	
	For large enough $d$, consider the surjective morphism
	\begin{eqnarray*}
		\ho^0(\X_p, \li^{\otimes d})\times \ho^0(\X_p, \li^{\otimes (d-N)}) \longrightarrow \ho^0(\X_p, \li^{\otimes d})
	\end{eqnarray*}
	which sends $(\sigma_0, \sigma_1)\in \ho^0(\X_p, \li^{\otimes d})\times \ho^0(\X_p, \li^{\otimes (d-N)})$ to the section $\sigma=\sigma_0+\sigma_1\cdot \sigma_{\mathcal{D}}$. For any $\sigma=\sigma_0+\sigma_1\cdot \sigma_{\mathcal{D}}$ in $\ho^0(\X_p, \li^{\otimes d})$, the singular locus of $\mathrm{div}\sigma\cap \mathcal{D}_p$ is independent of $\sigma_1$, i.e.
	\[
		\mathrm{Sing}(\mathrm{div}\sigma\cap \mathcal{D}_p)=\mathrm{Sing}(\mathrm{div}\sigma_0\cap \mathcal{D}_p).
	\]
	If $(\sigma_0, \sigma_1)\in \ho^0(\X_p, \li^{\otimes d})\times \ho^0(\X_p, \li^{\otimes (d-N)})$ is such that $\mathrm{Sing}(\mathrm{div}\sigma_0\cap \mathcal{D}_p)$ is finite, we assume that $\mathrm{Sing}(\mathrm{div}\sigma_0\cap \mathcal{D}_p)=\{ x_1,\dots ,x_l \}$. Applying Lemma \ref{singnum} to $\mathcal{D}$, we have $l=O(d^{n-2})$ with coefficients depending on $\li$ and $\mathcal{D}$ but not on $p$. For a fixed $\sigma_0$ and any $x_i\in \mathrm{Sing}(\mathrm{div}\sigma_0\cap \mathcal{D}_p)$, let $x_i'$ be the first order infinitesimal neighbourhood of $x_i$ in $\X_p$. If $\mathrm{div}\left(\sigma_0+\sigma_1\sigma_{\mathcal{D}_p}\right)$ is singular at $x_i$, then the image of $\sigma_0+\sigma_1\sigma_{\mathcal{D}_p}$ in $\ho^0(x_i', \li^{\otimes d})$ by the natural restriction morphism is $0$. 
	Let $\mathfrak{m}_{\X_p, x_i}$ be the ideal sheaf of $x_i$ in $\X_p$. We have a natural exact sequence of sheaves on $\X_p$ 
	\[
		0\longrightarrow \li^{\otimes (d-N)} \longrightarrow \li^{\otimes d} \longrightarrow \li^{\otimes d}\otimes \mathcal{O}_{\mathcal{D}}\longrightarrow 0,
	\]
	where the morphism $\li^{\otimes (d-N)} \longrightarrow \li^{\otimes d}$ is the multiplication by $\sigma_{\mathcal{D}}$. Restricting this exact sequence of sheaves to the closed subscheme $x_i'$, we get a right exact sequence
	\[
		\frac{\mathcal{O}_{\X_p}}{\mathfrak{m}_{\X_p, x_i}^2}\otimes \li^{\otimes (d-N)} \longrightarrow \frac{\mathcal{O}_{\X_p}}{\mathfrak{m}_{\X_p, x_i}^2}\otimes \li^{\otimes d}\longrightarrow \frac{\mathcal{O}_{\X_p}}{\mathfrak{m}^2_{\X_p, x_i}+\mathcal{I}_{\mathcal{D}}}\otimes \li^{\otimes d}\longrightarrow 0,
	\]
	where $\mathcal{I}_{\mathcal{D}}$ is the ideal sheaf of $\mathcal{D}$. Note that the sheaf $\frac{\mathfrak{m}_{\X_p, x_i}}{\mathfrak{m}_{\X_p, x_i}^2}\otimes \li^{\otimes (d-N)}$ is contained in the kernel of the first morphism as $x_i\in \mathcal{D}$. So the above right exact sequence induces the following right exact sequence
	\[
		 \frac{\mathcal{O}_{\X_p}}{\mathfrak{m}_{\X_p, x_i}}\otimes \li^{\otimes (d-N)} \longrightarrow \frac{\mathcal{O}_{\X_p}}{\mathfrak{m}_{\X_p, x_i}^2}\otimes \li^{\otimes d}\longrightarrow \frac{\mathcal{O}_{\X_p}}{\mathfrak{m}^2_{\X_p, x_i}+\mathcal{I}_{\mathcal{D}}}\otimes \li^{\otimes d}\longrightarrow 0.
	\]
	Since $\mathcal{D}$ is nonsingular at $x_i$, the multiplication by $\sigma_{\mathcal{D}}$ morphism
	\[
		\frac{\mathcal{O}_{\X_p}}{\mathfrak{m}_{\X_p, x_i}^2}\otimes \li^{\otimes (d-N)} \longrightarrow \frac{\mathcal{O}_{\X_p}}{\mathfrak{m}_{\X_p, x_i}^2}\otimes \li^{\otimes d}
	\]
	is not a zero map. This implies that the induced morphism
	\[
		\frac{\mathcal{O}_{\X_p}}{\mathfrak{m}_{\X_p, x_i}}\otimes \li^{\otimes (d-N)} \longrightarrow \frac{\mathcal{O}_{\X_p}}{\mathfrak{m}_{\X_p, x_i}^2}\otimes \li^{\otimes d}
	\]
	is also non-zero.
	As the left term is a $\kappa(x)$ linear space of dimension $1$, this morphism is in fact injective and hence we have the following short exact sequence
	\[
		0\longrightarrow \frac{\mathcal{O}_{\X_p}}{\mathfrak{m}_{\X_p, x_i}}\otimes \li^{\otimes (d-N)} \longrightarrow \frac{\mathcal{O}_{\X_p}}{\mathfrak{m}_{\X_p, x_i}^2}\otimes \li^{\otimes d}\longrightarrow \frac{\mathcal{O}_{\X_p}}{\mathfrak{m}^2_{\X_p, x_i}+\mathcal{I}_{\mathcal{D}}}\otimes \li^{\otimes d}\longrightarrow 0.
	\]
	\\
	
	Now, with the same notation as above, $x_i\in \mathrm{Sing}(\mathrm{div}\sigma_0\cap \mathcal{D}_p)$ means that the image of the restriction of $\sigma_0$ in $\frac{\mathcal{O}_{\X_p}}{\mathfrak{m}^2_{\X_p, x_i}+\mathcal{I}_{\mathcal{D}}}\otimes \li^{\otimes d}$ is $0$. If $x_i\in \mathrm{Sing}(\mathrm{div}\sigma_0)\cap \mathcal{D}_p$, then $\sigma_0+\sigma_1\sigma_{\mathcal{D}}$ has image 0 by the restriction to $\frac{\mathcal{O}_{\X_p}}{\mathfrak{m}_{\X_p, x_i}^2}\otimes \li^{\otimes d}$. By the above exact sequence, this is a condition on $\sigma_1(x_i)\in\ho^0(x_i, \li^{\otimes (d-N)})$. By the exactness of the sequence, there is only one value of $\sigma_1(x_i)\in \ho^0(x_i, \li^{\otimes (d-N)})$ which makes $\mathrm{div}(\sigma_0+\sigma_1\sigma_{\mathcal{D}})$ singular at $x_i$. If moreover we have $\deg x_i \geq\frac{d}{Nn}$, then for any $\beta\in \ho^0(x_i, \li^{\otimes (d-N)})$, by Lemma \ref{hi} the number of sections in $\ho^0(\X_p, \li^{\otimes (d-N)})$ whose image in $\ho^0(x_i, \li^{\otimes (d-N)})$ by the restriction map is $\beta$ is bounded above by
	\[
		p^{{-\min(\lfloor\frac{d}{N}\rfloor, \deg x_i)}}\cdot \#\ho^0(\X_p, \li^{\otimes (d-N)})\leq p^{-\frac{d}{Nn}}\cdot \#\ho^0(\X_p, \li^{\otimes (d-N)}).
	\]
	Therefore, we have
	\begin{eqnarray*}
		& &\frac{\#\left\{ \sigma\in \ho^0(\X_p, \li^{\otimes d})\ ;\ \begin{array}{ll}
			\mathrm{Sing}\left(\mathrm{div}\sigma \cap \mathcal{D}_p\right) \text{ finite,}\\  \exists x\in\mathrm{Sing}(\mathrm{div}\sigma)\cap \mathcal{D}_p,\  \deg x\geq \frac{d}{Nn} 
		\end{array} \right\}}{\#\ho^0(\X_p, \li^{\otimes d})}\\
		&=&\frac{\#\left\{ (\sigma_0,\sigma_1)\ ;\ 
		\begin{array}{ll}
			\mathrm{Sing}\left(\mathrm{div}(\sigma_0+\sigma_1\sigma_{\mathcal{D}})\cap \mathcal{D}_p\right) \text{ finite,}\\  \exists x\in\mathrm{Sing}(\mathrm{div}(\sigma_0+\sigma_1\sigma_{\mathcal{D}}))\cap \mathcal{D}_p,\  \deg x\geq \frac{d}{Nn}
		\end{array}
		\right\}}{\#\ho^0(\X_p, \li^{\otimes d})\cdot \#\ho^0(\X_p, \li^{\otimes (d-N)})}\\
		&\leq& \frac{\#\ho^0(\X_p, \li^{\otimes d})\cdot \left[O(d^{n-2}) p^{-\frac{d}{Nn}}\#\ho^0(\X_p, \li^{\otimes (d-N)})\right]}{\#\ho^0(\X_p, \li^{\otimes d})\cdot \#\ho^0(\X_p, \li^{\otimes (d-N)})}\\
		&=&O(d^{n-2}) p^{-\frac{d}{Nn}}.
	\end{eqnarray*}
	Obviously, the last term is bounded above by $p^{-2}$ when $d$ is large enough. Hence we finish the proof.
\end{proof}

\subsection{Proof of Proposition \ref{singmed}}

\begin{proof}
	We fix a positive integer $N$ that satisfies Lemma \ref{relample}. Let $p$ be a prime such that $\X_p$ is smooth and irreducible. By Lemma \ref{sur}, for any closed point $x\in |\X_{p}|$ satisfying $d\geq N(n\deg x)$, i.e. $\deg x\leq \frac{d}{Nn}$, the restriction morphism
	\[
		\varphi_{p^2,x} : \ho^0(\X_{p^2}, \li^{\otimes d})\longrightarrow \ho^0(x', \li^{\otimes d}),
	\]
	is surjective, where $x'$ is the first order infinitesimal neighbourhood of $x$ in $\X$. Therefore the proportion of global sections in $\ho^0(\X_{p^2}, \li^{\otimes d})$ whose divisor satisfies $\dim_{\kappa(x)}\frac{\mathfrak{m}_{\mathrm{div}\sigma,x}}{\mathfrak{m}_{\mathrm{div}\sigma,x}^2}=n$ is equal to
	\[
		\frac{\#\Ker\ \varphi_{p^2,x}}{\#\ho^0(\X_{p^2}, \li^{\otimes d})}=p^{-(n+1)\deg x}.
	\]
	Then with the constant $c_0$ defined in Section \ref{c0}, we have
	\begin{eqnarray*}
		& &\frac{\#\left\{ \sigma\in \ho^0(\X_{p^2}, \li^{\otimes d})\ ;\ \exists x\in |\X_{p^2}|,\ \deg x\leq \frac{d}{Nn},\ \dim_{\kappa(x)}\frac{\mathfrak{m}_{\mathrm{div}\sigma,x}}{\mathfrak{m}_{\mathrm{div}\sigma,x}^2}=n \right\}  }{\#\ho^0(\X_{p^2}, \li^{\otimes d})}\\
		&\leq&\sum_{ x\in |\X_{p^2}|,\ \deg x\leq \frac{d}{Nn}}p^{-(n+1)\deg x}\\
		&\leq& \sum_{e=1}^{\lfloor \frac{d}{Nn} \rfloor}\#\X_{p^2}(\F_{p^e}) p^{-(n+1)e} \\
		&\leq& \sum_{e=1}^{\lfloor \frac{d}{Nn} \rfloor} c_0p^{(n-1)e}\cdot p^{-(n+1)e}\\
		&\leq& \sum_{e=1}^{\infty} c_0 p^{-2e}= c_0p^{-2}+\sum_{e=2}^{\infty} c_0 p^{-2e}\leq 2c_0p^{-2}.
	\end{eqnarray*}
	
	Note that Lemma \ref{posdivsing} tells us that in $\ho^0(\X_{p^2}, \li^{\otimes d})$, the proportion of sections whose divisor has positive dimensional singular locus is bounded above by $c_Hp^{-2}$. Consequently, as the restriction 
	\[
		\varphi_{d,p} : \ho^0(\X_{p^2}, \li^{\otimes d}) \longrightarrow \ho^0(\X_p, \li^{\otimes d})
	\]
	is surjective, the proportion of sections $\sigma\in \ho^0(\X_{p^2}, \li^{\otimes d})$ such that 
	\[
		\dim\mathrm{Sing}\big(\mathrm{div} (\sigma|_{\X_p})\big)>0
	\]
	is also bounded above by $c_Hp^{-2}$. To finish the proof, it suffices to bound the proportion of sections in the set
	\[
		\left\{ \sigma\in \ho^0(\X_{p^2}, \li^{\otimes d})\ ;\ \mathrm{Sing}(\mathrm{div}\sigma|_{\X_p}) \text{ finite,}\ \exists x\in |\X_{p^2}|,\ \deg x \geq\frac{d}{Nn},\ \dim_{\kappa(x)}\frac{\mathfrak{m}_{\mathrm{div}\sigma,x}}{\mathfrak{m}_{\mathrm{div}\sigma,x}^2}=n \right\}.
	\]

	Now we take a subset $E_{d,p}\subset \ho^0(\X_{p^2}, \li^{\otimes d})$ such that the restriction map $\varphi_{d,p}$ induces a bijection from $E_{d,p}$ to $\ho^0(\X_p, \li^{\otimes d})$. For example, if we choose a $\Z/p^2\Z$-basis of $\ho^0(\X_{p^2}, \li^{\otimes d})$, we can take $E_{d,p}$ to be the set of sections having coefficients in $\{0,1,\dots, p-1\}\subset \Z/p^2\Z$ when written as linear combination of sections in this basis. Then any section $\sigma\in\ho^0(\X_{p^2}, \li^{\otimes d})$ can be written uniquely as
	\[
		\sigma=\sigma_1+p\sigma_2
	\]
	for some $\sigma_1,\sigma_2\in E_{d,p}$. Then $\mathrm{Sing}(\mathrm{div}(\sigma|_{\X_p})) =\mathrm{Sing}(\mathrm{div}(\sigma_1|_{\X_p})) $. Now let $\sigma$ be a section in $ \ho^0(\X_{p^2}, \li^{\otimes d})$ such that $\mathrm{Sing}(\mathrm{div}(\sigma|_{\X_p}))$ is finite. We may assume that as a set, $|\mathrm{Sing}(\mathrm{div}(\sigma|_{\X_p}))|=\{ x_1,\dots, x_l  \}$. Then by Lemma \ref{singnum}, we have
	\[
		l=O(d^{n-1}).
	\]
	Moreover, for $i\in\{1,\dots, l \}$, $\mathrm{div}(\sigma_1+p\sigma_2)$ is singular at $x_i$ if and only if the image of $\sigma_1+p\sigma_2$ in $\ho^0(x_i', \li^{\otimes d})$ is $0$, where $x_i'$ is the first order infinitesimal neighbourhood of $x_i$ in $\X_{p^2}$. Let $\mathfrak{m}_{x_i}$ be the ideal sheaf of $x_i$ in $\X_{p^2}$. Then $x_i'$ is defined by the ideal sheaf $\mathfrak{m}_{x_i}^2$. Now we have a right exact sequence of sheaves
	\[
		\frac{\mathcal{O}_{\X_{p^2}}}{\mathfrak{m}_{x_i}^2}\otimes \li^{\otimes d} \longrightarrow \frac{\mathcal{O}_{\X_{p^2}}}{\mathfrak{m}_{x_i}^2}\otimes \li^{\otimes d} \longrightarrow \frac{\mathcal{O}_{\X_{p^2}}}{\mathfrak{m}_{x_i}^2+p\mathcal{O}_{\X_{p^2}}}\otimes \li^{\otimes d}\longrightarrow 0,
	\]
	where the first morphism is the multiplication by $p$. Note that $\frac{\mathfrak{m}_{x_i}}{\mathfrak{m}_{x_i}^2}\otimes \li^{\otimes d}\subset \frac{\mathcal{O}_{\X_{p^2}}}{\mathfrak{m}_{x_i}^2}\otimes \li^{\otimes d}$ is contained in the kernel of the second morphism, we obtain an exact sequence
	\[
		 \frac{\mathcal{O}_{\X_{p^2}}}{\mathfrak{m}_{x_i}}\otimes \li^{\otimes d} \longrightarrow \frac{\mathcal{O}_{\X_{p^2}}}{\mathfrak{m}_{x_i}^2}\otimes \li^{\otimes d} \longrightarrow \frac{\mathcal{O}_{\X_{p^2}}}{\mathfrak{m}_{x_i}^2+p\mathcal{O}_{\X_{p^2}}}\otimes \li^{\otimes d}\longrightarrow 0.
	\]
	Note that $\X$ is regular. The multiplication by $p$ map
	\[
		\frac{\mathcal{O}_{\X_{p^2}}}{\mathfrak{m}_{x_i}^2}\otimes \li^{\otimes d} \longrightarrow \frac{\mathcal{O}_{\X_{p^2}}}{\mathfrak{m}_{x_i}^2}\otimes \li^{\otimes d} 
	\]
	cannot be zero, so is $ \frac{\mathcal{O}_{\X_{p^2}}}{\mathfrak{m}_{x_i}}\otimes \li^{\otimes d} \longrightarrow \frac{\mathcal{O}_{\X_{p^2}}}{\mathfrak{m}_{x_i}^2}\otimes \li^{\otimes d}$. But if this morphism is not zero, it must be injective as $\frac{\mathcal{O}_{\X_{p^2}}}{\mathfrak{m}_{x_i}}\otimes \li^{\otimes d} $ is in fact a sheaf supported on $x_i$ where its stalk is a $\kappa(x_i)$-vector space of dimension $1$. Hence we get a short exact sequence
	\[
		 0\longrightarrow \frac{\mathcal{O}_{\X_{p^2}}}{\mathfrak{m}_{x_i}}\otimes \li^{\otimes d} \longrightarrow \frac{\mathcal{O}_{\X_{p^2}}}{\mathfrak{m}_{x_i}^2}\otimes \li^{\otimes d} \longrightarrow \frac{\mathcal{O}_{\X_{p^2}}}{\mathfrak{m}_{x_i}^2+p\mathcal{O}_{\X_{p^2}}}\otimes \li^{\otimes d}\longrightarrow 0.
	\]
	Since these sheaves are all supported on $x_i$, this sequence induces the following exact sequence of groups
	\[
		0\longrightarrow \ho^0(x_i, \li^{\otimes d})\longrightarrow \ho^0(x_i', \li^{\otimes d}) \longrightarrow \ho^0(x_i'\cap \X_p, \li^{\otimes d}) \longrightarrow 0.
	\]
	It tells us that there is only one value of $ \ho^0(x_i, \li^{\otimes d})$ for $\sigma_2(x_i)$ which makes $\dim_{\kappa(x)}\frac{\mathfrak{m}_{\mathrm{div}\sigma,x}}{\mathfrak{m}_{\mathrm{div}\sigma,x}^2}=n$ for $\sigma=\sigma_1+p\sigma_2$. Note that Lemma \ref{hi} tells us that when $d$ is large enough, for any $x\in |\X_{p^2}|$ satisfying $\deg x\geq \frac{d}{Nn}$ and any value $s\in \ho^0(x,\li^{\otimes d})$, we have
	\[
		\frac{\#\{ \sigma\in \ho^0(\X_p, \li^{\otimes d})\ ;\ \sigma(x)=s  \} }{\#\ho^0(\X_p,\li^{\otimes d})} \leq  p^{-\lfloor\frac{d}{N}\rfloor} \leq p^{1-\frac{d}{N}}.
	\]
	Since $\varphi_{d,p}$ induces a bijection between $E$ and $\ho^0(\X_p, \li^{\otimes d})$, we have
	\[
		\#\{ \sigma\in E_{d,p}\ ;\ \sigma(x)=s  \} \leq p^{1-\frac{d}{N}}\cdot \#\ho^0(\X_p, \li^{\otimes d}).
	\]
	Therefore, when $d$ is large enough, we have
	\begin{eqnarray*}
		& &\frac{\#
		\left\{ \sigma\in \ho^0(\X_{p^2}, \li^{\otimes d})\ ;\ \begin{array}{ll}
			\mathrm{Sing}(\mathrm{div}(\sigma|_{\X_p})) \text{ finite,}\ \exists x\in |\X_{p^2}|,\\  \deg x\geq \frac{d}{Nn},\ \dim_{\kappa(x)}\frac{\mathfrak{m}_{\mathrm{div}\sigma,x}}{\mathfrak{m}_{\mathrm{div}\sigma,x}^2}=n
		\end{array} \right\}}{\#\ho^0(\X_{p^2}, \li^{\otimes d})}\\
		&=&\frac{\#\left\{ (\sigma_1,\sigma_2)\in E_{d,p}\times E_{d,p}\ ;\ 
		\begin{array}{ll}
			\mathrm{Sing}(\mathrm{div}(\sigma_1|_{\X_p})) \text{ finite,}\ \exists x\in |\X_{p^2}|,\ \deg x\geq \frac{d}{Nn}, \\  \dim_{\kappa(x)}\frac{\mathfrak{m}_{\mathrm{div}\sigma,x}}{\mathfrak{m}_{\mathrm{div}\sigma,x}^2}=n, \text{ with }\sigma=\sigma_1+p\sigma_2
		\end{array}
		\right\}}{\#\ho^0(\X_{p^2}, \li^{\otimes d})}\\
		&\leq& \frac{\#E_{d,p}\cdot \left[O(d^{n-1}) p^{1-\frac{d}{N}}\#\ho^0(\X_p, \li^{\otimes d})\right]}{\#\ho^0(\X_{p^2}, \li^{\otimes d})}\\
		&=&O(d^{n-1}) p^{1-\frac{d}{N}}\leq p^{-2}.
	\end{eqnarray*}
	This finishes the proof.
\end{proof}

\subsection{Proof of Theorem \ref{mainvar}}\label{finalarakelov}

Note that in Section \ref{smallresidual} we have shown that the proportion of sections in $\ho_{\mathrm{Ar}}^0(\X,\Li^{\otimes d})$ whose divisor has no singular point of residual characteristic smaller than or equal to $d^{\frac{1}{n+1}}$ tends to $\zeta_\X(n+1)^{-1}$ already. Setting $\epsilon_0=\frac{\varepsilon_0}{2}$ where $\varepsilon_0$ is defined in Proposition \ref{epsilon}, Theorem \ref{mainvar} can be reduced to the following : 
\begin{prop}\label{medsing}
	Let $\X$ be a regular projective arithmetic variety of absolute dimension $n$, and let $\Li$ be an ample Hermitian line bundle on $\X$. For any $0<\varepsilon<\epsilon_0$, set
	\[
		\mathcal{Q}_{d}^m:=\left\{ \sigma\in \ho^0(\X,\Li^{\otimes d})\ ;\ \begin{array}{ll}
		\mathrm{div}\sigma \text{ has a singular point of residual} \\
		\text{characteristic between } d^{\frac{1}{n+1}} \text{ and } e^{\varepsilon d}
		\end{array}\right\}.
	\]
	When $d$ is sufficiently large, we have
	\[
		\frac{\#\left(\mathcal{Q}_{d}^m\cap \ho^0_{\mathrm{Ar}}(\X, \overline{\mathcal{L}}^{\otimes d})\right)}{\#\ho^0_{\mathrm{Ar}}(\X, \overline{\mathcal{L}}^{\otimes d})} = O(d^{-\frac{1}{n+1} }).
	\]
	Here the constant involved in the big $O$ depends only on $\X$ and $\Li$.
	
	In particular, denoting $\mathcal{Q}^m=\bigcup_{d>0}\mathcal{Q}_{d}^m$, we have
	\[
		\mu_{\mathrm{Ar}}(\mathcal{Q}^m)=0.
	\]
\end{prop}

\begin{proof}
	Since for any $\sigma\in \ho^0(\X, \Li^{\otimes d})$, $\mathrm{div}\sigma$ has a singular point on the fiber $\X_p$ implies $\sigma\ \mathrm{mod} \ p^2\in \mathcal{Q}_{d,p^2}$, we have
	\begin{eqnarray*}
		\frac{\#\left(\mathcal{Q}_d^{m}\cap \ho^0_{\mathrm{Ar}}(\X, \overline{\mathcal{L}}^{\otimes d})\right)}{\#\ho^0_{\mathrm{Ar}}(\X,\Li^{\otimes d})}
		&\leq& \sum_{d^{\frac{1}{n+1}}\leq p \leq e^{\varepsilon d}} \frac{\#\{ \sigma\in \ho^0_{\mathrm{Ar}}(\X,\Li^{\otimes d})\ ;\ \sigma\ \mathrm{mod}\ p^2\in \mathcal{Q}_{d,p^2} \}}{\#\ho^0_{\mathrm{Ar}}(\X,\Li^{\otimes d})}.
	\end{eqnarray*}
	As $p\geq d^{\frac{1}{n+1}}$ implies that $p^2$ is odd, we can apply Proposition \ref{ed} to the case $N=p^2$ and $E=\mathcal{Q}_{d,p^2}$, and obtain that for any $0<\delta<\varepsilon_0$ with $\varepsilon_0$ defined in Proposition \ref{epsilon}, when $d$ is large enough and $p^2\leq e^{\delta d}$, we have
	\begin{eqnarray*}
		 & &\frac{\#\{ \sigma\in \ho^0_{\mathrm{Ar}}(\X,\Li^{\otimes d})\ ;\ \sigma\ \mathrm{mod}\ p^2\in \mathcal{Q}_{d,p^2} \}}{\#\ho^0_{\mathrm{Ar}}(\X,\Li^{\otimes d})}\\
		 &\leq& 4p^{-2\mathrm{rk}(\ho^0(\X,\Li^{\otimes d}))}\cdot \#\mathcal{Q}_{d,p^2}  \\
		 &=&4\frac{\#\mathcal{Q}_{d,p^2} }{\#\mathrm{H}^0(\X_{p^2}, \li^{\otimes d})}.
	\end{eqnarray*}
	Since $\X$ is a regular arithmetic variety, it is irreducible and generically smooth. So if $d$ is large enough, for any prime number $p\geq d^{\frac{1}{n+1}}$, $\X_p$ is irreducible and smooth over $\F_p$. Then Proposition \ref{singmed} tells us that there exists a constant $c>0$ such that
	\[
		\frac{\#\mathcal{Q}_{d,p^2} }{\#\mathrm{H}^0(\X_{p^2}, \li^{\otimes d})}\leq cp^{-2}.
	\]
	Note that if $\varepsilon>0$ satisfies $\varepsilon<\epsilon_0$, then $2\varepsilon$ satisfies the condition on $\delta$. Therefore we conclude with 
	\begin{eqnarray*}
		\frac{\#\left(\mathcal{Q}_d^{m}\cap \ho^0_{\mathrm{Ar}}(\X, \overline{\mathcal{L}}^{\otimes d})\right)}{\#\ho^0_{\mathrm{Ar}}(\X,\Li^{\otimes d})} &\leq& \sum_{d^{\frac{1}{n+1}}\leq p \leq e^{\varepsilon d}} \frac{\#\{ \sigma\in \ho^0_{\mathrm{Ar}}(\X,\Li^{\otimes d})\ ;\ \sigma\ \mathrm{mod}\ p^2\in \mathcal{Q}_{d,p^2} \}}{\#\ho^0_{\mathrm{Ar}}(\X,\Li^{\otimes d})}\\
		&\leq&\sum_{d^{\frac{1}{n+1}}\leq p \leq e^{\varepsilon d}}4\frac{\#\mathcal{Q}_{d,p^2} }{\#\mathrm{H}^0(\X_{p^2}, \li^{\otimes d})} \\
		&\leq& \sum_{d^{\frac{1}{n+1}}\leq p \leq e^{\varepsilon d}} 4cp^{-2} \\
		&=& 4c\left(\sum_{d^{\frac{1}{n+1}}\leq p \leq \infty} p^{-2}\right)\\
		&<& 4c d^{-\frac{1}{n+1}},
	\end{eqnarray*}
	which is the statement of the proposition.
\end{proof}

\begin{proof}[Proof of Theorem \ref{mainvar}]
	Since $\mathcal{P}_{A,\varepsilon}\subset \mathcal{P}_B\subset \mathcal{P}_A\cup \mathcal{Q}^m$, we get that 
	\[
		|\mu_{\mathrm{Ar}}(\mathcal{P}_{A,\varepsilon})-\mu_{\mathrm{Ar}}(\mathcal{P}_B)|\leq \mu_{\mathrm{Ar}}(\mathcal{Q}^m)=0.
	\]
	Therefore we have
	\[
		\mu_{\mathrm{Ar}}(\mathcal{P}_{A,\varepsilon})=\mu_{\mathrm{Ar}}(\mathcal{P}_B)=\zeta_{\X}(1+n)^{-1}.
	\]
	This finishes the proof.
\end{proof}

\subsection{Proof of Corollary \ref{singR}}

\begin{proof}
	Let $\X$ be a regular projective arithmetic variety of dimension $n$, and let $\Li$ be an ample Hermitian line bundle on $\X$. By Proposition \ref{epsilon}, there exists a positive constant $\varepsilon_0$ such that for any large enough integer $d$, $\ho^0(\X, \Li^{\otimes d})$ has a basis consisting of sections with norm smaller than $e^{-\varepsilon_0 d}$. Choose the constant $c$ to be a real number satisfying $1<c<e^{\varepsilon_0}$. For any $R>1$, set $\Li'=(\li, \lVert\cdot\rVert' )$ where $\lVert\cdot\rVert'=\lVert\cdot\rVert R^{-1}$. Since $(\li, \lVert\cdot\rVert e^{-\delta})$ is ample for any $\delta>0$, the Hermitian line bundle $\Li'$ is also ample. Then by construction, for any large enough integer $d$, $\ho^0(\X, \Li'^{\otimes d})$ has a basis consisting of sections with norm smaller than $R^{-d}e^{-\varepsilon_0 d}=e^{-(\varepsilon_0+\log R)d}$. Set $\varepsilon=\frac{1}{2}\log(cR)$. Then we have
	\[
		\varepsilon=\frac{1}{2}(\log c+\log R)< \frac{1}{2}(\varepsilon_0+\log R),
	\]
	and we can apply Theorem \ref{mainvar} to $\Li'$ with constant $\varepsilon$ chosen as above. Then the density result is exactly what we need to prove.
	
\end{proof}

\appendix
\section{Bertini smoothness theorem over finite fields}\label{finitefield}

In the Appendix, we prove a slightly generalized version of B. Poonen's Bertini theorem over finite fields. The precise statement is the following :

\begin{thm}\label{bertinifini}
	Let $\F_q$ be a finite field of characteristic $p$. Let $Y$ be a projective scheme of dimension $n$ over $\F_q$, and $X$ a smooth subscheme of $Y$ of dimension $m$. Let $\mathcal{L}$ be an ample line bundle on $Y$. Assume that there exists a smooth open subscheme $U$ in $Y$ containing $X$. 
	Set 
	\[
		\mathcal{P}_{d}:=\{\sigma\in \ho^0(Y, \mathcal{L}^{\otimes d})\ ;\ \mathrm{div}\sigma\cap X \text{ is smooth of dimension }m-1\}
	\]
	and $\mathcal{P}=\bigcup_{d\geq 0}\mathcal{P}_{d}$.
	We have
	\[
		\mu(\mathcal{P})=\lim_{d\rightarrow \infty} \frac{ \#\mathcal{P}_{d}}{\# \ho^0(Y, \mathcal{L}^{\otimes d})}
		=\zeta_X(m+1)^{-1}>0.
	\]
\end{thm}

Here $\zeta_X$ is the zeta function
\[
	\zeta_X(s)=\prod_{x\in |X|}\big(1-\#\kappa(x)^{-s}\big)^{-1}.
\]
	
\begin{rmq}
	If we take $Y=\pr^n_{\F_q}$, $\li=\mathcal{O}(1)$, we get Poonen's theorem. 
\end{rmq}
	
Note that for a $\sigma\in \ho^0(Y, \mathcal{L}^{\otimes{d}})$, $\mathrm{div}\sigma\cap X$ is smooth if and only if it is non-singular at every closed point of $\mathrm{div}\sigma\cap X$. To prove this theorem, we classify the closed points of $X$ by their degree, so that for each degree there exist only finitely many closed points. In Poonen's proof, he classifies the closed points into three parts for each $\li^{\otimes d}$, which are the following: closed points of degree smaller than or equal to a chosen positive integer $r$, closed points of degree between $r$ and $\frac{d}{m+1}$, and closed points of degree bigger than $\frac{d}{m+1}$. Then he estimates the number of sections in $\ho^0(Y, \mathcal{L}^{\otimes d})$ whose divisor has singular points in these parts, respectively. 
	
Our proof follows his method in a faithful way. But we need more explicit bounds for bad sections, so as to get the speed of convergence for the final limit. For technical reasons, we need the following result:
\begin{lem}\label{ample}
	Let $\li$ be an ample line bundle on a projective scheme $Y$ over a field $k$. Then there exists a positive integer $N$ such that
	\begin{enumerate}[label=\roman*)]
		\item $\li^{\otimes d}$ is very ample for all $d\geq N$;
		\item for any $a, b\geq N$, the natural morphism
		\[
			\ho^0(Y,\li^{\otimes a})\otimes \ho^0(Y,\li^{\otimes b})\longrightarrow \ho^0(Y,\li^{\otimes (a+b)})
		\]
		is surjective.
	\end{enumerate}
\end{lem}
\begin{proof}
	This is a classical result. The first statement is part of \cite[Theorem 1.2.6]{La04}, and the second statement can be deduced directly from \cite[Theorem 1.8.3]{La04}. 
\end{proof}
	
We choose a positive integer $r$, an integer $N$ satisfying this lemma and depending possibly on $q$, and set
\begin{eqnarray*}
	\mathcal{P}_{d,\leq r}&=&\{ \sigma\in \ho^0(Y,\li^{\otimes d})\ ;\ \forall x\in X,\ \deg x\leq r,\ \mathrm{div}\sigma\cap X \text{ is smooth of dimension } m-1 \text{ at }x \}, \\
	\mathcal{Q}_{d, >r}^{\mathrm{med}} &=& \{ \sigma\in \ho^0(Y,\li^{\otimes d})\ ;\ \exists x\in X,\ r<\deg x\leq \frac{d}{(m+1)N},\ \mathrm{div}\sigma\cap X \text{ is singular at }x  \},  \\
	\mathcal{Q}_{d}^{\mathrm{high}}&=& \{ \sigma\in \ho^0(Y,\li^{\otimes d})\ ;\ \exists x\in X,\ \deg x\geq \frac{d}{(m+1)N},\ \mathrm{div}\sigma\cap X \text{ is singular at }x \}.
\end{eqnarray*}
Then clearly
\[
	\mathcal{P}_{d}\subset\mathcal{P}_{d,\leq r}\subset \mathcal{P}_{d}\cup \mathcal{Q}_{d, >r}^{\mathrm{med}}\cup \mathcal{Q}_{d}^{\mathrm{high}}.
\]
	
We give bound for the proportion of these three sets.
	
\subsection{Singular points of small degree}

\begin{lem}\label{surject}
	Let $Y$ be a projective scheme over $\F_q$, $\li$ an ample line bundle over $Y$. Let $Z$ be a finite sub-scheme of $Y$. Let $N$ be a positive integer satisfying Lemma \ref{ample}. Then the restriction morphism
	\[
		\phi_{d,Z}: \ho^0(Y, \li^{\otimes d})\longrightarrow \ho^0(Z, \li^{\otimes d})
	\]
	is surjective for all $d\geq N h_Z$, where $h_Z=\dim_{\F_q} \ho^0(Z, \mathcal{O}_Z)$.
\end{lem}
\begin{proof}
	If $\li$ is very ample, by \cite[Lemma 2.1]{Po04}, $\phi_{d,Z}$ is surjective when $d\geq h_Z-1$, and this lemma is also true. When $L$ is only ample, for any $\delta_0\geq N$, $\li^{\delta_0}$ is very ample and $\phi_{d\delta_0,Z}$ is surjective for any $d\geq h_Z-1$. Now for any $d\geq N h_Z$, we can find $s_d\geq h_Z-1$ and $N\leq r_d\leq 2N$ such that $d=s_dN+r_d$. By Lemma \ref{ample}, we have a surjection
	\[
		\ho^0(Y, \li^{\otimes s_dN})\otimes \ho^0(Y, \li^{\otimes r_d})\longrightarrow \ho^0(Y, \li^{\otimes d}).
	\]
	Moreover, since $Z$ is finite, for all $d\geq 0$, we have $\ho^0(Z, \li^{\otimes d})\simeq \ho^0(Z, \mathcal{O}_Z)$. These isomorphisms are not canonical, but can give us an isomorphism
	\[
		\ho^0(Z, \li^{\otimes s_dN})\otimes_{\ho^0(Z, \mathcal{O}_Z)} \ho^0(Z, \li^{\otimes r_d})\longrightarrow \ho^0(Z, \li^{\otimes d})
	\]
	which makes the following diagram commutative:
	\[
		\xymatrix{
			\ho^0(Y,\mathcal{L}^{\otimes s_d N})\otimes \ho^0(Y, \mathcal{L}^{\otimes r_d}) \ar@{->>}[r] \ar[d] & \ho^0(Y,\mathcal{L}^{\otimes d})\  \ar[d] \\
			\ho^0(Z, \mathcal{L}^{\otimes s_dN})\otimes_{\ho^0(Z,\mathcal{O}_Z)} \ho^0(Z, \mathcal{L}^{\otimes r_d}) \ar[r]^{\qquad\qquad\quad\ \ \sim} & \ho^0(Z, \mathcal{L}^{\otimes d}).
		}
	\]
	Thus it suffices to show that the left vertical morphism is surjective. Since $Y$ is projective, $\ho^0(Y, \mathcal{L}^{\otimes r_d})$ is of finite $\F_q$-dimension. Let $\ho^0(Y, \mathcal{L}^{\otimes r_d})=\bigoplus_i \F_q t_i$. As $\li^{r_d}$ is very ample, it is globally generated. So for each $z\in |Z|$, we can find one $t_i$ in the set of generators such that $t_i(z)\not=0$. Since surjectivity is stable under field base change, replacing $\F_q$ by a finite field extension, we can assume that there exists a linear combination $t=\sum_i a_it_i$ such that $t(z)\not=0$ for any $z\in |Z|$. Then we have
	\[
		\xymatrix{
			\ho^0(Y, \mathcal{L}^{\otimes s_dN})\otimes \F_q t  \ar@{^(->}[r] \ar@{->>}[d] & \ho^0(Y,\mathcal{L}^{\otimes s_d N})\otimes \ho^0(Y, \mathcal{L}^{\otimes r_d}) \ar[d] \\
			\ho^0(Z, \mathcal{L}^{\otimes s_dN})\otimes \F_q t|_Z \ar[r] & \ho^0(Z, \mathcal{L}^{\otimes s_dN})\otimes_{\ho^0(Z,\Oc_Z)} \ho^0(Z, \mathcal{L}^{\otimes r_d}).
		}
	\]
	By our construction, the section $t$ trivializes $\ho^0(Z, \mathcal{L}^{\otimes r_d})$. So the bottom morphism is an isomorphism. Hence the right vertical morphism is surjective. By the commutativity of the first diagram, the morphism
	\[
		\ho^0(Y,\li^{\otimes d})\longrightarrow \ho^0(Z, \li^{\otimes d})
	\]
	is also surjective, which is what we need to show.
\end{proof}

With this lemma, we can control the proportion of $\mathcal{P}_{d,\leq r}$. 

\begin{prop}\label{smallf}
	Let $Y$ be a projective scheme of dimension $n$ over $\F_q$ equipped with an ample line bundle $\li$. Let $X$ be a subscheme smooth over $\F_q$ of dimension $m$ of $Y$. Set
	\[
		\mathcal{P}_{d,\leq r}=\{ \sigma\in \ho^0(Y,\li^{\otimes d})\ ;\ \forall x\in X,\ \deg x\leq r,\ \mathrm{div}\sigma\cap X \text{ is smooth of dimension } m-1 \text{ at }x \},
	\]
	and $\mathcal{P}_{\leq r}=\bigcup_{d\geq 0}\mathcal{P}_{d,\leq r}$.
	We have 
	\[
		\mu(\mathcal{P}_{\leq r})=\lim_{d\rightarrow \infty} \frac{\#\mathcal{P}_{d,\leq r}}{\#\ho^0(Y, \li^{\otimes d})}=\prod_{\deg x\leq r} \left( 1-q^{-(m+1)\deg x} \right).
	\]
	In fact, with a positive integer $N$ satisfying Lemma \ref{ample}, for any $d\geq N\Big(\sum_{\deg x\leq r}(1+m)\deg x\Big)$, we have
	\[
		\frac{\#\mathcal{P}_{d,\leq r}}{\#\ho^0(Y, \li^{\otimes d})}=\prod_{\deg x\leq r} \left( 1-q^{-(m+1)\deg x} \right).
	\]
\end{prop}
\begin{proof}
	For any closed point $x$ of $X$, let $x'$ be the closed subscheme in $X$ defined by $\mathfrak{m}_x^2$, where $\mathfrak{m}_x$ is the ideal sheaf of $x$ in $X$. Then $x'$ is the first order infinitesimal neighbourhood of $x$ in $X$. Let $X'_{\leq r}$ be the union of the closed subschemes $x'$ for all $x\in X$ with $\deg x\leq r$. Note that the number of closed points of $X$ with degree smaller than or equal to $r$ is finite. This union is a disjoint finite union. So $X'_{\leq r}$ is a finite subscheme of $X$ defined by the ideal sheaf $\prod_{\deg x\leq r}\mathfrak{m}_x^2$. Hence
	\[
		\ho^0(X'_{\leq r},\mathcal{O}_{X'_{\leq r}})=\ho^0(X'_{\leq r}, \prod_{\deg x\leq r } \mathcal{O}_{x'})=\prod_{\deg x\leq r }\ho^0(Y, \mathcal{O}_X/\mathfrak{m}_x^2).
	\]
	Since $X$ is smooth over $\F_q$, for any closed point $x$, we have 
	\[
		\dim_{\kappa(x)} \ho^0(X, \mathfrak{m}_x/\mathfrak{m}_x^2)=m.
	\]
	Hence
	\[
		\dim_{\F_q} \ho^0(X, \mathcal{O}_X/\mathfrak{m}_x^2)=(1+m)\deg x,
	\]
	and we have
	\[
		\dim_{\F_q}\ho^0(X'_{\leq r},\mathcal{O}_{X'_{\leq r}})=\sum_{\deg x\leq r} (1+m)\deg x .
	\]
	Apply Lemma \ref{surject} to the case $Z=X'_{\leq r}$. We get that the morphism
	\[
		\ho^0(Y, \li^{\otimes d})\longrightarrow \ho^0(X'_{\leq r}, \li^{\otimes d})
	\]
	is surjective if $d\geq N\Big(\sum_{\deg x\leq r}(1+m)\deg x\Big)$. 
		
	Note that for a section $\sigma\in \ho^0(Y, \li^{\otimes d})$, the intersection $\mathrm{div}\sigma\cap X$ is singular at a closed point $x\in X$ if and only if the image of $\sigma$ in $\ho^0(x', \li^{\otimes d})$ by the restriction map
	\[
		\ho^0(Y, \li^{\otimes d})\longrightarrow \ho^0(x', \li^{\otimes d})
	\]
	is zero. So $\mathrm{div}\sigma\cap X$ has no singular point of degree smaller than or equal to $r$ if and only if its restriction to $\ho^0(X'_{\leq r},\mathcal{O}_{X'_{\leq r}})\simeq\prod_{\deg x\leq r }\ho^0(x', \li^{\otimes d})$ lies in the subset $\prod_{\deg x\leq r }\left(\ho^0(x', \li^{\otimes d})-\{0\}\right)$. So for any $d\geq N\Big(\sum_{\deg x\leq r}(1+m)\deg x\Big)$,
	\begin{eqnarray*}
		\frac{\#\mathcal{P}_{d,\leq r}}{\#\ho^0(Y, \li^{\otimes d})}&=&\frac{\#\prod_{\deg x\leq r }\left(\ho^0(X, \mathcal{O}_X/\mathfrak{m}_x^2)-\{0\}\right)}{\#\prod_{\deg x\leq r }\ho^0(X, \mathcal{O}_X/\mathfrak{m}_x^2)}\\
		&=&\frac{\prod_{\deg x\leq r }\left(q^{(m+1)\deg x}-1\right)}{\prod_{\deg x\leq r }q^{(m+1)\deg x}}\\
		&=&\prod_{\deg x\leq r} \left( 1-q^{-(m+1)\deg x} \right).
	\end{eqnarray*}
	This shows the result.
\end{proof}
	
\subsection{Singular points of medium degree}

\begin{lem}\label{singpt}
	Let $Y$ be a projective scheme over $\F_q$ of dimension $n$ and $X$ a smooth subscheme of $Y$ of dimension $m$. Let $\li$ be an ample line bundle of $Y$. Let $N$ be a positive integer satisfying Lemma \ref{ample}. For a fixed $d$, let $x\in X$ be a closed point of $X$ of degree $e$ such that 
	\[
		e\leq \lfloor \frac{d}{N(m+1)} \rfloor.
	\]
	Then the proportion of $\sigma\in \ho^0(Y, \li^{\otimes d})$ such that $\mathrm{div}\sigma\cap X$ is not smooth of dimension $m-1$ at $x$ is $q^{-(m+1)e}$.
\end{lem}
\begin{proof}
	Let $x'$ be the first order infinitesimal neighbourhood of $x$ in $X$. Apply Lemma \ref{surject} to the case $Z=x'$. We obtain that the restriction map 
	\[
		\ho^0(Y, \li^{\otimes d})\longrightarrow \ho^0(x', \li^{\otimes d})
	\]
	is surjective when 
	\[
		Nh^0(x', \li^{\otimes d})=N(m+1)\deg x\leq d,
	\]
	that is, when
	\[
		\deg x\leq \lfloor \frac{d}{N(m+1)} \rfloor.
	\]
	Since a section $\sigma\in \ho^0(Y, \li^{\otimes d})$ is such that $\mathrm{div}\sigma\cap X$ is singular at $x$ if and only if the image of $\sigma$ in $\ho^0(x', \li^{\otimes d})$ is $0$. Hence when the degree condition for $x$ is satisfied, the proportion of such sections is equal to
	\[
		\frac{1}{\#\ho^0(x', \li^{\otimes d})}=q^{-(m+1)e},
	\]
	Thus we get the result. 
\end{proof}
\begin{prop} \label{mediumf}
	Let $Y$ be a projective scheme over $\F_q$ of dimension $n$ and $X$ a smooth subscheme of $Y$ of dimension $m$. Let $\li$ be an ample line bundle of $Y$. Let $N$ be a positive integer satisfying Lemma \ref{ample}. Set
	\[
		\mathcal{Q}_{d, >r}^{\mathrm{med}} = \{ \sigma\in \ho^0(Y,\li^{\otimes d})\ ;\ \exists x\in X,\ r<\deg x\leq \frac{d}{(m+1)N},\ \mathrm{div}\sigma\cap X \text{ is singular at }x  \}. 
	\]
	Then there exists a constant $c_0$ such that
	\[
		\frac{\#\mathcal{Q}_{d, >r}^{\mathrm{med}}}{\# \ho^0(Y, \mathcal{L}^{\otimes d})} \leq 2c_0q^{-r}.
	\]
	In particular,
	\[
		\overline{\mu}(\mathcal{Q}_{d, >r}^{\mathrm{med}} )\leq 2c_0q^{-r}.
	\]
\end{prop}
\begin{proof}
	Identifying $Y$ to a closed subscheme of a projective space, we can see $X$ as a subscheme of the same projective space. By \cite{LW54}, we can find a constant $c_0>0$ such that for any $e\geq 1$,
	\[
		\#X(\F_{q^e})\leq c_0 q^{me}.
	\]
	Let $N$ be the positive integer as in the previous lemma. Then the lemma tells us that if $x\in X$ is a closed point of degree $e\leq \lfloor \frac{d}{(m+1)N}\rfloor$, the proportion of sections $\sigma\in \ho^0(Y, \li^{\otimes d})$ such that $\mathrm{div}\sigma\cap X$ is singular at $x$ is $q^{-(m+1)e}$. Therefore we have
	\begin{eqnarray*}
		&&\frac{\#\mathcal{Q}_{d, >r}^{\mathrm{med}}}{\# \ho^0(Y, \mathcal{L}^{\otimes d})} \\
		&\leq& \sum_{e=r+1}^{\lfloor\frac{d}{(m+1)N}\rfloor}\#X(\mathbb{F}_{q^e})\cdot q^{-(m+1)e} \\
		&< & \sum_{e=r}^{\infty}\#X(\mathbb{F}_{q^e})q^{-(m+1)e} \\
		&\leq & \sum_{e=r}^{\infty}c_0q^{em}\cdot q^{-(m+1)e}=\sum_{e=r}^{\infty}c_0q^{-e}=\frac{c_0q^{-r}}{1-q^{-1}}.
	\end{eqnarray*}
	Since $q\geq 2$, we have $1-q^{-1}\geq \frac{1}{2}$ and $\frac{c_0q^{-r}}{1-q^{-1}}\leq 2c_0q^{-r}$. Hence we get $\overline{\mu}(\mathcal{Q}_{d, >r}^{\mathrm{med}} )\leq  2c_0q^{-r}$. This implies our result.
\end{proof}
	
\subsection{Singular points of high degree}

\begin{lem}
	Let $Y$ be a projective scheme over $\F_q$ of dimension $n$. Let $Z$ be a finite closed subscheme of $Y$ whose support is included in the smooth locus of $Y$. Let $\li$ be an ample line bundle over $Y$ and let $N$ be a positive integer satisfying Lemma \ref{ample}. After replacing $\F_q$ by a finite extension of $\F_q$ if needed, we can find a linear subspace $V\subset \ho^0(Y, \li^{N})$ of dimension $n+1$ such that the rational map
	\[
		\varphi: Y \dashrightarrow \pr(V^{\vee}),
	\]
	with $V^{\vee}$ the dual space of $V$, is dominant, and that $\varphi$ induces a closed embedding of $Z$ in $\pr(V^{\vee})$.
\end{lem}
\begin{proof}
	We may assume that $|Z|=\{ z_1,\dots, z_l \}$. It suffices to find a linear subspace $V\subset \ho^0(Y, \li^{N})$ of dimension $n+1$ such that the induced rational map
	\[
		\varphi: Y \dashrightarrow \pr(V^{\vee})
	\]
	is defined and étale on a neighbourhood of $Z$ in $Y$, and satisfies the condition that $\varphi(z_i)\not=\varphi(z_j)$ for any $z_i\not=z_j\in |Z|$. 
		
	Since $\li^{\otimes N}$ is very ample on $Y$, we can first embeds $Y$ in $\pr^K=\pr(\ho^0(Y, \li^{\otimes N})^{\vee})$. Replacing $\F_q$ by a larger finite field if needed, we can find a section $s\in \ho^0(Y, \li^{\otimes N})$ which is non-zero at any point of $|Z|$. Write $U_1:=Y-\mathrm{div}s$. the embedding of $Y$ in $\pr^K$ induces an embedding of $U_1$ in $\A^K$. The hyperplane $\pr^K-\A^K$ is defined by the section $s\in \ho^0(\pr^K, \mathcal{O}(1))=\ho^0(Y, \li^{\otimes N})$. Moreover, the scheme $Z$, being a closed subscheme of $U_1$, is also embedded in $\A^K$. To finish the proof, we only need to find a projection $\A^K\rightarrow \A^n$ which is étale when restricted to a neighbourhood of $Z$ in $U_1$ and injective when restricted to $Z$. In fact, we show that a general projection satisfies these two conditions. Here general means all projections $\A^K\rightarrow \A^n$ contained in a non-empty open subsheme of $\mathrm{Gr}(n, K)$, which is the moduli space of such projections. 
		
	For a general projection $\A^K\rightarrow \A^n$, the composition $\varphi: U_1\rightarrow \A^K\rightarrow \A^n$ is étale on a neighbourhood of $Z$ in $U_1$. To see this, we show that for a general projection $\A^K\rightarrow \A^n$, $\varphi: U_1\rightarrow \A^K\rightarrow \A^n$ is étale at any point of $Z$. For $z_i\in Z$, the exact sequence
	\[
		0 \longrightarrow C_{U_1/\A^K,z_i} \longrightarrow \Omega_{\A^K, z_i} \longrightarrow \Omega_{U_1,z_i}\longrightarrow 0
	\]
	splits, and $ \Omega_{U_1,z_i}$ is free of rank $n$ by hypothesis. Therefore for any projection $\A^K\rightarrow \A^n$, the composition $\varphi: U_1\rightarrow \A^K\rightarrow \A^n$ induces a morphism on differential sheaves
	\[
		\Omega_{\A^n, \varphi(z_i)} \longrightarrow \Omega_{\A^K, z_i}\longrightarrow \Omega_{U_1, z_i}.
	\]
	By the Jacobian criterion, a general projection $\A^K\rightarrow \A^n$ induces an isomorphism
	\[
		\Omega_{\A^n, \varphi(z_i)} \stackrel{\sim}\longrightarrow \Omega_{U_1, z_i}.
	\]
	As $Z$ is a finite scheme, a general projection $\A^K\rightarrow \A^n$ is étale at any point of $Z$. Moreover, if a $\A^K\rightarrow \A^n$ sends two different points $z_i, z_j$ of $Z$ to the same point, then it contracts the $\A^1$ containing $z_i,z_j$. Projections contracting a certain fixed line is contained in a strictly closed subscheme of the moduli space of projections $\A^K\rightarrow \A^n$. Thus a general projection do not contract this line. Since $Z$ is finite, there are only finitely many lines in $\A^K$ joining two of the points in $Z$. Therefore a general projection $\A^K\rightarrow \A^n$ sends the set of points $|Z|$ injectively to $\A^n$, hence injective when restricted to $Z$. Such a projection induces a rational map $\pr^K \dashrightarrow \pr^n$, which shows our result. 
			
\end{proof}
	
\begin{lem}\label{hi}
	Let $Y,Z,\li,N$ be as in the above lemma.
	Set $h_Z=\dim_{\F_q}\ho^0(Z, \mathcal{O}_Z)$.Then for any $d>2N$ the proportion of global sections $\sigma\in \ho^0(Y, \li^{\otimes d})$ which are sent to $0$ by the restriction morphism
	\[
		\ho^0(Y. \li^{\otimes d}) \longrightarrow \ho^0(Z, \li^{\otimes d})
	\]  
	is at most $q^{-\min(\lfloor\frac{d}{N}\rfloor, h_Z)}$.
\end{lem}
\begin{proof}
	By the previous lemma, we can find a finite extension $\F$ of $\F_q$ and a subspace $V\subset \ho^0(Y_\F, \li^{\otimes N})$ of dimension $n+1$ which induces a dominant rational map 
	\[
		\varphi: Y_{\F}\dashrightarrow \pr(V^{\vee})
	\]
	such that $\varphi|_{Z_{\F}}$ is injective. Now we use this rational map to show that 
	\[
		\dim_\F\mathrm{Im}\left( \ho^0(Y_\F, \li^{\otimes d})\longrightarrow \ho^0(Z_{\F}, \li^{\otimes d}) \right)\geq \min(\lfloor\frac{d}{N}\rfloor, h_Z ).
	\]
	As the dimension of the image is invariant under field base change, we then get
	\[
		\dim_{\F_q}\mathrm{Im}\left( \ho^0(Y, \li^{\otimes d})\longrightarrow \ho^0(Z, \li^{\otimes d}) \right)\geq \min(\lfloor\frac{d}{N}\rfloor, h_Z ).
	\]
		
	Let $(\sigma_0,\dots, \sigma_n)$ be a base of $V$ and $H_0=\mathrm{div}\sigma_0$ in $Y_{\F}$. The sections $\sigma_i$ can also be regarded as global sections of $\mathcal{O}(1)$ on $\pr(V^{\vee})$. This way we can identify $\pr(V^{\vee})-\mathrm{div}\sigma_0$ with $\A^n$ with coordinates $x_1=\frac{\sigma_1}{\sigma_0},\dots, x_n=\frac{\sigma_n}{\sigma_0}$. Then the rational map $\varphi$ can be represented by a morphism
	\[
		\varphi: Y_{\F}-H_0 \longrightarrow \A^n.
	\]
	Moreover, we can assume that $Z_{\F}$ is a closed subscheme of $Y_{\F}-H_0$. 
		
	For all $r>0$ with $(r+1)N<d$, the sheaf $\li^{\otimes d}\otimes \mathcal{O}(-rH_0)\simeq \li^{\otimes (d-rN)}$ is very ample on $Y_{\F}$. So we can find a section $\sigma_{H_0}\in \ho^0(Y_{\F}, \li^{\otimes d})$ which vanishes of order $r$ along $H_0$ but does not vanish identically on $Y_{\F}$. Let $P\in \F[x_1,\dots, x_n]$ be a polynomial of total order smaller than or equal to $r$. 
	Then the section $\varphi^*(P)\cdot \sigma_{H_0}\in \ho^0(Y_{\F}-H_0, \li^{\otimes d})$ extends to a global section on $Y_{\F}$. As $\varphi$ is dominant, linearly independent polynomials of degree smaller than or equal to $r$ induce linearly independent sections in $\ho^0(Y_{\F}-H_0, \li^{\otimes d})$, hence in $\ho^0(Y_{\F}, \li^{\otimes d})$. Thus we get a injective homomorphism
	\[
		\F[x_1,\dots, x_n]^{\leq r} \lhook\joinrel\longrightarrow \ho^0(Y_{\F}, \li^{\otimes d}).
	\]
	Moreover, we can choose an isomorphism $\ho^0(Z_{\F}, \mathcal{L}^{\otimes d})\stackrel{\sim}\longrightarrow \ho^0(Z_{\F}, \Oc_{Z_{\F}})$ so that the following diagram commutes:
	\[
		\xymatrixcolsep{0.06pc}\xymatrix{
			\F[x_1,\dots, x_n]^{\leq r} \ar@{^{(}->}[rr]\ar[d]&  & \ho^0(Y_{\F}, \mathcal{L}^{\otimes d}) \ar[d] \\
			\ho^0(\varphi(Z_{\F}), \Oc_{\varphi(Z_{\F})}) \ar[rd]^{\sim} & & \ho^0(Z_{\F}, \mathcal{L}^{\otimes d}) \ar[ld]_{\sim} \\
			& \ho^0(Z_{\F}, \Oc_{Z_{\F}}).
		}
	\]
	Then we have
		\begin{eqnarray*}
			& &\dim_{\F}\mathrm{Im}\left( \ho^0(Y_{\F}, \mathcal{L}^{\otimes d})\longrightarrow \ho^0(Z_{\F}, \mathcal{L}^{\otimes d}) \right) \\
			&\geq &\dim_{\F}\mathrm{Im}\left( \F[x_1,\dots, x_n]^{\leq r}\longrightarrow \ho^0(\varphi(Z_{\F}),\Oc_{\varphi(Z_{\F})})\right) \\
			&\geq & \min(h_Z, r+1),
		\end{eqnarray*}
		where the last inequality follows from Lemma 2.5 of \cite{Po04}. Now choose $r=\lfloor\frac{d}{N}\rfloor-1$, which is possible as $(r+1)N=\lfloor\frac{d}{N}\rfloor\cdot N\leq d$. Then the above inequality becomes
		\[
			\dim_{\F}\mathrm{Im}\left( \ho^0(Y_{\F}, \mathcal{L}^{\otimes d})\longrightarrow \ho^0(Z_{\F}, \mathcal{L}^{\otimes d}) \right)\geq \min(h_Z, \lfloor\frac{d}{N}\rfloor).
		\]
	This induces, as said above, that
	\[
		\dim_{\F_q}\mathrm{Im}\left( \ho^0(Y, \li^{\otimes d})\longrightarrow \ho^0(Z, \li^{\otimes d}) \right)\geq \min(\lfloor\frac{d}{N}\rfloor, h_Z ).
	\]
	Therefore the proportion of global sections $\sigma\in \ho^0(Y, \li^{\otimes d})$ which are sent to $0$ by the restriction morphism $\ho^0(Y. \li^{\otimes d}) \longrightarrow \ho^0(Z, \li^{\otimes d}) $ is at most $q^{-\min(\lfloor\frac{d}{N}\rfloor, h_Z)}$.
\end{proof}

\begin{prop}\label{hifinite}
	Let $\F_q$ be a finite field of characteristic $p$. Let $Y$ be a projective scheme of dimension $n$ over $\F_q$, and $X$ a smooth subscheme of $Y$ of dimension $m$. Assume that there exists a smooth open subscheme $U$ in $Y$ containing $X$. Let $\mathcal{L}$ be an ample line bundle on $Y$, and let $N$ be a sufficiently large integer. 
	Set 
	\[
		\mathcal{Q}_{d}^{\mathrm{high}}= \{ \sigma\in \ho^0(Y,\li^{\otimes d})\ ;\ \exists x\in X,\ \deg x\geq \frac{d}{(m+1)N},\ \mathrm{div}\sigma\cap X \text{ is singular at }x \}
	\]
	and $\mathcal{Q}^{\mathrm{high}}=\bigcup_{d\geq 0}\mathcal{Q}_{d}^{\mathrm{high}}$.
	There exists a constant $c>0$ only depending on $X,U$ and the choice of $N$, independent of $d$, such that 
	\[
		\frac{ \#\mathcal{Q}_d^{\mathrm{high}}}{\# \ho^0(Y, \mathcal{L}^{\otimes d})}=O(d^{m}\cdot q^{-c\frac{d}{p}}).
	\]
	Here the constant involved in the big $O$ only depends on the sheaf $\li$, and the schemes $\overline{X}=X\times_{\Spec\ \F_q}\Spec\ \overline{\F_q}$, $\overline{Y}=Y\times_{\Spec\ \F_q}\Spec\ \overline{\F_q}$, where $\overline{\F_q}$ is any algebraic closure of $\F_q$, hence is independent of $d,q$.
	In particular,
	\[
		\overline{\mu}(\mathcal{Q}^{\mathrm{high}})=\limsup_{d\rightarrow \infty} \frac{ \#\mathcal{Q}_d^{\mathrm{high}}}{\# \ho^0(Y, \mathcal{L}^{\otimes d})}=0. 
	\]
\end{prop}

We need some reduction before proving this proposition. 
	
\begin{lem}\label{cover}
	 Let $\{ U_{\alpha} \}_{\alpha\in I}$ be a finite open covering of $U$. If the proposition is true for all $X\cap U_{\alpha}$, then it is also true for $X$.
\end{lem}
\begin{proof}[Proof]
	If $\{ U_{\alpha} \}_{\alpha\in I}$ is a finite open covering of $U$, then $X=\bigcup_{\alpha\in I} (X\cap U_{\alpha})$. If we write 
	\[
		\mathcal{Q}_{d,\alpha}^{\mathrm{high}}= \{ \sigma\in \ho^0(Y,\li^{\otimes d})\ ;\ \exists x\in X\cap U_{\alpha},\ \deg x\geq \frac{d}{(m+1)N},\ \mathrm{div}\sigma\cap X\cap U_{\alpha} \text{ is singular at }x \},
	\]
	then we have
	\[
		\mathcal{Q}_{d}^{\mathrm{high}}\subset \bigcup_{\alpha\in I}\mathcal{Q}_{d,\alpha}^{\mathrm{high}}.
	\]
	Hence if the proposition is true for all $X\cap U_{\alpha}$, we can find constants $c_{U_{\alpha}}$ for each $U_{\alpha}$ such that
	\[
		 \frac{ \#\mathcal{Q}_{d,\alpha}^{\mathrm{high}}}{\# \ho^0(Y, \mathcal{L}^{\otimes d})}=O(d^{m}\cdot q^{-c_{U_{\alpha}}\frac{d}{p}}),
	\]
	Then setting $c=\min_{\alpha}\{ c_{U_{\alpha}} \}$, we get
	\[
		\frac{ \#\mathcal{Q}_d^{\mathrm{high}}}{\# \ho^0(Y, \mathcal{L}^{\otimes d})}\leq \sum_{\alpha\in I} \frac{ \#\mathcal{Q}_{d,\alpha}^{\mathrm{high}}}{\# \ho^0(Y, \mathcal{L}^{\otimes d})}=O(d^{m}\cdot q^{-c\frac{d}{p}}).
	\]
\end{proof}
Therefore we may replace $X$ by one of the $X\cap U_{\alpha}$. In particular, we have the following : 
\begin{cor}\label{trivi}
	We may assume that the smooth open subscheme $U$ containing $X$ in the statement satisfies the following condition : there exist $t_1, \dots, t_n\in \ho^0(U, \mathcal{O}_Y)$ such that $X$ is defined by $t_{m+1}=\cdots=t_n=0$, and that
	\[
		\Omega_{U/\F_q}^1\simeq \bigoplus_{i=1}^n\Oc_U \mathrm{d} t_i,\qquad \Omega_{X/\F_q}^1\simeq \bigoplus_{i=1}^m\Oc_X\mathrm{d}t_i.
	\]
\end{cor}
\begin{proof}[Proof]
	As the two condition in the corollary is satisfied locally at any point of $X$, the corollary follows from Lemma \ref{cover} as $U$ is quasi-compact.
\end{proof}

\begin{lem}\label{complem}
	For any positive integer $M$, with a choice of an integer $N_0\geq M$, we can cover $U$ by finitely many open subschemes $U'$ satisfying the condition that we can find a section $\tau\in \ho^0(Y, \li^{\otimes (N_0+1)})$ and sections $\tau_1,\dots, \tau_s\in  \ho^0(Y, \li^{\otimes N_0})$ such that
	\[
		U'=Y-\mathrm{div}(\tau)=\bigcup_{j=1}^s\left( Y-\mathrm{div}(\tau_j)\right).
	\]
\end{lem}
\begin{proof}[Proof]
	First, take an integer $N'_0>0$ satisfying Lemma \ref{ample} and that $\mathcal{I}_{Y-U}\otimes \li^{\otimes d}$ is globally generated for all $d\geq N'_0$. Here $\mathcal{I}_{Y-U}$ is the ideal sheaf of the closed subscheme $Y-U$ with the reduced induced structure. We may then choose non-zero sections 
	\[
		\tau'_1,\dots, \tau'_t\in \ho^0(Y, \mathcal{I}_{Y-U}\otimes \li^{\otimes N'_0})\subset \ho^0(Y, \li^{\otimes N'_0}),
	\]
	generating $\mathcal{I}_{Y-U}\otimes \li^{\otimes N'_0}$. This means that set theoretically, we have $Y-U=\bigcap_{i}\mathrm{div}(\tau'_i)$. In other words,
	\[
		U=\bigcup_{i}\big( Y-\mathrm{div}(\tau'_i) \big),
	\]
	where $Y-\mathrm{div}(\tau'_i)$ are open subschemes of $Y$. Without loss of generality, we may assume that $U$ itself is one of such open subschemes, i.e. there exists a section $\tau'\in \ho^0(Y, \li^{\otimes N'_0})$ such that $U=Y-\mathrm{div}(\tau')$. We denote $\mathrm{div}(\tau')$ by $D$. 

	Now set $N_0=r N'_0-1$ for some positive integer $r$ such that $N_0\geq M$ and that the sheaf $\mathcal{I}_D\otimes \li^{\otimes N_0}$ is globally generated. Then in particular we can find sections $\tau_1,\dots, \tau_s\in \ho^0(Y, \mathcal{I}_D\otimes \li^{\otimes N_0})\subset \ho^0(Y, \li^{\otimes N_0})$ such that $D=\bigcap_{j=1}^s\mathrm{div}(\tau_j)$ set theoretically. This suggests that 
	\[
		U=\bigcup_{j=1}^s\left( Y-\mathrm{div}(\tau_j)\right).
	\]
	We also set $\tau=(\tau')^r\in \ho^0(Y, \li^{\otimes (N_0+1)})$. In this situation we still have $D=\mathrm{div}(\tau)$ set theoretically. The section $\tau$ and sections $\tau_1,\dots, \tau_s$ are then what we need in the lemma.
\end{proof}

\begin{cor}
	We may assume that the smooth open subscheme $U$ containing $X$ in the statement satisfies the condition in Corollary \ref{trivi} and the condition that we can find a section $\tau\in \ho^0(Y, \li^{\otimes (N_0+1)})$ and sections $\tau_1,\dots, \tau_s\in \ho^0(Y, \li^{\otimes N_0})$ such that
	\[
		U=Y-\mathrm{div}(\tau)=\bigcup_{j=1}^s\left( Y-\mathrm{div}(\tau_j)\right).
	\]
\end{cor}
\begin{proof}[Proof]
	This is a direct consequence of Lemma \ref{cover} and Lemma \ref{complem}.
\end{proof}

For any $d>0$ and any $1\leq j\leq s$, consider the morphism
\begin{eqnarray*}
	\Phi_j: \ho^0(Y,\mathcal{L}^{\otimes d}) & \longrightarrow & \ho^0(U, \Oc_Y) \\
	\sigma\qquad & \longmapsto &\ \ \frac{\sigma \cdot \tau_j^d}{\tau^d}.
\end{eqnarray*}
For simplicity of notations, we don't distinguish morphisms $\Phi_j$ for different $d$. This will not cause any confusion as the source of $\Phi_j$ will be clear by the context. Then for any $\sigma\in \ho^0(Y, \li^{\otimes d})$, we have
\[
	\mathrm{div}\sigma \cap U=\bigcup_{1\leq j\leq s}\Big(U_j\cap \mathrm{div} \Phi_j(\sigma) \Big),
\]
where $U_j:=\left( Y-\mathrm{div}(\tau_j)\right)$. Let $\partial_i\in Der_{\F_q}(\mathcal{O}_U, \mathcal{O}_U)\simeq \Hom_{\mathcal{O}_U}(\Omega^1_{U/\F_q}, \mathcal{O}_U)$ be the dual of $\mathrm{d}t_i\in \ho^0(U, \Omega^1_{U/\F_q})$. If a global section $\sigma\in \ho^0(Y, \li^{\otimes d})$ is such that $\mathrm{div}\sigma\cap X$ is singular at a closed point $x\in X$, then for a $U_j$ containing $x$, we have
\[
	\Phi_j(\sigma)(x)=\left( \partial_1 \Phi_j(\sigma)\right)(x)=\cdots =\left(\partial_m \Phi_j(\sigma) \right)(x)=0.
\]

We want to show that there exists a positive integer $N_1$ such that for each $i$, $\left(\partial_i \Phi_j(\sigma) \right)\cdot \tau^{d+N_1}$ can be extended to a global section of $\li^{\otimes (N_0+1)(d+N_1)}$. To show this, we need to study the derivation map
\[
	\mathrm{d} : \ho^0(U, \mathcal{O}_Y)\longrightarrow \ho^0(U, \Omega_{Y/\F_q}^1).
\]
For any section $f\in \ho^0(U, \mathcal{O}_Y)$, we denote its image under the derivation by $\mathrm{d}f\in \ho^0(U, \Omega_{Y/\F_q}^1)$.
\begin{lem}\label{preextend}
	In the setting of the above corollary, there exists a positive integer $d_0$ such that when $d\geq d_0$, for any $f\in \ho^0(U, \mathcal{O}_Y)$, if $f\cdot \tau^d$ can be extended to a global section in $ \ho^0(Y, \li^{\otimes (N_0+1)d})$, then $\mathrm{d} f\cdot \tau^{d+1}$ can be extended to a global section in $\ho^0(Y, \Omega_{Y/\F_q}^1\otimes \li^{\otimes (N_0+1)(d+1)})$.
\end{lem}
\begin{proof}[Proof]
	By assumption, the sheaf $\li^{\otimes (N_0+1)}$ is very ample. So it induces a closed embedding
	\[
		Y\lhook\joinrel\longrightarrow \pr\left( \ho^0(Y, \li^{\otimes (N_0+1)}) \right).
	\]
	To simplify the notation, we denote the projective space $\pr\left( \ho^0(Y, \li^{\otimes (N_0+1)})\right)$ by $\pr^{K_0}$ with homogeneous coordinates $T_0,T_1,\dots, T_{K_0}$. In particular, let $T_0$ be the section corresponding to $\tau$ in $\ho^0(\pr^{K_0},\mathcal{O}(1))$. Therefore the closed embedding $Y\lhook\joinrel\longrightarrow \pr^{K_0}$ identifies $U$ with a closed subscheme of $\A^{K_0}=\pr^{K_0}-\mathrm{div}\ T_0$. Write $x_i=\frac{T_i}{T_0}$. Then $x_1,\dots, x_{K_0}$ form a system of coordinates of $\A^{K_0}$. Let $d_0>0$ be an integer such that for any $d\geq d_0$, the restriction morphism 
	\[
		\ho^0(\pr^{K_0},\mathcal{O}(d)) \longrightarrow \ho^0(Y, \li^{\otimes (N_0+1)d})
	\]
	is surjective. Take $f\in \ho^0(U, \mathcal{O}_Y)$ and $d\geq d_0$ such that $f\cdot \tau^d$ can be extended to a global section in $ \ho^0(Y, \li^{\otimes (N_0+1)d})$. The surjectivity of the above restriction morphism suggests that we can find $\widetilde{F} \in\ho^0(\pr^{K_0},\mathcal{O}(d))$ whose restriction to $\ho^0(Y, \li^{\otimes (N_0+1)d})$ is the chosen extension of $f\cdot \tau^d$. Set
	\[
		\widetilde{f}= \frac{\widetilde{F}}{T_0^d}\in \ho^0(\A^{K_0},\mathcal{O}_{\pr^{K_0}}).
	\]
	Then $\widetilde{f}$ has image $f$ when restricted to $U$. So $\widetilde{f}= \ho^0(\A^{K_0},\mathcal{O}_{\pr^{K_0}})$ is a section such that $\widetilde{f}\cdot T_0^d$ can be extended to a section in $\ho^0(\pr^{K_0},\mathcal{O}(d))$ and that $\widetilde{f}|_U=f$. 
	
	Now we consider the derivation
	\[
		 \ho^0(\A^{K_0},\mathcal{O}_{\pr^{K_0}}) \longrightarrow  \ho^0(\A^{K_0},\Omega_{\pr^{K_0}/\F_q}^1)
	\]
	sending $\widetilde{f}$ to 
	\[
		\mathrm{d}\widetilde{f}=\sum_{i=1}^{K_0}\left(\frac{\partial}{\partial x_i}\widetilde{f}\right) \mathrm{d}x_i.
	\]
	The fact that $\widetilde{f}\cdot T_0^d$ can be extended to a section in $\ho^0(\pr^{K_0},\mathcal{O}(d))$ means that $\widetilde{f}$ is a polynomial of total degree smaller than or equal to $d$ in $x_1,\dots x_{K_0}$. As for each $i$, $\frac{\partial}{\partial x_i}\widetilde{f}$ is of degree strictly smaller than $\widetilde{f}$ if non-zero, all $\left(\frac{\partial}{\partial x_i}\widetilde{f}\right)\cdot T_0^{d-1}$ can be extended to a global section in $\ho^0(\pr^{K_0},\mathcal{O}(d-1))$. Note that we have a short exact sequence
	\[
		0\longrightarrow \Omega_{\pr^{K_0}/\F_q}^1 \longrightarrow \mathcal{O}_{\pr^{K_0}/\F_q}(-1)^{\oplus (K_0+1)} \longrightarrow \mathcal{O}_{\pr^{K_0}/\F_q} \longrightarrow 0
	\]
	identifying $\Omega_{\pr^{K_0}/\F_q}^1$ with a locally free subsheaf of $\mathcal{O}_{\pr^{K_0}/\F_q}(-1)^{\oplus (K_0+1)}$. Under this identification, for each $1\leq i\leq K_0$ we can write 
	\[
		\mathrm{d}x_i=\mathrm{d} \left( \frac{T_i}{T_0} \right)=\frac{1}{T_0}e_i-\frac{T_i}{T_0^2}e_0, 
	\]
	where $e_0,\dots, e_{K_0}$ is a chosen basis of $ \mathcal{O}_{\pr^{K_0}/\F_q}(-1)^{\oplus (K_0+1)} $. Therefore for each $1\leq i\leq K_0$, $\mathrm{d} x_i\cdot T_0^2$ can be extended to a global section in $\ho^0(\pr^{K_0},\Omega_{\pr^{K_0}/\F_q}^1\otimes \mathcal{O}(2))$. As a consequence, $\mathrm{d}\widetilde{f}\cdot T_0^{d+1}$ can be extended to a global section in the space $\ho^0(\pr^{K_0},\Omega_{\pr^{K_0}/\F_q}^1\otimes \mathcal{O}(d+1))$.

	When restricted to $Y$, we have a natural morphism
	\[
		 \left(\Omega_{\pr^{K_0}/\F_q}^1\otimes \mathcal{O}(d+1)\right)\Big|_Y\simeq \left(\Omega_{\pr^{K_0}/\F_q}^1\right)\Big|_Y\otimes \li^{\otimes(d+1)}\stackrel{r}\longrightarrow \Omega_{Y/\F_q}^1\otimes \li^{\otimes(d+1)},
	\]
	which gives us a section $r(\mathrm{d}\widetilde{f}\cdot T_0^{d+1})\in \ho^0(Y, \Omega_{Y/\F_q}^1\otimes \li^{\otimes(N_0+1)(d+1)})$. To finish the proof, it suffices to show that 
	\[
		r\left(\mathrm{d}\widetilde{f}\cdot T_0^{d+1}\right)\Big|_U=\mathrm{d}f\cdot \tau^{d+1}.
	\]
	But as $r$ is induced by $\left(\Omega_{\pr^{K_0}/\F_q}^1\right)\big|_Y\longrightarrow \Omega_{Y/\F_q}^1$, it commutes with the multiplication by $T_0, \tau$. So the equality is clear. Therefore $r\left(\mathrm{d}\widetilde{f}\cdot T_0^{d+1}\right)$ is an extension of $\mathrm{d}f\cdot \tau^{d+1}$ in $\ho^0(Y, \Omega_{Y/\F_q}^1\otimes \li^{\otimes(N_0+1)(d+1)})$.
\end{proof}

\begin{lem}\label{beforeextend}
	There exists a positive integer $N'_1$ which only depends on the sections $t_1,\dots, t_n\in \ho^0(U, \mathcal{O}_Y)$ and $\tau\in \ho^0(Y, \li^{\otimes (N_0+1)})$ satisfying the following condition : if a section $f \in  \ho^0(U, \mathcal{O}_Y)$ is such that $\mathrm{d}f\cdot \tau^{d}$ can be extended to a global section in $\ho^0(Y, \Omega_{Y/\F_q}^1\otimes \li^{\otimes(N_0+1)d})$, then for any $1\leq i\leq m$ and any $\delta \geq N'_1$, the section $\partial_i f \cdot \tau^{d+\delta}$ can be extended to a global section of $\li^{\otimes (N_0+1)(d+\delta)}$. If this condition is satisfied by $N'_1$, it is also satisfied by any integer bigger than $N'_1$.
\end{lem}
\begin{proof}[Proof]
	Note that for each $1\leq i\leq n$, the dual $\partial_i\in Der_{\F_q}(\mathcal{O}_U, \mathcal{O}_U)\simeq \Hom_{\mathcal{O}_U}(\Omega^1_{U/\F_q}, \mathcal{O}_U)$ of $\mathrm{d}t_i\in \ho^0(U, \Omega^1_{U/\F_q})$ can be regarded as a section in $\ho^0\left( U, \mathcal{H}om_{\mathcal{O}_Y}(\Omega^1_{Y/\F_q}, \mathcal{O}_Y) \right)$. Since $\mathcal{H}om_{\mathcal{O}_Y}(\Omega^1_{Y/\F_q}, \mathcal{O}_Y)$ is coherent and that $U$ is the complement of $\mathrm{div}\tau$ in $Y$, there exists an integer $N'_{1,i}>0$ such that for any $\delta\geq N'_{1,i}$, the section $\partial_i\cdot \tau^{\delta}\in \ho^0\left( U, \mathcal{H}om_{\mathcal{O}_Y}(\Omega^1_{Y/\F_q}, \mathcal{O}_Y)\otimes \li^{(N_0+1)\delta} \right)$ can be extended to a global section in $\ho^0\left( Y, \mathcal{H}om_{\mathcal{O}_Y}(\Omega^1_{Y/\F_q}, \mathcal{O}_Y)\otimes \li^{(N_0+1)\delta} \right)$. Set 
	\[
		N'_1=\max \{ N'_{1,i}\ ;\ 1\leq i\leq n\}.
	\]
	Then when $\delta\geq N'_1$, for any $1\leq i\leq n$ the section $\partial_i\cdot \tau^{\delta}\in \ho^0\left( U, \mathcal{H}om_{\mathcal{O}_Y}(\Omega^1_{Y/\F_q}, \mathcal{O}_Y)\otimes \li^{(N_0+1)\delta} \right)$ can be extended to a global section in 
	\[
		\ho^0\left( Y, \mathcal{H}om_{\mathcal{O}_Y}(\Omega^1_{Y/\F_q}, \mathcal{O}_Y)\otimes \li^{(N_0+1)\delta} \right)\simeq \Hom(\Omega^1_{Y/\F_q}, \li^{(N_0+1)\delta}). 
	\]
	Note that 
	\[
		\Hom(\Omega^1_{Y/\F_q}, \li^{(N_0+1)\delta}) \simeq \Hom(\Omega^1_{Y/\F_q}\otimes \li^{(N_0+1)d}, \li^{(N_0+1)(d+\delta)}).
	\]
	The global section 
	\[
		(\partial_i\cdot \tau^{\delta})\big( \mathrm{d}f\cdot \tau^{d} \big)\in \ho^0(Y, \li^{(N_0+1)(d+\delta)})
	\]
	satisfies
	\[
		\left((\partial_i\cdot \tau^{\delta})\big( \mathrm{d}f\cdot \tau^{d} \big)\right)\Big|_U= \left(\partial_i f \right)\cdot \tau^{d+\delta},
	\]
	which means that $\left(\partial_i f \right)\cdot \tau^{d+\delta}$ can be extended to a global section of $\li^{(N_0+1)(d+\delta)}$. Hence we conclude.
\end{proof}
\begin{lem}\label{extend}
	There exists a positive integer $N_1$ which only depends on the sheaf $\li$, the sections $t_1,\dots, t_n\in \ho^0(U, \mathcal{O}_Y)$ and $\tau\in \ho^0(Y, \li^{\otimes (N_0+1)})$ such that for any $\sigma \in  \ho^0(Y, \li^{\otimes d})$, any $1\leq i\leq m$, $1\leq j\leq s$, the section $\left(\partial_i \Phi_j(\sigma) \right)\cdot \tau^{d+\delta}$ can be extended to a global section of $\li^{\otimes (N_0+1)(d+\delta)}$ for any $\delta \geq N_1$. If this condition is satisfied by $N_1$, it is also satisfied by any integer bigger than $N_1$.
\end{lem}
\begin{proof}[Proof]
	For any $\sigma \in  \ho^0(Y, \li^{\otimes d})$, as $\Phi_j(\sigma)=\frac{\sigma \cdot \tau_j^d}{\tau^d}$, evidently $\Phi_j(\sigma)\cdot \tau^d$ can be extended to a global section in $\ho^0(Y, \li^{\otimes (N_0+1)d})$. Let $d_0$ be the constant defined in Lemma \ref{preextend}. In particular, $\Phi_j(\sigma)\cdot \tau^{d+d_0}$ can also be extended to a global section in $\ho^0(Y, \li^{\otimes (N_0+1)(d+d_0)})$. Then by Lemma \ref{preextend} $\mathrm{d}\Phi_j(\sigma)\cdot \tau^{d+d_0+1}$ can be extended to a global section in $\ho^0(Y, \Omega_{Y/\F_q}^1\otimes \li^{\otimes (N_0+1)(d+d_0+1)})$. 
	Set $N_1=N'_1+d_0+1$. So when $\delta\geq N_1$, $\delta-d_0-1\geq N'_1$. Applying Lemma \ref{beforeextend}, we obtain that for any $\delta\geq N_1$, the section 
	\[
		\left(\partial_i \Phi_j(\sigma) \right)\cdot \tau^{d+\delta}=\left(\partial_i \Phi_j(\sigma) \right)\cdot \tau^{(d+d_0+1)+(\delta-d_0-1)}
	\]
	can be extended to a global section in $\ho^0(Y, \li^{(N_0+1)(d+\delta)})$.
\end{proof}
	
Now we assume that $p$ is the characteristic of $\F_q$ and that $(N_0+1, p)=1$. We also assume that $N_1\geq N_0$. Take $N= (N_0+1)(N_1+p-1)+p$. For any $d\geq N$, there exists an $l_d$ such that $N_1\leq l_d\leq N_1+p$ such that 
\[
	d\ \mathrm{mod}\ p  \equiv (N_0+1)l_d\ \mathrm{mod}\ p.
\]
Set $k_d=\frac{1}{p}[d-(N_0+1)l_d]$. Then for any $d\geq N$, fixing extensions of $t_i\tau^{l_d}$ to global sections for all $1\leq i\leq m$, we can construct morphisms of groups
\begin{eqnarray*}
	\ho^0(Y, \mathcal{L}^{\otimes k_d}) & \longrightarrow & \ho^0(Y, \mathcal{L}^{\otimes d}) \\
	\beta & \longmapsto & \beta^p\cdot t_i\tau^{l_d}
\end{eqnarray*}
for any $1\leq i\leq m$, and 
\begin{eqnarray*}
	\ho^0(Y, \mathcal{L}^{\otimes k_d}) & \longrightarrow & \ho^0(Y, \mathcal{L}^{\otimes d}) \\
	\gamma & \longmapsto & \gamma^p\cdot \tau^{l_d}.
\end{eqnarray*}
We construct a surjective morphism
\begin{eqnarray*}
	\ho^0(Y,\mathcal{L}^{\otimes d})\times \left( \prod_{i=1}^m \ho^0(Y,\mathcal{L}^{\otimes k_d}) \right) \times \ho^0(Y, \mathcal{L}^{\otimes k_d})  \longrightarrow  \ho^0(Y, \mathcal{L}^{\otimes d})
\end{eqnarray*}
which sends $(\sigma_0, (\beta_1,\dots, \beta_m), \gamma)$ to
\[
	\sigma=\sigma_0+\sum_{i=1}^m \beta_i^pt_i\tau^{l_d}+\gamma^p\tau^{l_d}.
\]
Naturally, this morphism commutes with all $\Phi_j$:
\[
	\Phi_j(\sigma)=\Phi_j(\sigma_0)+\sum_{i=1}^m \Phi_j(\beta_i)^pt_i\Phi_j(\tau)^{l_d}+\Phi_j(\gamma)^p\Phi_j(\tau)^{l_d}.
\]
Since
\[
	\partial_i[\Phi_j(\beta_i)^p t_i \Phi_j(\tau)^{l_d}]=\Phi_j(\beta_i)^p \Phi_j(\tau)^{l_d} + l_d \Phi_j(\beta_i)^pt_i\Phi_j(\tau)^{l_d-1}\cdot \partial_i \Phi_j(\tau),
\]
and for $i'\not=i$, 
\[
	\partial_i[\Phi_j(\beta_{i'})^p t_{i'} \Phi_j(\tau)^{l_d}]=l_d \Phi_j(\beta_{i'})^p t_{i'}\Phi_j(\tau)^{l_d-1}\cdot \partial_i \Phi_j(\tau),
\]
the differential of $\Phi_j(\sigma)$ can be written as
\begin{eqnarray*}
	\partial_i\Phi_j(\sigma) &=& \left[ \sum_{i'=1}^m l_d\Phi_j(\beta_{i'})^p t_{i'}\Phi_j(\tau)^{l_d-1}+l_d \Phi_j(\gamma)^p\Phi_j(\tau)^{l_d-1} \right]\cdot \partial_i \Phi_j(\tau)\\
	& & \qquad \qquad\qquad\qquad\qquad\qquad+\partial_i\Phi_j(\sigma_0)+\Phi_j(\beta_i)^p\Phi_j(\tau)^{l_d} \\\\
	&=& \partial_i\Phi_j(\sigma_0)+ \frac{l_d(\Phi_j(\sigma)-\Phi_j(\sigma_0))}{\Phi_j(\tau)}\partial_i\Phi_j(\tau)+ \Phi_j(\beta_i)^p \Phi_j(\tau)^{l_d}.
\end{eqnarray*}
To prove Proposition \ref{hifinite}, for a section $\sigma\in \ho^0(Y, \mathcal{L}^{\otimes d})$, we need to study the singular locus of $\mathrm{div}\sigma\cap X$. Since
\[
	\mathrm{div}\sigma \cap U=\bigcup_{1\leq j\leq s}\Big(U_j\cap \mathrm{div} \Phi_j(\sigma) \Big),
\]
and $X\subset U$,	we have
\[
	\mathrm{Sing}(\mathrm{div}\sigma\cap X)\subset \bigcup_{1\leq j\leq s}\Big( U_j\cap \mathrm{Sing}(\mathrm{div}(\Phi_j(\sigma))\cap X) \Big).
\]
Note that for a $\sigma\in \ho^0(Y, \mathcal{O}_Y)$, $\mathrm{div}(\Phi_j(\sigma))\cap X$ is singular at a point $x\in X$ if and only if
\[
	\Phi_j(\sigma)(x)=\partial_1\Phi_j(\sigma)(x)=\cdots=\partial_m\Phi_j(\sigma)(x)=0,
\]
by conditions on $\partial_i$, $1\leq i\leq n$. Therefore we have
\[
	\mathrm{Sing}(\mathrm{div}(\Phi_j(\sigma))\cap X)=\mathrm{div}(\Phi_j(\sigma))\cap \mathrm{div}(\partial_1\Phi_j(\sigma))\cap\cdots\cap \mathrm{div}(\partial_m\Phi_j(\sigma))
\]
in $U$.
	
Now for any $(\sigma_0, (\beta_1,\dots, \beta_m), \gamma) \in \ho^0(Y,\mathcal{L}^{\otimes d})\times \left( \prod_{i=1}^m \ho^0(Y,\mathcal{L}^{\otimes k_d}) \right) \times \ho^0(Y, \mathcal{L}^{\otimes k_d})$, set 
\[
	g_{j,i}(\sigma_0,\beta_i)=\partial_i \Phi_{j}(\sigma_0)-\frac{l_d\Phi_{j}(\sigma_0)}{\Phi_{j}(\tau)} \partial_i\Phi_{j}(\tau)+\Phi_{j}(\beta_i)^p\Phi_{j}(\tau)^{l_d},
\]
and
\[
	W_{j,i}:=X\cap U_j\cap \{ g_{j,1}=\cdots=g_{j,i}=0 \}.
\]
Then for any $\sigma=\sigma_0+\sum_{i=1}^{n-1} \beta_i^pt_i\tau_{0,p}^{l_d}+\gamma^p\tau_{0,p}^{l_d}$, comparing the expressions of $g_{j,i}$ and $\partial_i\Phi_{j}$, which give
\[
	g_{j,i}(\sigma_0,\beta_i)=\partial_i\Phi_{j}(\sigma)-\frac{l_d\Phi_j(\sigma)}{\Phi_j(\tau)}\partial_i\Phi_j(\tau),
\]
we have
\[
	g_{j,i}(\sigma_0,\beta_i)|_{ \mathrm{div\ } \Phi_{j}(\sigma)}=\partial_i\Phi_{j}(\sigma)|_{ \mathrm{div\ } \Phi_{j}(\sigma)}.
\]
Applying Lemma \ref{extend} to the section $\sigma\in \ho^0(Y, \mathcal{L}^{\otimes d})$, we know that $\left(\partial_i \Phi_j(\sigma) \right)\cdot \tau^{d+\delta}$ can be extended to a global section of $\li^{\otimes (N_0+1)(d+\delta)}$ for any $\delta \geq N_1$. By the same lemma, the section $\partial_i\Phi_j(\tau) \cdot \tau^{(N_0+1)+\delta}$ extends to a global section of $\li^{\otimes (N_0+1)((N_0+1	)+\delta)}$ for any $\delta \geq N_1$. Note that on $U$ we have
\begin{eqnarray*}
	\frac{l_d\Phi_j(\sigma)}{\Phi_j(\tau)}&=& l_d\frac{\sigma \cdot \tau_j^d}{\tau^d} \cdot\left(\frac{\tau \cdot \tau_j^{N_0+1}}{\tau^{N_0+1}}\right)^{-1}\\
	&=&l_d\frac{\sigma \cdot \tau_j^{d-N_0-1}}{\tau^{d-N_0}}.
\end{eqnarray*}
So the section
\begin{eqnarray*}
	\frac{l_d\Phi_j(\sigma)}{\Phi_j(\tau)}\partial_i\Phi_j(\tau)\cdot \tau^{d+\delta} &=& \left( \frac{l_d\Phi_j(\sigma)}{\Phi_j(\tau)}\cdot \tau^{d-N_0}\right)\cdot \left( \partial_i\Phi_j(\tau) \cdot \tau^{(N_0+1)+(\delta-1)} \right)
\end{eqnarray*}
extends to a global section of $\li^{\otimes (N_0+1)(d+\delta)}$ for any $\delta \geq N_1+1$. 
Therefore the section 
\[
	g_{j,i}(\sigma_0,\beta_i)\cdot \tau^{d+\delta}=\left(\partial_i\Phi_{j}(\sigma)-\frac{l_d\Phi_j(\sigma)}{\Phi_j(\tau)}\partial_i\Phi_j(\tau)\right)\cdot \tau^{d+\delta}\in \ho^0(U, \li^{\otimes (N_0+1)(d+\delta)})
\]
can be extended to a global section in $\ho^0(Y, \li^{\otimes (N_0+1)(d+\delta)})$ for any $\delta\geq N_1+1$.
	
\begin{lem}\label{reddim}
	For $0\leq i\leq m-1$, with a fixed choice of $\sigma_0,\beta_1,\dots, \beta_i$ such that $\dim W_{j,i}\leq m-i$, the proportion of $\beta_{i+1}$ in $\ho^0(Y,\mathcal{L}^{\otimes k_d})$ such that $\dim W_{j,i+1} \leq m-i-1$ is  $1-O(d^i\cdot q^{2-\frac{d}{(N_0+1)N_1p}})$, where $p$ is the characteristic of $\F_q$ and the constant involved is independent of $d,q$.
\end{lem} 
\begin{proof}[Proof]
	Let $V_1,\dots V_s$ be the $(m-i)$-dimensional $\F_q$-irreducible components of the reduced scheme $(W_{j.i})_{\mathrm{red}}$. The closure of the $V_i$'s in $Y$ are contained in the set of $(m-i)$-dimensional $\F_q$-irreducible components of $\overline{X}\cap \mathrm{div}\ g_{j,1}\tau^{d+N_1+1}\cap \cdots \cap \mathrm{div}\ g_{j,i}\tau^{d+N_1+1}$. Since the sections $g_{j,l}\tau^{d+N_1+1}$ are global sections of $\ho^0(Y, \li^{\otimes (N_0+1)(d+N_1+1)})$, and that $\li^{\otimes (N_0+1)}$ induces a closed embedding of $Y$ into $\pr(\ho^0(Y, \li^{\otimes (N_0+1)})^{\vee})$, the sections $g_{j,l}\tau^{d+N_1+1}$ can be extended uniquely to sections of $\ho^0(\pr(\ho^0(Y, \li^{\otimes (N_0+1)})^{\vee}), \mathcal{O}(d+N_1+1))$.
	Applying refined Bézout's theorem (see \cite[Theorem 12.3]{Fu84} for a precise statement), we get
	\[
		s\leq (\deg \overline{X})(\deg \ g_{j,1}\tau^{d+N_1+1})\cdots (\deg \ g_{j,i}\tau^{d+N_1+1})=(\deg \overline{X})(d+N_1+1)^i=O(d^i),
	\]
	where coefficients involved in $O(d^i)$ only depends on $\deg \overline{X}$. 
	Since for $1\leq e\leq s$ we have $\dim V_e\geq 1$, so for each $V_e$ there exists a $t_i$ such that $t_i|_{V_e}$ is not constant. We want to bound
	\[
		G_{e,j}^{\mathrm{bad}}:=\{ \beta_{i+1}\in \ho^0(Y, \li^{\otimes k_d}) \ ;\ g_{j, i+1}(\sigma_0, \beta_{i+1}) \text{ is identically }0 \text{ on }V_e \}. 
	\]
	Note that if $\beta_{i+1},\beta'_{i+1}\in G_{e,j}^{\mathrm{bad}}$, then $\beta_{i+1}-\beta'_{i+1}$ is identically $0$ on $V_e$. In fact, as on $V_e$
	\begin{eqnarray*}
		& & g_{j,i+1}(\sigma_0, \beta_{i+1})-g_{j,i+1}(\sigma_0, \beta'_{i+1}) \\
		&=& \Phi_j(\beta_{i+1})^p\Phi_j(\tau)^{l_d}-\Phi_j(\beta'_{i+1})^p\Phi_j(\tau)^{l_d} \\
		&=& \Phi_j(\beta_{i+1}-\beta'_{i+1})^p\Phi_j(\tau)^{l_d},
	\end{eqnarray*}
	and $\Phi_j(\tau)$ is everywhere non-zero, we have that if $G_{e,j}^{\mathrm{bad}}\not=\emptyset$, then it is a coset of the subspace of sections of $\ho^0(Y, \li^{\otimes k_d})$ which vanishes on $V_e$. When $d$ is large, we can decompose $k_d$ by $k_d=k_{d,1}(N_0+1)+k_{d,2}N_0$ with $k_{d,1},k_{d,2}\geq 0$ and $k_{d,2}$ minimal among all the decompositions. Note that if $k_{d,2}\geq N_0+1$, we can replace $k_{d,2}$ by $k_{d,2}-(N_0+1)$ and $k_{d,1}$ by $k_{d,1}+N_0$, which gives another decomposition of $k_d$. So when $k_{d,2}$ is minimal, we have $k_{d,2}\leq N_0$ and therefore 
	\[
		k_{d,1}\geq \frac{k_d-N_0^2}{N_0+1}.
	\]
	Then the sections
	\[
		\tau^{k_{d,1}}\tau_j^{k_{d,2}}, t_i\tau^{k_{d,1}}\tau_j^{k_{d,2}},\dots, t_i^{\lfloor \frac{k_{d,1}}{N_1} \rfloor}\tau^{k_{d,1}}\tau_j^{k_{d,2}}\in \ho^0(Y, \li^{\otimes k_d})
	\]
	restricting to $V_e$ are linearly independent. So the codimension of the subspace of sections in $\ho^0(Y, \li^{\otimes k_d})$ vanishing on $V_e$ is bigger than or equal to $\lfloor \frac{k_{d,1}}{N_1} \rfloor+1$. As $\lfloor \frac{k_{d,1}}{N_1} \rfloor+1\geq \frac{k_{d,1}}{N_1} $, this implies that the probability that $g_{j, i+1}$ vanishes on one of the $V_e$'s is at most 
	\[
		s\cdot q^{-\lfloor \frac{k_{d,1}}{N_1} \rfloor-1}\leq s\cdot q^{-\frac{k_{d,1}}{N_1}}.
	\]
	As $k_{d,1}\geq \frac{k_d-N_0^2}{N_0+1}$, $N_1\geq N_0$ and 
	\[
		k_d=\frac{1}{p}[d-(N_0+1)l_d]>\frac{1}{p}[d-(N_0+1)(N_1+p)],
	\]
	we get
	\begin{eqnarray*}
		s\cdot q^{- \frac{k_{d,1}}{N_1}}&\leq&s\cdot q^{-\frac{k_d-N_0^2}{(N_0+1)N_1}}\\
		&\leq&s\cdot q^{\frac{N_0^2}{(N_0+1)N_1}-\frac{d-(N_0+1)(N_1+p)}{ (N_0+1)N_1p}}\\
		&\leq&s\cdot q^{\frac{N_0}{N_1}+\frac{N_1+p}{N_1p}-\frac{d}{ (N_0+1)N_1p}}\\
		&\leq& s\cdot q^{2-\frac{d}{ (N_0+1)N_1p}} \\
		&=&O(d^i q^{2-\frac{d}{ (N_0+1)N_1p}}),
	\end{eqnarray*}
	where the constant involved in only depends on the degree of $\overline{X}$ as a closed subscheme of $\pr(\ho^0(Y, \li^{\otimes (N_0+1)})^{\vee})$, hence is independent of $d$ and $q$. Since $\dim W_{j,i+1} \leq m-i-1$ if and only if $g_{j,i+1}$ does not vanishing on any $V_e$, we get that the proportion of $\beta_{i+1}$ in $\ho^0(Y,\mathcal{L}^{\otimes k_d})$ such that $\dim W_{j,i+1} \leq m-i-1$ is $1-O(d^i\cdot q^{2-\frac{d}{(N_0+1)N_1p}})$.
\end{proof}
		
\begin{lem}\label{finiteset}
	With a fixed choice of $\sigma_0, \beta_1,\dots, \beta_m$ such that $W_{j,m}$ is finite, we have for $d$ sufficiently large, the proportion of $\gamma $ in $\ho^0(Y,\mathcal{L}^{\otimes k_d})$ such that 
	\[
		\mathrm{div}\sigma\cap W_{j,m}\cap \left\{ x\in |X|\ ;\ \deg x\geq\frac{d}{(m+1)N} \right\}=\emptyset
	\]
	is
	\[
		1-O(d^m q^{-\frac{d}{(m+1)N}}),
	\]
	where $\sigma=\sigma_0+\sum_{i=1}^{m} \beta_i^pt_i\tau^{l_d}+\gamma^p\tau^{l_d}$ and the constant involved only depends on the degree of $\overline{X}$ as a closed subscheme of $\pr(\ho^0(Y, \li^{\otimes (N_0+1)})^{\vee})$, hence is independent of $d,q$.
\end{lem}
\begin{proof}[Proof]
	Applying once more Bézout's theorem, we obtain that 
	\[
		\#W_{j,m}=O(d^m)
	\]
	with constant involved independent of $q$. For any $x\in W_{j,m}$, the set $H^{\mathrm{bad}}$ of sections $\gamma\in \ho^0(Y, \li^{\otimes k_d})$ such that $x$ is contained in $\mathrm{div}\sigma$ with $\sigma=\sigma_0+\sum_{i=1}^{m} \beta_i^pt_i\tau^{l_d}+\gamma^p\tau^{l_d}$ is a coset of 
	\[
		\mathrm{Ker}\left( \mathrm{ev}_x\circ \Phi_j: \ho^0(Y, \mathcal{L}^{\otimes k_d})\rightarrow \kappa(x) \right),
	\]
	where $\kappa(x)$ is the residual field of $x$ and $\mathrm{ev}_x$ is the evaluation at $x$. If moreover $\deg x\geq\frac{d}{(m+1)N}$, Lemma \ref{hi} tells us that 
	\[
		\frac{\#H^{\mathrm{bad}}}{\#\ho^0(Y, \li^{\otimes k_d})} \leq q^{-\min( \lfloor\frac{k_d}{N_0}\rfloor, \frac{d}{(m+1)N})}. 
	\]
	Thus when $d$ tends to infinity, the proportion of sections $\gamma \in \ho^0(Y, \li^{\otimes k_d})$ such that for $\sigma=\sigma_0+\sum_{i=1}^{m} \beta_i^pt_i\tau^{l_d}+\gamma^p\tau^{l_d}$,
	\[
		\mathrm{div}\sigma\cap W_{j,m}\cap \Big\{ x\in |X|\ ;\ \deg x\geq\frac{d}{(m+1)N} \Big\}\not=\emptyset,
	\]
	is bounded above by
	\begin{eqnarray*}
		\#W_{j,m}\cdot q^{-\min( \lfloor\frac{k_d}{N_0}\rfloor, \frac{d}{(m+1)N})}= O(d^m q^{-\min( \lfloor\frac{k_d}{N_0}\rfloor, \frac{d}{(m+1)N})}),
	\end{eqnarray*}
	where the constant involved is independent of $d,q$. Since $k_d=\frac{1}{p}[d-(N_0+1)l_d]$ with $N_1\leq l_d< N_1+p$, we have $k_d\geq \frac{1}{p}[d-(N_0+1)(N_1+p)]$. Then 
	\[
		\lfloor\frac{k_d}{N_0}\rfloor\geq \frac{d-(N_0+1)(N_1+p)}{N_0 p}-1\geq \frac{d-(N_0+1)(N_1+2p)}{N_0 p}\geq \frac{d}{2N_0p}
	\]
	for large $d$. However, since $N= (N_0+1)(N_1+p-1)+p$, we have 
	\begin{eqnarray*}
		\frac{d}{(m+1)N}= \frac{d}{(m+1)\big[(N_0+1)(N_1+p-1)+p\big]}\leq \frac{d}{2(N_0+1)(N_1+p-1)}.
	\end{eqnarray*}
	When $d$ is large, clearly 
	\[
		 \frac{d}{2(N_0+1)(N_1+p-1)}\leq \frac{d}{2N_0p}.
	\]
	Therefore $\lfloor\frac{k_d}{N_0}\rfloor\geq \frac{d}{(m+1)N}$, and the proportion of sections $\gamma \in \ho^0(Y, \li^{\otimes k_d})$ such that for $\sigma=\sigma_0+\sum_{i=1}^{m} \beta_i^pt_i\tau^{l_d}+\gamma^p\tau^{l_d}$,
	\[
		\mathrm{div}\sigma\cap W_{j,m}\cap \{ x\in |X|\ ;\ \deg x\geq\frac{d}{(m+1)N} \}=\emptyset
	\]
	is
	\[
		1-O(d^m q^{-\frac{d}{(m+1)N}}),
	\]
	where the constant involved is independent of $d,q$.
\end{proof}

\begin{proof}[Proof of Proposition \ref{hifinite}]
	Choose 
	\[
		\Big(\sigma_0, (\beta_1,\dots, \beta_m), \gamma\Big)\in \ho^0(Y,\mathcal{L}^{\otimes d})\times \left( \prod_{i=1}^m \ho^0(Y,\mathcal{L}^{\otimes k_d}) \right) \times \ho^0(Y, \mathcal{L}^{\otimes k_d})
	\]
	uniformly at random. Lemma \ref{reddim} and \ref{finiteset} show that as $d\rightarrow \infty$, writing 
	\[
		\sigma=\sigma_0+\sum_{i=1}^{m} \beta_i^pt_i\tau^{l_d}+\gamma^p\tau^{l_d},
	\]
	the proportion of $(\sigma_0, (\beta_1,\dots, \beta_m), \gamma)$ such that 
	\[
		\dim W_{j,i}=m-i,\ \ 0\leq i\leq m
	\]
	and
	\[
		\mathrm{div}\sigma\cap W_{j,m}\cap \left\{ x\in |X|\ ;\ \deg x\geq\frac{d}{(m+1)N} \right\}=\emptyset,
	\]
	is
	\begin{eqnarray*}
		& &\left[ \prod_{i=0}^{m-1}\left(1-O(d^i\cdot q^{2-\frac{d}{(N_0+1)N_1p}})\right) \right]\cdot \left(1-O(d^m q^{-\frac{d}{(m+1)N}}) \right) \\
		&=&\left(1-O(d^{m-1}\cdot q^{2-\frac{d}{(N_0+1)N_1p}})\right)\cdot \left(1-O(d^m q^{-\frac{d}{(m+1)N}}) \right).
	\end{eqnarray*}
	Since for $d$ sufficiently large, 
	\begin{eqnarray*}
		\frac{d}{(m+1)N}&=& \frac{d}{(m+1)\big[(N_0+1)(N_1+p-1)+p\big]}\\
		&\geq& \frac{d}{(m+1)\big[(N_0+2)(N_1+p-1)\big]}\\
		&\geq&\frac{d}{(m+1)(N_0+2)N_1 p}
	\end{eqnarray*}
	and
	\[
		\frac{d}{(N_0+1)N_1p}-2\geq \frac{d}{2(N_0+2)N_1 p}\geq \frac{d}{(m+1)(N_0+2)N_1 p}
	\]
	the probability above can be written as 
	\begin{eqnarray*}
		& &\left(1-O\Big(d^{m-1}\cdot q^{-\frac{d}{(m+1)(N_0+2)N_1 p}}\Big)\right)\cdot \left(1-O\Big(d^{m}\cdot q^{-\frac{d}{(m+1)(N_0+2)N_1 p}}\Big)\right)\\
		&=&1-O\Big(d^{m}\cdot q^{-\frac{d}{(m+1)(N_0+2)N_1 p}}\Big).
	\end{eqnarray*}
	On the other hand, for $\sigma=\sigma_0+\sum_{i=1}^{m} \beta_i^pt_i\tau^{l_d}+\gamma^p\tau^{l_d}$, as 
	\[
		g_{j,i}(\sigma_0,\beta_i)|_{ \mathrm{div\ } \Phi_{j}(\sigma)}=\partial_i\Phi_{j}(\sigma)|_{ \mathrm{div\ } \Phi_{j}(\sigma)},
	\]
	we have
	\begin{eqnarray*}
		& &\mathrm{Sing}(\mathrm{div}(\Phi_j(\sigma))\cap X)\cap U_j \\
		&=& \mathrm{div}(\Phi_j(\sigma))\cap U_j\cap \{ \partial_1\Phi_j(\sigma)=\cdots=\partial_m\Phi_j(\sigma)=0 \}\\
		&=&\mathrm{div}(\Phi_j(\sigma))\cap U_j\cap \{ g_{j,1}(\sigma_0,\beta_1)=\cdots=g_{j,m}(\sigma_0,\beta_m)=0\} \\
		&=& \mathrm{div}(\Phi_j(\sigma))\cap W_{j,m}.
	\end{eqnarray*}
	Since $\sigma=\sigma_0+\sum_{i=1}^{m} \beta_i^pt_i\tau^{l_d}+\gamma^p\tau^{l_d}$ defines a surjective homomorphism of groups
	\begin{eqnarray*}
		\ho^0(Y,\mathcal{L}^{\otimes d})\times \left( \prod_{i=1}^m \ho^0(Y,\mathcal{L}^{\otimes k_d}) \right) \times \ho^0(Y, \mathcal{L}^{\otimes k_d})  \longrightarrow  \ho^0(Y, \mathcal{L}^{\otimes d}),
	\end{eqnarray*}
	we obtain that when $d\rightarrow \infty$, the proportion of $\sigma\in \ho^0(Y, \mathcal{L}^{\otimes d})$ such that
	\[
		\mathrm{Sing}\Big(\mathrm{div}\big(\Phi_j(\sigma)\big)\cap X\Big) \cap U_j\cap \left\{ x\in |X|\ ;\ \deg x\geq\frac{d}{(m+1)N} \right\}=\emptyset
	\]
	is
	\[
		1-O\Big(d^{m}\cdot q^{-\frac{d}{(m+1)(N_0+2)N_1 p}}\Big).
	\]
	Since 
	\[
		\mathrm{Sing}(\mathrm{div}\sigma\cap X)\subset \bigcup_{j} \Big( U_j\cap \mathrm{Sing}(\mathrm{div}(\Phi_j(\sigma))\cap X)\Big),
	\]
	setting $c=\frac{1}{(m+1)(N_0+2)N_1}$, we have when $d\rightarrow \infty$, the proportion of $\sigma\in \ho^0(Y, \mathcal{L}^{\otimes d})$ such that
	\[
		\mathrm{Sing}\Big(\mathrm{div}(\sigma)\cap X\Big) \cap \left\{ x\in |X|\ ;\ \deg x\geq\frac{d}{(m+1)N} \right\}=\emptyset
	\]
	is
	\[
		1-O\Big(d^{m}\cdot q^{-c\frac{d}{p}}\Big),
	\]
	which finishes the proof of Proposition \ref{hifinite}.
\end{proof}

\begin{cor}\label{hidim}
	In the same setting as in Proposition \ref{hifinite}, there exists a constant $c>0$ independent of $d,q$ such that 
	\[
		\frac{ \#\{ \sigma \in \ho^0(Y, \mathcal{L}^{\otimes d})\ ;\ \dim\left(\mathrm{Sing}\Big( \mathrm{div}\sigma\cap X \Big)\right)>0 \}}{\# \ho^0(Y, \mathcal{L}^{\otimes d})}=O(d^{m}\cdot q^{-c\frac{d}{p}}),
	\]
	the constant involved is independent of $d,q$.
\end{cor}
\begin{proof}
	This follows directly from Proposition \ref{hifinite} once we notice that
	\[
		\left\{ \sigma \in \ho^0(Y, \mathcal{L}^{\otimes d})\ ;\ \dim\big(\mathrm{Sing}\big( \mathrm{div}\sigma\cap X \big)\big)>0 \right\}\subset  \mathcal{Q}_{d}^{\mathrm{high}}.
	\]
\end{proof}
	
\subsection{Proof of Bertini smoothness theorem over finite fields}

\begin{proof}[Proof of Theorem \ref{bertinifini}]
	The zeta function $\zeta_X(s)$ is convergent for $s>\dim X$. So in particular $\zeta_X(m+1)^{-1}=\prod_{x\in |X|}(1-q^{-(m+1)\deg x})$ is convergent. By Proposition \ref{smallf},
	\[
		\lim_{r\rightarrow \infty}\mu(\mathcal{P}_{\leq r})=\zeta_X(m+1)^{-1}.
	\]
	On the other hand, by construction of $\mathcal{P}_{d,\leq r},\mathcal{Q}_{d, >r}^{\mathrm{med}},\mathcal{Q}_{d}^{\mathrm{high}}$, we have
	\begin{eqnarray*}
		\mathcal{P}_{d} \subset \mathcal{P}_{d,\leq r} \subset \mathcal{P}_{d}\cup \mathcal{Q}_{d, >r}^{\mathrm{med}}\cup \mathcal{Q}_{d}^{\mathrm{high}}.
	\end{eqnarray*}
	Hence
	\[
		\left|  \frac{ \#\mathcal{P}_{d}}{\# \ho^0(Y, \mathcal{L}^\otimes{d})} -\frac{\#\mathcal{P}_{d,\leq r}}{\#\ho^0(Y, \li^{\otimes d})} \right|\leq  \frac{ \#\mathcal{Q}_{d, >r}^{\mathrm{med}}}{\# \ho^0(Y, \mathcal{L}^\otimes{d})}+\frac{ \#\mathcal{Q}_{d}^{\mathrm{high}}}{\# \ho^0(Y, \mathcal{L}^\otimes{d})}.
	\]
	When $d\rightarrow \infty$, by Proposition \ref{mediumf} and \ref{hifinite} we have
	\[
		\frac{ \#\mathcal{Q}_{d, >r}^{\mathrm{med}}}{\# \ho^0(Y, \mathcal{L}^\otimes{d})}+\frac{ \#\mathcal{Q}_{d}^{\mathrm{high}}}{\# \ho^0(Y, \mathcal{L}^\otimes{d})}
		=O(q^{-r})+O(d^{m}\cdot q^{-c\frac{d}{p}}).
	\]
	Hence $\overline{\mu}(\mathcal{P})$ and $\underline{\mu}(\mathcal{P})$ differ from $\mu(\mathcal{P}_{\leq r})$ by at most $\overline{\mu}(\mathcal{Q}_{>r}^{\mathrm{med}})+\overline{\mu}(\mathcal{Q}_{d}^{\mathrm{high}})=O(q^{-r})$. So letting $r$ tend to $\infty$, we get
	\[
		\mu(\mathcal{P})=\lim_{r\rightarrow \infty}\mu(\mathcal{P}_{\leq r})=\zeta_X(m+1)^{-1}.
	\]
\end{proof}


\begin{thebibliography}{abcde}
\bibitem[AB95]{AB95} Ahmed Abbes and Thierry Bouche. Théorème de Hilbert-Samuel ``arithmétique", \emph{Ann. Inst. Fourier (Grenoble)}, 45(2):375-401, 1995.
\bibitem[Au01]{Au01} Pascal Autissier. Points entiers et théorèmes de Bertini arithmétiques, \emph{Ann. Inst. Fourier (Grenoble)}, 51(6):1507-1523, 2001.
\bibitem[Au02]{Au02} Pascal Autissier. Corrigendum: ``Integer points and arithmetical Bertini theorems" (French), \emph{Ann. Inst. Fourier (Grenoble)}, 52(1):303-304, 2002.
\bibitem[BSW16]{BSW16} Manjul Bhargava, Arul Shankar and Xiaoheng Wang. Squarefree values of polynomial discriminants I, \emph{arXiv preprint}, arXiv: 1611.09806v2, 2016.	
\bibitem[BV19]{BV19} Emmanuel Breuillard and Péter P. Varjú. Irreducibility of random polynomials of large degree, \emph{Acta Math.} 223(2):195-249, 2019.
\bibitem[Ch17]{Ch17} François Charles. Arithmetic ampleness and an arithmetic Bertini theorem, \emph{arXiv preprint}, arXiv:1703.02481, 2017.
\bibitem[CP16]{CP16} François Charles and Bjorn Poonen. Bertini irreducibility theorems over finite fields. \emph{J. Amer. Math. Soc.}, 29(1):81-94, 2016.
\bibitem[Fu84]{Fu84} William Fulton. \emph{Introduction to intersection theory in algebraic geometry}, Published for the Conference Board of the Mathematical Sciences, Washington, DC, 1984.
\bibitem[GS91]{GS91} Henri Gillet and Christophe Soulé. On the number of lattice points in convex symmetric bodies and their duals. \emph{Israel J. Math.} 74(2-3):347-357, 1991.
\bibitem[GS92]{GS92} Henri Gillet and Christophe Soulé. An arithmetic Riemann-Roch theorem. \emph{Invent. Math.} 110:473-543, 1992.
\bibitem[Gr98]{Gr98} Andrew Granville. \emph{ABC} allows us to count squarefrees, \emph{Internat. Math. Res. Notices}, 19:991-1009, 1998.
\bibitem[Jo83]{Jo83} Jean-Pierre Jouanolou. \emph{Théorèmes de Bertini et applications}, volume 42 of \emph{Progress in Mathematics}, Birkhäuser Boston Inc., Boston, MA, 1983 (French).
\bibitem[La04]{La04} Robert Lazarsfeld. \emph{Positivity in algebraic geometry. I}, volume 48 of \emph{Ergebnisse der Mathematik und ihrer Grenzgebiete. 3. Folge. A Series of Modern Surveys in Mathematics [Results in Mathematics and Related Areas. 3rd Series. A Series of Modern Surveys in Mathematics]}. Springer-Verlag, Berlin, 2004. Classical setting: line bundles and linear series.
\bibitem[LW54]{LW54} Serge Lang and André Weil. Number of points of varieties in finite fields, \emph{Amer. J. Math.} 76:819-827, 1954.
\bibitem[Po03]{Po03} Bjorn Poonen. Squarefree values of multivariable polynomials, \emph{Duke Math. J.} 118:353-373, 2003.
\bibitem[Po04]{Po04} Bjorn Poonen. Bertini theorems over finite fields. \emph{Ann. of Math. (2)}, 160(3):1099-1127, 2004.
\bibitem[Se12]{Se12} Jean-Pierre Serre. \emph{Lectures on $N_X(p)$}, volume 11 of \emph{Chapman \& Hall/CRC Research Notes in Mathematics.} CRC Press, Boca Raton, FL, 2012.
\bibitem[Yu08]{Yu08} Xinyi Yuan. Big line bundles over arithmetic varieties. \emph{Invent. Math.}, 173(3):603-649, 2008.
\bibitem[Zh92]{Zh92} Shouwu Zhang. Positive line bundles on arithmetic surfaces. \emph{Ann. of Math. (2)}, 136(3):569-587, 1992.
\bibitem[Zh95]{Zh95}Shouwu Zhang. Positive line bundles on arithmetic varieties. \emph{J. Amer. Math. Soc.}, 8(1):187-221, 1995.
\end{thebibliography}
\end{document}